%% file: main.tex
\documentclass{article}
\usepackage[utf8]{inputenc}
\usepackage{amsfonts,amsmath}
\usepackage{amsthm}
\usepackage[bibstyle=authortitle,backend=biber, style=numeric]{biblatex}
\usepackage{hyperref}
\usepackage{geometry}
\usepackage{comment}
\usepackage{color}
\usepackage{esint}
\usepackage{graphicx}
\usepackage{xcolor}
\usepackage{tikz}
\usetikzlibrary{mindmap}
\graphicspath{ {./Brenier/} }
\theoremstyle{definition} \newtheorem{definition}{Definition}[section]
\theoremstyle{definition} \newtheorem{remark}[definition]{Remark}
\theoremstyle{plain} \newtheorem{lemma}[definition]{Lemma}
\theoremstyle{plain} \newtheorem{proposition}[definition]{Proposition}
\theoremstyle{plain} \newtheorem{theorem}[definition]{Theorem}
\theoremstyle{plain} \newtheorem{corollary}[definition]{Corollary}
\theoremstyle{definition} 
\theoremstyle{definition} \newtheorem{assumption}[definition]{Assumption}

\DeclareMathOperator{\BV}{BV}

\DeclareMathOperator{\diag}{diag}
\DeclareMathOperator{\dist}{dist}
\DeclareMathOperator{\supp}{supp}

\DeclareMathOperator{\Lip}{Lip}

\DeclareMathOperator{\Graph}{Graph}

\newcommand{\R}{\mathbb{R}}
\newcommand{\Q}{\mathbb{Q}}
\newcommand{\N}{\mathbb{N}}
\newcommand{\Z}{\mathbb{Z}}
\newcommand{\C}{\mathbb{C}}
\newcommand{\T}{\mathbb{T}}
\newcommand{\TV}{\text{\rm Tot.Var.}}

\newcommand*{\comme}[1]{\textcolor{cyan}{#1}}

\newcommand{\Id}{\mathrm{id}}
\newcommand{\ind}{1\!\!\mathrm{I}}

\newcommand{\M}{\mathcal{M}}

\renewcommand{\L}{\mathscr L}

\newcommand{\jump}{\mathrm{jump}}
\newcommand{\cont}{\mathrm{cont}}

\renewcommand{\L}{\mathcal L}
\renewcommand{\H}{\mathcal H}

\numberwithin{equation}{section}
\def\Xint#1{\mathchoice
	{\XXint\displaystyle\textstyle{#1}}%
	{\XXint\textstyle\scriptstyle{#1}}%
	{\XXint\scriptstyle\scriptscriptstyle{#1}}%
	{\XXint\scriptscriptstyle\scriptscriptstyle{#1}}%
	\!\int}
\def\XXint#1#2#3{{\setbox0=\hbox{$#1{#2#3}{\int}$ }
		\vcenter{\hbox{$#2#3$ }}\kern-.6\wd0}}

\def\dashint{\Xint-}

\title{Properties of mixing BV vector fields}
\author{Stefano Bianchini, Martina Zizza}
\date{\today}

\begin{document}

\maketitle

\begin{abstract}
   
   \noindent We consider the density properties of divergence-free vector fields $b \in L^1([0,1],\BV([0,1]^2))$ which are ergodic/weakly mixing/strongly mixing: this means that their Regular Lagrangian Flow $X_t$ is an ergodic/weakly mixing/strongly mixing measure preserving map when evaluated at $t=1$. \\
    Our main result is that there exists a $G_\delta$-set $\mathcal U \subset L^1_{t,x}([0,1]^3)$ made of divergence free vector fields such that
    \begin{enumerate}
    \item the map $\Phi$ associating $b$ with its RLF $X_t$ can be extended as a
continuous function to the $G_\delta$-set $\mathcal{U}$;
        \item ergodic vector fields $b$ are a residual $G_\delta$-set in $\mathcal{U}$;
        \item weakly mixing vector fields $b$ are a residual $G_\delta$-set in $\mathcal{U}$;
    \item strongly mixing vector fields $b$ are a first category set in $\mathcal{U}$;
    \item exponentially (fast) mixing vector fields are a dense subset of $\mathcal{U}$.
    \end{enumerate}
    The proof of these results is based on the density of BV vector fields such that $X_{t=1}$ is a permutation of subsquares, and suitable perturbations of this flow to achieve the desired ergodic/mixing behavior. These approximation results have an interest of their own. \\
    A discussion on the extension of these results to $d \geq 3$ is also presented. 
\end{abstract}
\begin{center}
    Preprint SISSA 19/2021/MATE
\end{center}



\tableofcontents
\input{introduction}
\input{preliminaries}
\input{cyclic}
\input{stronglymixing}

\input{permutations}
\input{appendix}

\printbibliography[title={Bibliography}]
\end{document}

%% file: introduction.tex
\section{Introduction}
	
Consider a divergence free vector field $b : \R^+ \times \mathbb T^d \to \R^d$ and the continuity equation
\begin{equation}
\label{Equa:orign_trasn}
    \partial_t \rho_t + D \cdot (\rho_t b_t) = 0, \quad \rho_{t=0} = \rho_0.
\end{equation}
In recent year the following question has been addressed: is the solution $\rho_t$ approaching weakly to a constant as $t \to \infty$? The meaning of "approaching to a constant" is usually formalized as
\begin{equation}
\label{Equa:conv_funct}
    \rho_t \underset{t\to \infty}{\to} \dashint \rho_0 \mathcal L^d \quad \text{weakly in $L^2$},
\end{equation}
since $\|\rho_t\|_{L^p}$ is constant (at least for positive solutions and sufficiently regular vector fields) and this is referred to as \emph{functional mixing} (another notion of mixing is the \emph{geometric mixing} introduced in \cite{Bressan_conj}, but for our purposes the functional mixing above is the most suitable, since it is related to Ergodic Theory. \\
Without any functional constraint on the space derivative $Db_t$, it is quite easy to obtain mixing in finite time: a well known example is \cite{Depauw}. A similar idea, used in a nolinear setting, can be found in \cite{Bressan_illposed}. The problem is usually formulated as follows: assume that $b \in L^\infty_{t,x} \cap L^\infty_t W^{s,p}_x$, what is the maximal speed of convergence in \eqref{Equa:conv_funct}? \\

\noindent This question has been addressed in several papers. In \cite{Alberti_mix} the 2d-case has been throughly analyzed, and the main results are the explicit construction of mixing vector fields when the initial data is fixed: the authors are able to achieve the optimal exponential mixing rate for the case $W^{1,p}$, $p > 1$, and study also the case $s < 1$ (mixing in finite time) and $s > 1$ (mixing at a polynomial rate). Recall that for $s=1, p>1$ the mixing is at most exponential \cite{Crippa_DeLellis}, while the same estimate in $W^{1,1}$ (or equivalently $\BV$) is still open \cite{Bressan_conj}. In \cite{univ:mixer} the authors construct a vector field which mixes at an exponential rate every initial data, and it belongs to $L^\infty_t W^{s,p}_x$ for $s < \frac{1+\sqrt{5}}{2}$, $p \in [1,\frac{2}{2s+1-\sqrt{5}})$. The autonomous 2d vector field is special, having an Hamiltonian structure: indeed in \cite{Marconi_Bonicatto_poly} the authors show that the mixing is polynomial with rate $t^{-1}$ when $b \in \BV$. \\

\noindent In this paper we consider the different problem: how many vector fields are mixing? More precisely, we study the mixing properties of flows generated in the unit square $K=[0,1]^2$ by divergence-free vector fields $b:[0,1]\times K \rightarrow \R^2$ belonging to the space $L^\infty([0,1],\BV(K))$: in order to avoid problems at the boundary, we assume that the vector field $b$ is divergence-free and $\BV$ when extended  to whole $\R^2$. In order to shorten the notation, we will sometimes write $\BV(K)$, $K = [0,1]^2$ as the space $\BV(\R^2) \cap \{\supp b \subset K\}$. \\
All the results stated here can be extended to the case $x \in \T^2$ with minor modifications; our choice is in the spirit of \cite{Shnirelman}. \\
In this setting, there exists a unique flow $t \rightarrow X_t \in C([0,1],L^1(K;K))$ (called Regular Lagrangian Flow (RLF)) of the ODE
\begin{equation*}
	\begin{cases}
		\frac{d}{dt} X_t(y) = b(t, X_t(y)), \\
		X_{t=0}(y)=y,
	\end{cases}
\end{equation*}
which is measure-preserving and stable, see \cite{Ambrosio:BV,Ambrosio:Luminy} and Section \ref{sect:RLF}. 
Our idea is to consider the $\mathcal L^2$-a.e. invertible measure preserving map $X_{t=1} : K\rightarrow K$ as an \emph{automorphism} of the measure space $(K,\mathcal{B}(K),\L^2\llcorner_K)$ and apply the tools of Ergodic Theory. Here and in the following $\L^2\llcorner_K$ is the Lebesgue measure on $K$ and $\mathcal{B}(K)$ are the Borel subsets of $K$. We call $G(K)$ the group of automorphisms of $K$. We underline that the additional difficulty is to retain that the maps under consideration are generated by a divergence-free vector field in $L^\infty_t \BV_x$. \\

\noindent There is a rich literature in Ergodic Theory that has deeply investigated the genericity properties of mixing for invertible and measure-preserving maps. These results are due mostly to Oxtoby and Ulam \cite{Oxtoby}, Halmos  \cite{Halmos:ergodic,Halmos:weak:mix}, Katok and Stepin \cite{Katok} and Alpern \cite{Alpern}.  They proved that the set of ergodic transformation is a residual $G_\delta$-set (i) in the set $G(K)$ with the \emph{neighbourhood topology}\footnote{The neighbourhood topology is indeed the convergence in measure, see Subsection \ref{topology}.}  \cite{Halmos:weak:mix}, (ii) in the set of measure-preserving homeomorphisms of a connected manifold with the strong topology\footnote{A sequence of maps $T_n\rightarrow T$ in the strong topology if $T_n\rightarrow T$  and $T_n^{-1}\rightarrow T^{-1}$ uniformly on $K$.} \cite{Oxtoby,Katok}. Moreover, the transformations satisfying a stronger condition known as \emph{weak mixing}, that is 
\begin{equation*}
            \lim_{n\to\infty}\frac{1}{n}\sum_{j=0}^{n-1}\left[\mu(T^{-j}(A)\cap B)-\mu(A)\mu(B)\right]^2=0,
        \end{equation*}
for every $A,B$ measurable sets, are still a residual $G_\delta$ set \cite{Halmos:weak:mix,Katok}. In 1976 Alpern showed that these problems are indeed connected by using the Annulus Theorem \cite{Alpern}. A different result holds for strongly mixing maps, i.e such that
 \begin{equation*}
            \lim_{n\to\infty} \mu(T^{-n}(A)\cap B) = \mu(A)\mu(B).
        \end{equation*}
It was shown firstly by Rokhlin in \cite{Rokhlin} (see \cite{Weiss} for an exposition of Rokhlin's work) and then by D. Ornstein \cite{Halmos:lectures} that (strongly) mixing maps are a first category set in the neighborhood topology. 

\smallskip

\noindent In these settings, the genericity properties of measure-preserving (weakly) mixing or ergodic maps are fairly understood; to our knowledge a similar analysis has not been done for flows generated by vector fields with additional regularity requirements (e.g. $b \in L^\infty_t \BV_x$).
The aim of our work is to extend the above genericity results to divergence-free vector fields whose Regular Lagrangian Flow is ergodic and weakly mixing (in dimension $d=2$, but see the discussion below on the extension to every dimension $d \geq 3$). 

\medskip

\noindent Let $b\in L^\infty([0,1],\BV (K))$ be a divergence-free vector field. 

\begin{definition}
      We say that $b$ is \emph{ergodic (weakly mixing, strongly mixing)} if its unique Regular Lagrangian Flow evaluated at time $t=1$ 
      is ergodic (respectively weakly mixing, strongly mixing). 
  \end{definition} 
  
  \noindent In the original setting \eqref{Equa:orign_trasn}, if $b_{t+1} = b_t$ (i.e. it is time periodic of period $1$), then it is fairly easy to see that if $b$ is strongly mixing as in the above definition then \eqref{Equa:conv_funct} holds, while for weakly mixing vector fields $b$ it holds the weaker limit
\begin{equation*}
      \lim_{T \to \infty} \frac{1}{T} \int_0^T \bigg( \int_K \bigg(\rho_t - \dashint \rho_0 \mathcal L^2\bigg) \phi\mathcal L^2\bigg) ^2 dt = 0, \quad \forall \phi \in L^2(K).
\end{equation*}

    \noindent Our main result is the following.
    
\begin{theorem}
\label{Theo:main_intro}
    There exists a $G_\delta$-subset $\mathcal U \subset L^1([0,1];L^1(K)) \cap \{b : D \cdot b = 0\}$ containing all divergence-free vector fields in $L^\infty([0,1];\BV(K))$ with the following properties:
    \begin{enumerate}
    \item \label{Point1:main_intro_Phi} the map $\Phi$ associating $b$ with its RLF $X_t$,
    \begin{equation*}
     \Phi:\lbrace b\in  L^\infty_t\BV_x: D\cdot b_t=0\rbrace \rightarrow C([0,1],L^1(K)),
     \end{equation*}
     can be extended as a
continuous function to the $G_\delta$-set $\mathcal{U}$;
        \item \label{Point2:main_intro_ergo} ergodic vector fields $b$ are a residual $G_\delta$-set in $\mathcal{U}$;
        \item \label{Point3:main_intro_weak} weakly mixing vector fields $b$ are a residual $G_\delta$-set in $\mathcal{U}$;
    \item \label{Point4:main_intro_firstcat} strongly mixing vector fields $b$ are a first category set in $\mathcal{U}$;
    \item \label{Point5:main_intro_expo} exponentially (fast) mixing vector fields are a dense subset of $\mathcal{U}$.
    \end{enumerate}
\end{theorem}
\noindent We will reasonably call the flow $X_t = \Phi(b)$, $b \in \mathcal{U}$, as the Regular Lagrangian Flow of $b$, even if we are outside the setting where RLF are known to be unique: however $X_t$ is the unique flow which can be approximated by RLF of smoother vector fields $b^n$ as $b^n \to b$ in $L^1$. The existence of such a set $\mathcal{U}$ is due to purely topological properties of metric spaces (Proposition \ref{prop:gen}).

\medskip

\noindent Our proof adapts some ideas from \cite{Halmos:weak:mix} to our setting: we give an outline of Halmos' analysis. First of all, both ergodic automorphisms and weakly mixing automorphisms are a $G_\delta$-set \cite{Halmos:ergodic,Halmos:weak:mix}. Next, it is shown that the mixing properties are invariant for conjugation, i.e. if $T : [0,1] \to [0,1]$ is weakly/strongly mixing and $R : [0,1] \to [0,1]$ is an automorphism, then $R \circ T \circ R^{-1}$ is weakly/strongly mixing too. It remains to be proved that weakly mixing maps are dense: define a \emph{permutation} as an automorphism of $[0,1]$ sending dyadic intervals (subintervals of $[0,1]$ with dyadic endpoints) into dyadic intervals by translation. Cyclic permutations (i.e. permutations made by a unique cycle) of the same intervals are clearly conjugate. One of the key ingredients of Halmos' proof is that, for every non periodic automorphism (i.e. $T^n x \not= x$ for all $n$ in a conegligible set of points $x$), there exists a cyclic permutation close to it in the neighbourhood topology, and by the previous observation about conjugation of permutations one deduces that if $T$ is non-periodic then the maps of the form $R\circ T\circ R^{-1}$ form a dense set. In particular, the weakly mixing maps are a $G_\delta$-set containing a non periodic map, hence it is residual.

\medskip

\noindent In our setting, the fact that ergodic/weakly mixing vector fields form a $G_\delta$-subset of $\mathcal{U}$ is an easy consequence of the Stability Theorem for Regular Lagrangian Flows and the definition of the map $\Phi$ (see Point \eqref{Point1:main_intro_Phi} and Proposition \ref{Prop:weak_mix_G_delta}). Indeed, since both ergodic automorphisms and weakly mixing automorphisms are a $G_\delta$-set \cite{Halmos:ergodic,Halmos:weak:mix}, then by the continuity of the map $\Phi$
associating $b$ with the RLF $X_t$, also ergodic and weakly mixing vector fields are a $G_\delta$-set. \\
\smallskip
\noindent Unfortunately, we cannot use conjugation of a RLF $X_t$ with an automorphism $R$ of $K$, since in general $R \circ X_{t=1} \circ R^{-1}$ is not a RLF generated by $b \in L^\infty_t BV_x$ (or even $b \in \mathcal{U}$). However, we are able to prove the density in $\mathcal{U}$ of vector fields $b \in L^\infty_t \BV_x$ whose RLF is a cyclic permutation of subsquares of $K$, which is the natural extension of the permutation of intervals used in \cite{Halmos:weak:mix}. More precisely, the map $T = X_{t=1}$ sends by a rigid translation subsquares of some rational  grid $\N\times\N\frac{1}{D}$, where $D\in\N$, into subsquares of the same grid (it will be clear later that being dyadic as in \cite{Halmos:weak:mix} is not relevant, see Lemma \ref{lem: dyadic} and Remark \ref{Shnirelman: sbaglio}), and as a permutation of subsquares it is made by a single cycle. \\
\noindent The precise statement is the following.
\begin{theorem}
\label{riassuntivo}
	Let $b\in L^\infty([0,1],\BV(K))$ be a divergence-free vector field. Then for every $\epsilon>0$ there exist $1 \ll D \in\N$, two positive constants $C_1,C_2$ and a divergence-free vector field $b^c\in L^\infty([0,1],\BV(K))$ such that
	\begin{equation}
	\label{finali}
		||b-b^c||_{L^1(L^1)}\leq\epsilon, \quad ||\TV(b^c)(K)||_{\infty}\leq  C_1||\TV(b)(K)||_{\infty} + C_2
		\end{equation}
	and the map $X^c_{t=1}:K\rightarrow K$, where $X^c_t:[0,1]\times K\rightarrow K$ is the flow associated with $b^c$, is a ${D}^2$-cycle of subsquares of size $\frac{1}{{D}}$.
\end{theorem}
\noindent The above approximation is the most technical part of the paper, and it is the point which forces to state the theorem in $\mathcal{U}$ and not in the original space $b \in L^\infty_t \BV_x$: indeed, while achieving the density in the $L^1_{t,x}$-topology, the total variation increases because of the constants $C_1,C_2$ in \eqref{finali}. (It is possible to improve the first estimate of \eqref{finali} to $\|b - b^c\|_{L^\infty L^1} \leq \epsilon$, see Remark \ref{Rem:better_mixing}, but to avoid additional technicalities we concentrate on the simplest results leading to Theorem \ref{Theo:main_intro}.)

\medskip

\noindent We remark that the above approximation result is sufficient to prove that strongly mixing vector fields are a set of first category (Proposition \ref{Prop:weak_mix_G_delta}): indeed, Theorem \ref{riassuntivo} shows the density in $\mathcal{U}$ of divergence free vector fields whose flow is made of periodic trajectories with the same period $D^2$. This observation is the key to obtain Point \eqref{Point4:main_intro_firstcat} of Theorem \ref{Theo:main_intro}.

\medskip

\noindent Looking at cyclic permutations of subsquares is an important step to obtain ergodic (and then weakly mixing and strongly mixing) vector fields: indeed, instead of studying the map $X^c_{t=1}$ (the RLF generated by $b^c$ of Theorem \ref{riassuntivo} above) in the unit square $K$ with the Lebesgue measure $\mathcal L^2$, it is sufficient to work in the finite space made of the centers of the subsquares
\begin{equation}
    \label{Equa:Omega_intro}
\Omega = \bigg\{ x = \bigg( \frac{k_1 - 1/2}{D},\frac{k_2 - 1/2}{D} \bigg), k_1,k_2 = 1,\dots,D \bigg\},
\end{equation}
where the measure-preserving transformation $X^c_{t=1}$ reduces to a cyclic permutations. In particular in $\Omega$ it is already ergodic.

\medskip

\noindent Since we cannot use the conjugation argument as we observed above, the final steps of the proof of Theorem \ref{Theo:main_intro} differ from the ones of  \cite{Halmos:weak:mix}. Indeed we give a general procedure to perturb vector fields $b^c \in L^\infty_t \BV_x$ (whose RLF $X_t^c$ at $t=1$ is a cyclic permutation of subsquares) into ergodic vector fields $b^e$ (strongly mixing vector fields $b^s$) still belonging to $L^\infty_t \BV_x$: here the explicit form of $X^c$ plays a major role, allowing us to construct \emph{explicitly} the perturbations to $b^c$ (Subsections \ref{subS:ergodicity},\ref{S:mixing_maps}). \\
The key idea is to apply the (rescaled) \emph{universal mixer} vector field (introduced in \cite{univ:mixer}) whose RLF at time $t=1$ is the \emph{Folded Baker's map}
\begin{equation}
\label{Equa:univers_U_intro}
    U=U\llcorner_{t=1}=\begin{cases}
        \left(-2x+1,-\frac{y}{2}+\frac{1}{2}\right) & x\in\left[0,\frac{1}{2}\right), \\
         \left(2x-1,\frac{y}{2}+\frac{1}{2}\right) & x\in\left(\frac{1}{2},1\right],
    \end{cases} \qquad y \in [0,1],
\end{equation}
to the subsquares of the grid $\N\times\N\frac{1}{D}$ given by Theorem \ref{riassuntivo}. \\
In order to achieve ergodicity, it is sufficient to apply the universal mixer $U$ inside a single subsquare, because the action of $X^c_{t=1}$ is already ergodic in $\Omega$ being a cyclic permutation. Together with the fact that ergodic vector fields are a $G_\delta$-set, this gives the proof of Point \eqref{Point2:main_intro_ergo} of Theorem \ref{Theo:main_intro}. \\
The perturbation to achieve exponential mixing is more complicated, since we need to transfer mass across different subsquares. The idea is to apply the universal mixer $U$ to adjacent couples of subsquares, a procedure which assures that the mass of $\rho$ is eventually equidistributed among all subsquares. The exponential mixing is a consequence of the finiteness of $\Omega$ and the properties of $U$ (see Proposition \ref{Prop:smix_dense}). This concludes the proof of Point \eqref{Point5:main_intro_expo} of Theorem \ref{Theo:main_intro}, and since strongly mixing vector fields are a subset of weakly mixing vector fields we obtain also Point \eqref{Point3:main_intro_weak}, concluding the proof of the theorem.

\medskip

\noindent A completely analogous result can be obtained in any dimension by adapting the above steps, at the cost of additional heavy technicalities. In this work we decided to sketch the proof of the key estimates (i.e. the ones requiring new ideas) in the general case (see Section \ref{Sss:d-dim} and the comments in next section).

\subsection{A discussion of the key points and results leading to the main theorem} 

The sketch of the proof given above includes steps which are somehow standard either in the theory of linear transport or in ergodic analysis. Other points of the proof instead require to introduce new ideas or at least to significantly develop tools present in the  literature. This section is devoted to expand these critical parts in order to help the reader to understand the novelties contained in this work. \\
In most proofs we tried to get the best possible result in term of the $L^1$ and $\BV$ norms, with the hope to prove a similar statement inside the compact subset $\{b:\|b_t\|_{L^1}, \TV(b_t) \leq C\} \subset L^1_{t,x}$. However there are some delicate arguments where we have to increase the total variation of $b$ of a fixed amount: we will point out these points here below.

\medskip

\noindent The most technical part of this paper is the proof of the approximation theorem through cyclic permutations, Theorem \ref{riassuntivo}. It is based on two results: the first one, whose proof is the content of Section \ref{S:permutation}, is an approximation through vector fields whose RLF $X_t$ at time $t=1$ is a permutation of subsquares. The second one exploits the classical result that every permutation is a product of disjoint cycles (Proposition \ref{unique:cycle}), in order to merge these cycles into a single one.

\subsubsection{Approximations of flows by permutations}

The approximation through divergence free vector fields $b$ whose flow at $t=1$ is a permutation of squares has been already studied in \cite{Shnirelman} in the context of generalized flows for incompressible fluids. Indeed the starting point of Section \ref{S:permutation} is Lemma \ref{lem: dyadic}, whose statement is almost identical to Lemma 4.3 of \cite{Shnirelman}: it says that if $T$ is a smooth map sufficiently close to identity, there exists an arbitrarily close flow $\sigma_t$, $t \in [0,1]$, such that $\sigma_{t=0} = T$ and $\sigma_{t=1}$ maps affinely rectangles whose edges are on a dyadic grid into rectangles belonging to the same grid. Even if the ideas of the proof are completely similar to the original ones, we choose to make them more explicit (see also Remark \ref{Shnirelman: sbaglio} for some comments on the original proof). The extension to general dimension seems non-trivial (at least to us): we decided to sketch the main ideas in Section \ref{Sss:d-dim}, leaving the details to the interested reader (we observe however that our approach is somehow different from the $2d$ case, differently from what it is stated in \cite{Shnirelman}).

\medskip

\noindent At this point the proof diverges from \cite{Shnirelman}, due to the fact that in his case one has to control the $L^2$-norm of the vector field while here we need to build a perturbation of a vector field (not of a map) and to estimate its $\BV_x$-norm. In Lemma \ref{BV:estim:pert:flow} we prove that the perturbation $\sigma_t$ constructed in the above paragraph (i.e. in Lemma \ref{lem: dyadic}) can be encapsulated inside the flow $X_t$ so that the resulting vector field is close in $L^\infty_t L^1_x$ and remains in $L^\infty_t \BV_x$ if the grid is sufficiently small, always under the assumption that $X_t$ is close to identity.

\medskip

\noindent We finally arrive to the approximation theorem through permutations (Theorem \ref{Prop:piececlo}), which we think that  can have an independent interest:

\begin{theorem}
\label{Prop:piececlo_intro}
Let $b\in L^\infty([0,1];BV(K))$ be a divergence-free vector field. Then for every $\epsilon>0$ there exist $\delta', C_1,C_2>0$ positive constants, $D \in \N$ arbitrarily large and a divergence-free vector field $b^\epsilon\in L^\infty([0,1];BV(K))$ such that 
\begin{enumerate}
\item $\supp b^\epsilon_t\subset\subset K^{\delta'} = [\delta',1-\delta']^2$,
\item it holds
\begin{equation}
\label{main result_intro}
    \|b - b^\epsilon\|_{L^\infty(L^1)} \leq \epsilon, \quad || \TV(b^\epsilon) (K)||_{\infty}\leq  C_1 ||\TV(b)(K)||_\infty + C_2,
\end{equation}
\item the map $X^\epsilon\llcorner_{t=1}$ generated by $b^\epsilon$ at time $t=1$ translates each subsquare of the grid $\N \times \N \frac{1}{D}$ into a subsquare of the same grid, i.e. it is a permutation of squares.
\end{enumerate}
\end{theorem}

\noindent We remark that in the statement of Theorem \ref{Prop:piececlo} it is also assumed that there exists $\delta>0$ such that for $\mathcal{L}^1$-a.e. $t\in[0,1]$, $\supp b_t \subset\subset K^\delta$. This is for technical reason, as standard approximation methods allows to deduce the Theorem \ref{Prop:piececlo_intro} above from Theorem \ref{Prop:piececlo}. \smallskip

\noindent The starting point of its proof is to divide the time interval $[0,1]$ into subintervals $[t_i,t_{i+1}]$ and apply the previous perturbations (Lemma \ref{lem: dyadic}) to $b_t$, $t \in [t_i,t_{i-1}]$. We however need an additional mechanism in order to obtain a permutation of subsquares and not a piecewise affine map at $t=1$, as it would be the case if we only use the perturbations above.

\smallskip

\noindent The introduction of this new perturbation is done in Section \ref{Ss:BV_rot}: the idea is that if a measure preserving map $T$ is diagonal with rational eigenvalues, then there exists a subgrid and a map $R$ made by two rotations such that $T \circ R$ maps subsquares of the new grid into subsquares instead of rectangles (Lemma \ref{lem:rotations}). The key point is that the total variation of the new map is bounded independently on the grid size, while the $L^1$-norm converges to $0$ as the grid becomes smaller and smaller. This gives better BV estimates than the construction of \cite{Shnirelman}. \\
In the proof of the theorem, this rotation mechanism has to act differently in each subrectangle where the map is affine at all times $t_i$ of the time partition: this means that the rotation mechanism is calibrated on the single subrectangle, depending on its evolution, and it differs from one to the other. More precisely, whenever a rectangle is distorted of a fixed amount, we let the rotation to act on a subgrid in order to obtain a permutation of subsquares (see Figure \ref{Fig:2-dim_case_rot_intro}). Without entering into technicalities, we just notice that this has to be done during the time evolution, and the grid has to be chosen sufficiently small so that it is the same for all rotations.

\begin{figure}[t]
\centering
\def\svgwidth{\textwidth}
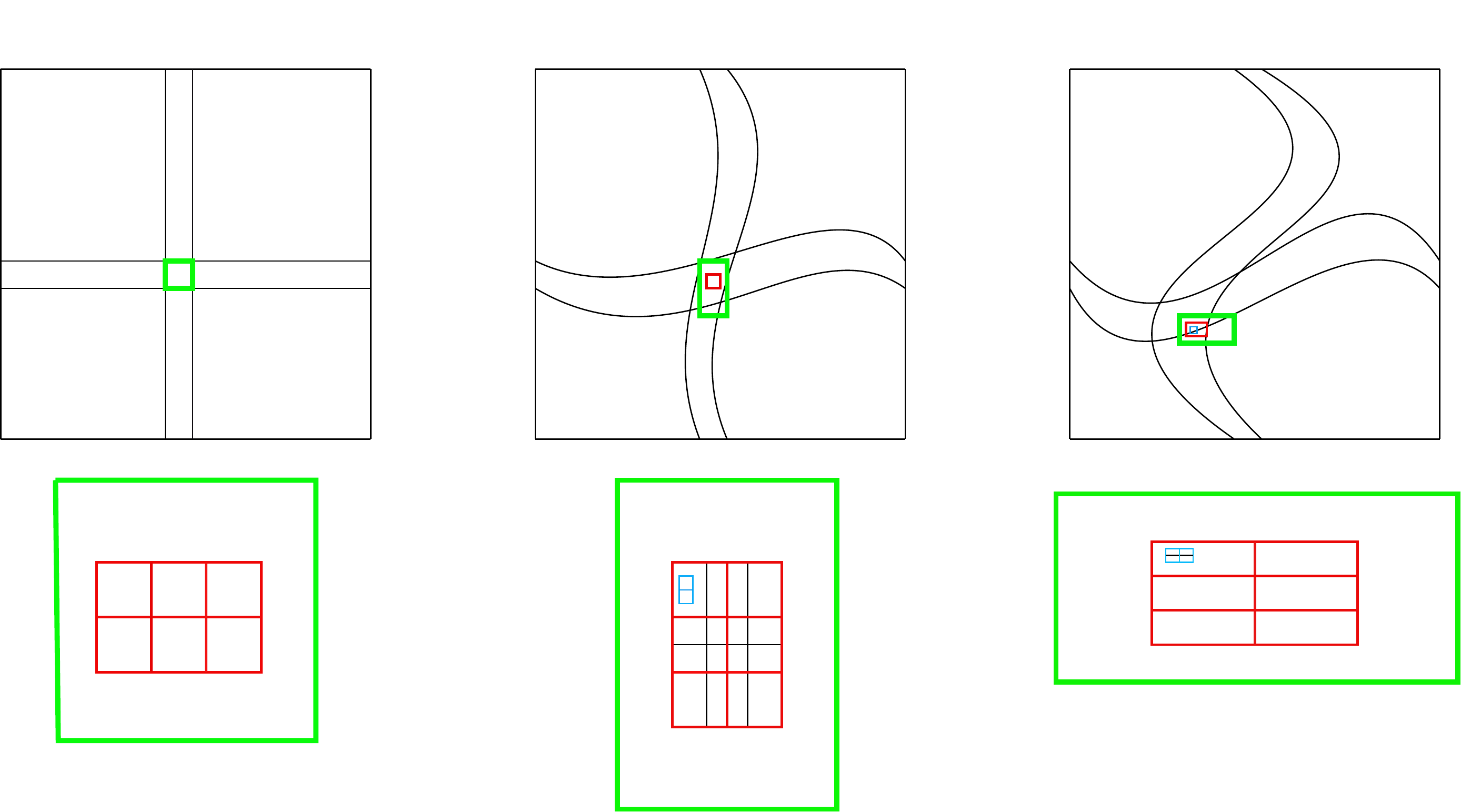
\caption{A graphic explanation of the action of rotations: the three top frames above shows the evolution of curves at different times $t_i$, where the perturbation of Lemma \ref{lem: dyadic} takes care of transforming the green square affinely into the green rectangles. The three bottom frames is the action of the rotation on a finer grid: the red rectangle is chosen so that the action of the affine map $X_{t=t_1}$ coincides with a rotation by $\pi/2$ (as a set, but the black grid (image through an affine map) is not the image of the red grid after a rotation), and then the red grid is mapped into itself when composed with a rotation of $\pi/2$ inside the red rectangle. At the next step, one chooses again a finer grid (the light blue one) to perform the same transformation, so that the blue grid is mapped into itself.}
\label{Fig:2-dim_case_rot_intro}
\end{figure}

\smallskip

\noindent The interesting part of the above theorem is the form of the estimate for the Total Variation in \eqref{main result_intro}. The constant $C_1$ is easily explained: in the first two paragraph above we explained how to perturb a smooth map/flow into a piecewise affine flow on subrectangles. The key argument is that we need to contrast rotations of the original vector field $b_t$, and then this action needs in principle a total variation of the order of $\TV(b_t)$: we believe that this constant $C_1$ can be optimized, but it is not necessary here, because the hard term is the one leading to $C_2$. \\
Indeed, the counter-rotation is of the order of the area of the region where it is acting (see Lemma \ref{TV:rotations}), if it is not too much distorted. The partition into times where these counter-rotations act is then related to their distorsion, which is related to the BV norm spent by the vector field. However, it is in general necessary to use a final counter-rotation, and this is the main reason for the constant $C_2$.

\smallskip

\subsubsection{From permutations of subsquares to ergodic/exponential mixing}

The advantage of having a flow $X_t$ such that $X_{t=1}$ is a permutations of subsquares is that its action is sufficiently simple to perturb in order to achieve a desired property. Nevertheless it requires some smart constructions, since in any case we need to control the $L^1$-distance and the BV norm.

\smallskip

\noindent The first step is to perturb a permutation of subsquares into a cyclic permutation of subsquares, i.e. a permutation made of a single cycle: this is clearly a necessary condition for ergodicity. Roughly speaking, the idea is to exchange two adjacent subsquares belonging to different cycles in order to merge them. We do this operation in two steps. \\
In Lemma \ref{lem:cycl:1} it is shown that one can arbitrarily refine the grid $\N \times \N \frac{1}{D}$ into $\N \times \N \frac{1}{DM}$ so that each cycle of length $k$ in the original grid becomes a cycle of length $k M^2$ in the new one. Moreover the perturbation is going to $0$ in $L^\infty_t L^1_x$ as $M \to \infty$ and its $L^\infty \BV_x$ is arbitrarily small when $D$ is large. \\ 
The above result allows now to exchange sets of size $(DM)^{-1}$ when merging cycles: this is done in Proposition \ref{unique:cycle}. This proposition faces a new problem: in the previous case the exchange of subsquares of size $(DM)^{-1}$ occurs within the same subsquare of size $D^{-1}$: the latter is only deformed during the evolution and hence the merging can be done in the whole time interval $[0,1]$. In the case of Proposition \ref{unique:cycle}, instead, we are exchanging subsquares of size $(DM)^{-1}$ which are then shifted away during the flow, since they belong to different subquares of the grid $D^{-1}$. This requires to do the exchange sufficiently fast (i.e. during the time where they share a common boundary, Remark \ref{Rem:better_mixing}), or to freeze the evolution for an interval of time $[0,\delta]$ and perform the exchanges here and then let the flow permuting the subsquares to evolve during the remaining time interval $[\delta,1]$. We choose for simplicity this second line, being easier and not changing the final result: notice however that now the constant $M$ plays the role of controlling the constant $\delta^{-1}$, appearing because the exchange action occurs in the time interval $[0,\delta]$.

\smallskip

\noindent Once we have a cyclic permutation of subsquares, the perturbation to get an ergodic vector field is straighforward: indeed as a map of the quotient space $\Omega$ (defined as in \eqref{Equa:Omega_intro} for the grid $\N \times \N \frac{1}{DM}$) it is already ergodic, hence one has only to add an ergodic vector field inside each subsquare: this is done in Proposition \ref{Prop:ergod_dense}, where the universal mixer $U$ given by \ref{Equa:univers_U_intro} is used inside a single subsquare. In particular, we can spread its action on the whole time interval $[0,1]$, and the smallness estimates are optimal.

\smallskip

\noindent To achieve the exponential mixing, instead, we need to transfer mass across different subsquares, and hence we face again the problem of Proposition \ref{unique:cycle} above: we let the mixing action occurs in an interval of time where the evolution is frozen, and then let the cyclic permutation to act in the time interval $[\delta,1]$ (see also Remark \ref{Rem:as_before}). The idea is again to use the universal mixer \eqref{Equa:univers_U_intro} to exchange mass across to nearby subsquares. The additional difficulty here is that in order to avoid resonant phenomena we mix all squares with 2 neighboring ones, so that by simple computations the Markov Shift obtained through this map is exponentially mixing, Proposition \ref{Prop:smix_dense}.

\smallskip

\noindent To collect all the above results into a proof of Point \eqref{Point5:main_intro_expo} of Theorem \ref{Theo:main_intro} is not difficult at this point, and we devote a section (Section \ref{Ss:proof_main_intro} and Corollary \ref{Cor:strong_mix_dense}) to shows how to merge these result and get the desired statement.

\subsection{Plan of the paper} 

The paper is organized as follows. 

\smallskip

\noindent In Section \ref{S:preliminaries}, after listing some of the notation used in the paper, we give a short overview on BV functions (Section \ref{Ss:BV_funct}) and Regular Lagrangian Flows (Section \ref{sect:RLF}), proving the extension of the continuous dependence to a complete set $\mathcal U$ in Proposition \ref{prop:gen} (providing the proof of Point (\ref{Point1:main_intro_Phi}) of Theorem \ref{Theo:main_intro}) and stating some technical estimates on composition of maps (Theorem \ref{change of variables} and \eqref{Equa:comp_ff}, \eqref{Equa:formula_comp}) and on the vector field \eqref{rot:flow:square} generating a rotation (Lemma \ref{TV:rotations}).

\smallskip

\noindent In Section \ref{S:ergodic} we collect some classical results in Ergodic Theory which are needed for Theorem \ref{Theo:main_intro}, and also give the proof of the $G_\delta$-properties of the set of ergodic/weakly mixing vector fields of Theorem \ref{Theo:main_intro}. First we introduce the basic definitions, then in Section \ref{topology} we clarify the relation with the neighborhood topology and the $L^1$-topology used in Theorem \ref{Theo:main_intro}. \\
In Section \ref{S:gen_weak_mix} we restate in our setting the well known fact that weakly mixing are a $G_\delta$-set, as well as the first category property of strongly mixing vector fields (Proposition \ref{Prop:weak_mix_G_delta}). The proof of the remaining parts of Theorem \ref{Theo:main_intro} is a corollary of the previous statement (Corollary \ref{Cor:strong_mix_dense}), if we know that the strongly mixing vector fields are dense. \\
The construction of exponential mixing vector fields is based on the analysis of Markov Shift: in Section \ref{subS:markov} we give the results which are linked to our construction.

\smallskip

\noindent In Section \ref{S:density} we present the proof of the the density of exponentially mixing vector fields, under the assumption that permutation flows are dense in $L^1_t \BV_x$ w.r.t. the $L^1_{t,x}$-norm. We decide to put first this construction because it is in some sense independent on the proof of the density of permutation flows: the idea is that different functional settings can be studied by changing this last part (i.e. the density of permutation flow), while keeping the construction of approximation by permutations more or less the same. \\
The first statement is Lemma \ref{lem:cycl:1} which allows to partition the subsquares of a given cycle into smaller subquares still belonging to the same cycle: as we already said, its proof is based on a doubling scale approach, merging larger and larger cycles. The usefullness of this estimate is shown in Proposition \ref{union:subrectangles}, where we need to exchange mass only on an area which is of order $M^{-2}$, and hence obtaining that the perturbation is small in $L^1_{t,x}$ and $L^\infty_t BV_x$ (Proposition \ref{unique:cycle_star} of Remark \ref{Rem:better_mixing} addresses the problem of exchanging two subquares during the evolution, a refinement not needed for the proof of Theorem \ref{Theo:main_intro}). \\
The last two subsection address the density of ergodic vector fields (Proposition \ref{Prop:ergod_dense}) and of exponentially mixing vector fields (Proposition \ref{Prop:smix_dense}): the basic idea is the same (i.e. perturb the cyclic permutation) but the strongly mixing case requires more effort and uses the Markov chains. Section \ref{Ss:proof_main_intro} shows at this point how the assumptions of Corollary \ref{Cor:strong_mix_dense} are verified, concluding the proof of Theorem \ref{Theo:main_intro} under the assumption of the density of vector fields whose flow is a permutation of subsquares.

\smallskip

\noindent The last section, Section \ref{S:permutation}, proves the cornerstone approximation result, i.e. the density of vector fields whose flow at $t=1$ is a permutation of subsquares, Theorem \ref{main result} (whose statement is the same of Theorem \ref{main result_intro}). \\
In Section \ref{Ss>pert_flow} we approximate a smooth flow close to identity with a BV flow which is locally affine in subrectangles: Lemma \ref{lem: dyadic} considers the 2d-case as in \cite{Shnirelman}, while the needed variations for the $d$-dimensional case are in Section \ref{Sss:d-dim}. \\
The BV estimates for such perturbed flow are studied in Section \ref{Ss>BV_pertu}. A preliminary result (Lemma \ref{Lem:zeta_map_X}) takes care of the conditions that the area of the subsquares has to be a dyadic rational, while the key estimates are in Lemma \ref{BV:estim:pert:flow}: an important fact is that as the grid becomes finer the perturbation becomes smaller. \\
An ingredient for obtaining a flow which is a permutation of subsquares is the use of rotations: in Section \ref{Ss:BV_rot} we study these elementary transformations. \\
The main approximation theorem, Theorem \ref{Prop:piececlo}, is stated and proved in Section \ref{Ss:main_appro_th}. Its proof uses all the ingredients of the previous sections, and an additional argument on how to encapsulate rotations in order to control the total variation.






%% file: Brenier/subsquares_1_intro.pdf_tex
\begingroup%
  \makeatletter%
  \providecommand\color[2][]{%
    \errmessage{(Inkscape) Color is used for the text in Inkscape, but the package 'color.sty' is not loaded}%
    \renewcommand\color[2][]{}%
  }%
  \providecommand\transparent[1]{%
    \errmessage{(Inkscape) Transparency is used (non-zero) for the text in Inkscape, but the package 'transparent.sty' is not loaded}%
    \renewcommand\transparent[1]{}%
  }%
  \providecommand\rotatebox[2]{#2}%
  \newcommand*\fsize{\dimexpr\f@size pt\relax}%
  \newcommand*\lineheight[1]{\fontsize{\fsize}{#1\fsize}\selectfont}%
  \ifx\svgwidth\undefined%
    \setlength{\unitlength}{799.29230938bp}%
    \ifx\svgscale\undefined%
      \relax%
    \else%
      \setlength{\unitlength}{\unitlength * \real{\svgscale}}%
    \fi%
  \else%
    \setlength{\unitlength}{\svgwidth}%
  \fi%
  \global\let\svgwidth\undefined%
  \global\let\svgscale\undefined%
  \makeatother%
  \begin{picture}(1,0.55575438)%
    \lineheight{1}%
    \setlength\tabcolsep{0pt}%
    \put(0,0){\includegraphics[width=\unitlength,page=1]{subsquares_1_intro.pdf}}%
    \put(0.08491887,0.52723804){\color[rgb]{0,0,0}\makebox(0,0)[lt]{\lineheight{1.25}\smash{\begin{tabular}[t]{l}$t=0$\end{tabular}}}}%
    \put(0.45127585,0.52583924){\color[rgb]{0,0,0}\makebox(0,0)[lt]{\lineheight{1.25}\smash{\begin{tabular}[t]{l}$t=t_1$\end{tabular}}}}%
    \put(0.82182928,0.52723804){\color[rgb]{0,0,0}\makebox(0,0)[lt]{\lineheight{1.25}\smash{\begin{tabular}[t]{l}$t=t_2$\end{tabular}}}}%
  \end{picture}%
\endgroup%

%% file: preliminaries.tex
\section{Preliminaries and notations}
\label{S:preliminaries}

First, a list of standard notations used throughout this paper.

\begin{itemize}
\item $\Omega\subset \R^n$ denotes in general an open set; $\mathcal{B}(\Omega)$ denotes the $\sigma$-algebra of Borel sets of $\Omega$; 
\item  $\text{dist}(x, A) $ is the distance of $x$ from the set $A\subset \Omega$, defined as the infimum
of $|x-y|$ as $y$ varies in $A$;
\item  $\forall A\subset \Omega$, $\mathring{A}$ denotes the interior of $A$ and $\partial A$ its boundary, moreover, if $\epsilon>0$, then $A^\epsilon$ is the $\epsilon$-neighbourhood of $A$, that is
    \begin{equation*}
        A^\epsilon=\lbrace x\in \Omega: \text{dist}(x,\partial A)\leq\epsilon\rbrace;
    \end{equation*}
    \item $\M_b(\Omega)$ bounded Radon measures;
    \item if $\nu\in\M_b(\Omega)$ then $\|\nu\|$ denotes its total variation;
    \item $\BV(\Omega)$ is the set of functions with bounded variation, and if $u\in\BV(\Omega)$ we will use instead $\TV(u)$ to denote $\|Du\|$;
    \item $\L^d$ denotes the $d$-dimensional Lebesgue measure on $\R^d$, and $\H^k$ the $k$-dimensional Hausdorff measure;
    \item $K=[0,1]^2$ is the unit square;
    \item  $\L^2\llcorner_K$ denotes the normalized Lebesgue measure on $K$;
    \item let $b:[0,1]\times \R^2\rightarrow \R^2$, and let $t,s\in[0,1]$ then we denote by $X(t,s,x)$ a solution of
    \begin{equation*}
        \begin{cases}
         \dot x(t)=b(t,x(t)) \\
         x(s)=x,
        \end{cases}
    \end{equation*}
    moreover we will use $X(t)(x)$ or (alternatively $X_t(x)$) for $X(t,0,x)$ (in our setting as a flow the function $X(t,s,x)$ is unique a.e.);
    \item $(S,\Sigma,\mu)$ denotes a locally compact separable metric space where $\mu$ is a normalized complete measure;
    \item $G(S)$ denotes the space of automorphisms of $S$.
\end{itemize}

\subsection{BV functions}
\label{Ss:BV_funct}

In this subsection we recall some results concerning functions of bounded variation. For a
complete presentation of the topic, see \cite{AFP}. Let  $u\in\BV(\Omega;\R^m)$ and $Du\in\M_b(\Omega)^{n\times m}$ the $n\times m$-valued measure representing its distributional derivative. We recall the decomposition of the measure $Du$ 
\begin{equation*}
    Du=D^{\textrm{cont}}u+D^\textrm{jump}u = D^\textrm{a.c.} u + D^\mathrm{cantor} u + D^\mathrm{jump} u,
\end{equation*}
where $D^\mathrm{cont} u, D^\mathrm{a.c.} u, D^\mathrm{cantor} u, D^\mathrm{jump}u$  are respectively  the continuous part, the absolutely continuous part, the Cantor part and the jump part of the measure. 
We also recall that for $u\in BV(\Omega)$ the following estimate on the translation holds: for every $C\subset\Omega$ compact and $z\in\R^n$ such that $|z|\leq \dist(C,\partial\Omega)$
\begin{equation}
\label{est:translat}
    \int_C |u(x+z)-u(x)|dx\leq \left| \sum_{i=1}^n z_iD_iu\right|(C^{|z|}).
\end{equation}

\subsection{Regular Lagrangian Flows}\label{sect:RLF}
Throughout the paper we will consider divergence-free vector fields $b:[0,1]\times K\rightarrow \R^2$ in the space $L^\infty([0,1];\BV(K))$ (in short $b\in L^\infty_t\BV_x$) such that $\supp(b_t)\subset\subset\mathring{K}$ for $\L^1$-a.e. $t\in [0,1]$: it is standard to extend the analysis to divergence-free BV-vector fields in $\R^2$ with support in $K$. When the velocity field $b$ is Lipschitz, then its \emph{flow} is well-defined in the classical sense, indeed it is the map $X:[0,1]\times K\rightarrow K$ satisfying
\begin{equation*}
    \begin{cases}
    \frac{d}{dt} X_t(x)=b(t, X_t(x)); & \\
    X_0(x)=x.
    \end{cases}
\end{equation*}
\noindent But when we allow the velocity fields to be discontinuous (as in our case $\BV$ regular in space) we can still give a notion of a flow (namely the \emph{Regular Lagrangian Flow}). These flows have the advantage to allow rigid \emph{cut and paste} motions, since they do not preserve the property of a set to be connected. More in detail, we give the following
\begin{definition}
    Let $b\in L^1([0,1]\times \R^2; \R^2)$. A map $X:[0,1]\times \R^2\rightarrow \R^2$ is a \emph{Regular Lagrangian Flow} (RLF) for the vector field $b$ if
    \begin{enumerate}
        \item for a.e. $x\in \R^2$ the map $t\rightarrow X_t(x)$ is an absolutely continuous integral solution of
        \begin{equation*}
            \begin{cases}
    \frac{d}{dt} x(t)=b(t, x(t)); & \\
    x(0)=x.
    \end{cases}
        \end{equation*}
        \item there exists a positive constant $C$ independent of $t$ such that
        \begin{equation*}
            \L^2(X_t^{-1}(A))\leq C\L^2(A),\quad\forall A\in\mathcal{B}(\R^2).
        \end{equation*}
    \end{enumerate}
\end{definition}
\noindent   DiPerna and Lions proved existence, uniqueness and stability for Sobolev vector fields with bounded divergence \cite{DiPerna:Lions}, while the extension to the case of $\BV$ vector fields with divergence in $L^1$ has been done by Ambrosio in \cite{Ambrosio:BV}. When dealing with divergence-free vector fields $b$ the unique Regular Lagrangian Flow  $t\rightarrow X_t$ associated with $b$ is a flow of measure-preserving maps, main objects of investigations in Ergodic Theory. In the sequel we will build flows of measure-preserving maps originating from divergence-free vector fields; more precisely, if a flow $X:[0,1]\times K\rightarrow K$ is invertible, measure-preserving for $\mathcal{L}^1$-a.e. $t$ and the map $t\rightarrow X_t$ is differentiable for $ \mathcal{L}^1$-a.e. $t$ and  $\dot{X}_t\in L^1(K)$, then the \emph{vector field associated with $X_t$} is the divergence-free vector field defined by
	\begin{equation*}
	b_t(x)=b(t,x)=\dot{X}_t(X^{-1}_t(x)).
	\end{equation*}

\begin{theorem}[Stability, Theorem 6.3 \cite{Ambrosio:Luminy}]
		Let   $b_n,b\in L^\infty([0,1],\BV(K))$ be divergence-free vector fields and let $X^n,X$ be the corresponding Regular Lagrangian Flows. Assume that 
		\begin{equation*}
			||b_n-b||_{L^1_{t,x}}\to 0\quad\text{as }n\to\infty,
		\end{equation*}
		then
		\begin{equation*}
			\lim_{n\to\infty}\int_{K} \sup_{t\in[0,1]} |X^n_t(x)-X_t(x)|dx =0.
		\end{equation*}
	\end{theorem}
	
\noindent In this setting we can extend the family of vector fields we consider to a Polish subspace of $L^1_{t,x}$ in which we still have a notion of uniqueness. This extension allows us to apply Baire Category Theorem for the results of genericity that we will give for weakly mixing vector fields. 
\begin{proposition}[Extension] 
\label{prop:gen}
 Let \begin{equation*}
     \Phi:\lbrace b\in  L^\infty_t\BV_x: D\cdot b_t=0\rbrace \subset \{ b \in L^1([0,1],L^1(K)), D\cdot b_t = 0\} \rightarrow C([0,1],L^1(K))
     \end{equation*}
     the map that associates $b$ with its unique Regular Lagrangian Flow $X_t$.  Then $\Phi$ can be extended as a continuous function to a $G_\delta$-set $\mathcal{U} $ containing $\lbrace b\in L^\infty_t\BV_x: D\cdot b_t=0\rbrace$. 
\end{proposition}

This proposition proves Point (\ref{Point1:main_intro_Phi}) of Theorem \ref{Theo:main_intro}.

\begin{proof}
We recall that for every $f:A\rightarrow Z$ continuous where $A\subset W$ is metrizable and $Z$ is a complete metric space, there exists a $G_\delta$-set $A\subset G$ and a continuous extension $\tilde{f}:G\rightarrow Z$ (Proposition 2.2.3, \cite{Srivastava}). Thus we have to prove the continuity of the map $\Phi$ which follows by
\begin{equation*}
\begin{split}
\|\Phi(b^n) - \Phi(b)\|_{C_t L^1_x} &= \sup_{t \in [0,1]} \int_K \big| X_t^n(x) - X_t(x) \big| dx \\
&\leq \int_K \sup_{t \in [0,1]} \big| X_t^n(x) - X_t(x) \big| dx.
\end{split}
\end{equation*}
This concludes the proof.
%
%
%
\end{proof}

\noindent We will also use the following tools to prove the main approximation theorems of the paper. The first one gives a rule to compute the total variation of the composition of vector fields, while the second one is a direct computation of the cost, in terms of the total variation of the vector field whose flow rotates rectangles. 

\begin{theorem}[Change of variables, Theorem 3.16, \cite{AFP}]
	\label{change of variables}
	Let $\Omega,\Omega'$ two open subsets of $\R^n$ and let $\phi:\Omega\rightarrow \Omega'$ invertible with Lipschitz inverse, then $\forall u\in BV(\Omega')$ the function $v=u\circ\phi$ belongs to $BV(\Omega)$ and 
	\begin{equation*}
	\TV(v)\leq \Lip(\phi^{-1})^{n-1} \TV(u).
	\end{equation*}
\end{theorem}
\begin{corollary}
Let $\Omega,\Omega'\subset\R^n$ open sets where  $\partial\Omega$ is Lipschitz and let $\phi:\bar\Omega\rightarrow \bar \Omega'$ invertible with Lipschitz inverse, then $\forall u\in BV(\R^n)$ the function 
\begin{equation*}
    v=\begin{cases} u \circ \phi & x \in \Omega, \\
    0 & \text{otherwise},
     \end{cases}
\end{equation*}
belongs to $BV(\R^n)$ and 
	\begin{equation*}
\begin{split}
	\TV(v)(\R^n) &\leq \Lip(\phi^{-1})^{n-1} \big( \TV(u)(\Omega') + \|\mathrm{Tr}(u,\partial \Omega')\|_{L^1(\mathcal H^{n-1}\llcorner_{\partial \Omega'})} \big) \\
&= \Lip(\phi^{-1})^{n-1} \TV(u)(\R^n).
\end{split}
	\end{equation*}
\end{corollary}

In the following, we have often to study the properties of the vector field $b_3$ associated with the composition $Y_3(t)$ of two smooth measure preserving flows $t \mapsto Y_i(t)$, $i=1,2$, with associated vector fields $b_1,b_2$. By direct computation
\begin{equation*}
    b_3(t,Y_3(t,y)) = \partial_t Y_1(t,Y_2(t,y)) = b_1(t,Y_3(t,y)) + \nabla Y_1(t,Y_2(t,y)) b_2(t,Y_2(t,y)),
\end{equation*}
\begin{equation}
\label{Equa:comp_ff}
\begin{split}
        b_3(t,x) &= b_1(t,x) + \nabla Y_1(t,Y_2(t,Y_3^{-1}(t,x))) b_2(t,Y_2(t,Y_3^{-1}(t,x))) \\
        &= b_1(t,x) + \nabla Y_1(t,Y_1^{-1}(t,x)) b_2(t,Y_1^{-1}(t,x)).
\end{split}
\end{equation}
Hence using Theorem \ref{change of variables} we conclude that (being $Y_1 \circ Y_2$ measure preserving too)
\begin{equation}
\label{Equa:formula_comp}
\begin{split}
    \TV(b_3) &\leq \TV(b_1) + \Lip(Y_1)^{n-1} \TV(DY_1(t) b_2) \\
&\leq \TV(b_1) + \|\nabla Y_1\|^n_\infty \TV(b_2) + \|\nabla Y_1\|^{n-1}_\infty \|b_2\|_\infty \TV(DY(t)).
\end{split}
\end{equation}

\noindent Throughout the paper we will extensively use a flow rotating rectangles and the vector field associated with it.  More precisely we define the \emph{rotation flow} $r_{t}:K\rightarrow K$ for $t\in[0,1]$ in the following way: call
\begin{equation*}
	V(x)=\max \left\{ \left| x_1-\frac{1}{2}\right|, \left| x_2-\frac{1}{2}\right|\right\}^2, \quad (x_1,x_2)\in K.
\end{equation*}
Then the \emph{rotation field} is $r:K\rightarrow \R^2$
\begin{equation}
	\label{rotation field}
	{r}(x)=\nabla V^\perp(x),
\end{equation}
where $\nabla^\perp =(-\partial_{x_2},\partial_{x_1})$ is the orthogonal gradient. Finally the rotation flow $r_{t}$ is the flow of the vector field $r$, i.e. the unique solution to the following ODE system:
\begin{equation}
\label{rot:flow:square}
	\begin{cases} 
	\dot {r}_{t}(x)={r}(r_{t}(x)), \\
	r_{0}(x)=x.
	\end{cases}
\end{equation}
This flow rotates the cube counterclockwise of an angle $\frac{\pi}{2}$ in a unit interval of time.
\begin{lemma}\label{TV:rotations}
Let $R\subset\R^2$ a rectangle of sides $a,b>0$. Consider the rotating flow
\begin{equation*}
    R_{t}=\chi^{-1}\circ r_t\circ \chi,
\end{equation*}
where $\chi:R\rightarrow K$ is the affine map sending $R$ into the unit cube and $r_t$ is the rotation flow defined in \eqref{rot:flow:square}. Let $b_t^R$ the divergence-free vector field associated with $R_t$. Then 
\begin{equation*}
\TV(b^R_t)(\R^2) = 4 a^2 + 4 b^2, \quad \forall t\in [0,1].
\end{equation*}
\end{lemma}

\begin{proof}
The potential $V$ generating the rotation of $\pi/2$ in this case is the function
\begin{equation*}
V(x) = \max \bigg\{ \frac{b}{a} \bigg( x_1 - \frac{a}{2} \bigg)^2, \frac{a}{b} \bigg( x_2 - \frac{b}{2} \bigg)^2 \bigg\},
\end{equation*}
where we assume that $R=[0,a] \times [0,b]$, so that the vector field is given by
\begin{equation*}
r(x) = \nabla^\perp V = \begin{cases}
\big( 0, \frac{2b}{a} \big( x_1 - \frac{a}{2} \big) \big) & |x_1| \geq \frac{b}{a} |x_2|, 0 \leq x_1 \leq a, \\
\big( -\frac{2a}{b} \big( x_2 - \frac{b}{2} \big),0 \big) & |x_1| < \frac{b}{a} |x_2|, 0 \leq x_2 \leq b.
\end{cases}
\end{equation*}
Hence by elementary computations
\begin{equation*}
\|D^\cont r\| = a^2 + b^2, \quad \|D^{\mathrm{jump}} r\| = 3 a^2 + 3 b^2,
\end{equation*}
and then we conclude.
    \end{proof}

\section{Ergodic Theory}
\label{S:ergodic}

We will consider flows of divergence-free vector fields from the point of view of Ergodic Theory. Even if we apply the results to the case $(K,\mathcal{B}(K),\L^2\llcorner_K)$ in this section we will give the notions of ergodicity and mixing in more general spaces \cite[Chapter 1]{Ergodic:theory}. More precisely, let $(\Omega,\Sigma,\mu)$ be a locally compact separable metric space where $\mu$ is complete and normalized, that is $\mu(\Omega)=1$. 
\begin{definition}
    An \emph{automorphism} of the measure space $(\Omega,\Sigma,\mu)$ is a one-to-one map $T:\Omega\rightarrow\Omega$ bi-measurable and measure-preserving, that is
    \begin{equation*}
        \mu(A)=\mu(T(A))=\mu(T^{-1}(A)), \quad \forall A\in\Sigma.
    \end{equation*}
\end{definition}

\noindent
We call $G(\Omega)$ the group of automorphisms of the measure space $(\Omega,\Sigma,\mu)$. 
\begin{definition}
A \emph{flow} $\lbrace X_t\rbrace$, $t\in\R$, is a one-parameter group of automorphisms of $(\Omega,\Sigma,\mu)$ such that for every $f:\Omega\rightarrow \R$ measurable, the function $f(X_t(x)) $ is measurable on $\Omega\times\R$.
\end{definition}
\begin{definition}
    Let $T:\Omega\rightarrow\Omega$ an automorphism. Then 
    \begin{itemize}
        \item $T$ is \textbf{\emph{ergodic}} if for every $A\in\Sigma$
        \begin{equation}
        \label{ergodicity}
            T(A)=A \quad \Rightarrow \quad \mu(A)=0 \textrm{ or } \mu(A)=1;
        \end{equation}
        \item $T$ is \textbf{\emph{weakly mixing}} if $\forall A,B\in\Sigma$
        \begin{equation}
        \label{weak:mix}
            \lim_{n\to\infty}\frac{1}{n}\sum_{j=0}^{n-1}\left[\mu(T^{-j}(A)\cap B)-\mu(A)\mu(B)\right]^2=0;
        \end{equation}
        \item $T$ is \textbf{\emph{(strongly) mixing}} if $\forall A,B\in\Sigma$
        \begin{equation}
        \label{mix}
            \lim_{n\to\infty}\mu(T^{-n}(A)\cap B)=\mu(A)\mu(B).
        \end{equation}
    \end{itemize}
    \end{definition}
    \begin{remark}
        It is a well-known and quite elementary fact that strongly mixing $\Rightarrow$ weakly mixing $\Rightarrow$ ergodic.
    \end{remark}
    
    \noindent We can give the analogous definitions for the flow:
    \begin{definition}
    Let $\lbrace X_t\rbrace$ a flow of automorphisms. Then 
    \begin{itemize}
        \item $\lbrace X_t\rbrace$ is \textbf{\emph{ergodic}} if for every $A\in\Sigma$
        \begin{equation}
        \label{f:ergodicity}
            X_t(A)=A \quad \Rightarrow \quad \mu(A)=0 \textrm{ or } \mu(A)=1;
        \end{equation}
        \item $\lbrace X_t\rbrace$ is \textbf{\emph{weakly mixing}} if $\forall A,B\in\Sigma$
        \begin{equation}
        \label{f:weak:mix}
            \lim_{t\to\infty}\dashint_0^t\left[\int_{\Omega}\chi_A(X_{-s}(x))\chi_B(x)d\mu-\mu(A)\mu(B)\right]^2ds=0;
        \end{equation}
        \item $\lbrace X_t\rbrace$ is \textbf{\emph{(strongly) mixing}} if $\forall A,B\in\Sigma$
        \begin{equation}
        \label{f:mix}
           \lim_{t\to\infty} \int_{\Omega}\chi_A(X_{-t}(x))\chi_B(x)d\mu=\mu(A)\mu(B).
        \end{equation}
    \end{itemize}
    \end{definition}

\subsection{The neighbourhood topology as a convergence in measure.}\label{topology}
  To get a genericity result it is necessary to identify the correct topology on $G(\Omega)$. Following the work of Halmos \cite{Halmos:weak:mix} we define the \emph{neighbourhood topology} as the topology generated by the following base of open sets: let $T\in G(\Omega)$ then 
    \begin{equation*}
    N(T)=\lbrace S\in G(\Omega): |T(A_i)\triangle S(A_i)|<\epsilon, \quad i=1,\dots, n\rbrace,
    \end{equation*}
    where $\epsilon>0$ and $A_i\in\Sigma$ are measurable sets.
    
    Since for our purposes we will consider the $L^1$ topology on $G(\Omega)$, we recall the following
   \begin{proposition}
   \label{prop:1}
   	Let $\lbrace{T_n}\rbrace,T\subset G(\Omega)$ and assume that $T_n\rightarrow T$ in measure. Then $T_n\rightarrow T$ in the neighbourhood topology.
   	Conversely, if $T_n\rightarrow T$ in the neighbourhood topology, then $T_n$ converges to $T$ in measure.
   \end{proposition}

\noindent Since in our case $\Omega$ is a compact set, then the convergence in measure is equivalent to the convergence in $L^1$: hence we will use the $L^1$ topology for maps as in Proposition \ref{prop:gen}.

\smallskip

\noindent 
We will be concerned with flows of vector fields extended periodically to the real line, that is $b(t+1)=b(t)$.
 Even if $X_t$ is not a flow of automoprhisms, the quantities in the r.h.s. of \eqref{f:weak:mix},\eqref{f:mix} can be computed and are related to the mixing properties of $T = X_{t=1}$. Also the ergodic properties of $T = X_1$ are equivalent to an ergodic property of $X_t$.
%

Let $\lbrace X_s\rbrace_{s\in [0,1]}$ be a family of autmorphisms of $\Omega$ such that $s\rightarrow X_s$ is continuous (hence uniformly continuous) with respect to the neighborhood topology of $G(\Omega)$. Let $T = X_{t=1}$ and define
\begin{equation*}
    X_{t} = X_s \circ T^n = T^n \circ X_s, \quad t = n + s, s \in [0,1).
\end{equation*}
  
 \begin{lemma}
  \label{map:to:flow}
  The following holds
  \begin{enumerate}
      \item if $T$ is ergodic then for every set $A \in \Sigma$
      \begin{equation*}
          \dashint_0^t \chi_{X_s(A)} ds \to_{L^1} |A|;
      \end{equation*}
      \item $T$ is weakly mixing iff for every $A,B\in \Sigma$ 
      \begin{equation*}
            \lim_{t\to\infty}\dashint_0^t \big[  |X_{s}(A)\cap B| - |A||B| \big]^2 ds=0;
        \end{equation*}
        \item $T$ is mixing iff $\forall A,B\in\Sigma$
        \begin{equation*}
           \lim_{t\to\infty} \big|  X_{t}(A)\cap B \big| = |A||B|.
        \end{equation*}
  \end{enumerate}
  \end{lemma}

The proof of this lemma is given in Appendix \ref{S:appendix}, since we believe it is standard and not strictly related to our results.

    \begin{definition}
      Let $b\in L^\infty([0,1],\BV(\R^2))$, $\supp b_t \subset K$, be a divergence-free vector field. We will say that $b$ is ergodic (weakly mixing, strongly mixing) if its unique RLF $X_t$ evaluated at $t=1$ is ergodic (respectively weakly mixing, strongly mixing). 
  \end{definition}  
  
  \subsection{Genericity of weakly mixing}
\label{S:gen_weak_mix}

Let $\mathcal{U}$ be the $G_\delta$-subset of $L^1_{t,x}$ where the Regular Lagrangian Flow can be uniquely extended by continuity (Proposition \ref{prop:gen}). The first statement has the same proof of [Theorem 2, \cite{Halmos:weak:mix}] and [Page 77,\cite{Halmos:lectures}]:

\begin{proposition}
\label{Prop:weak_mix_G_delta}
The set of ergodic/weakly-mixing vector fields is a $G_\delta$-set in $\mathcal{U}$, the set of strongly mixing is a first category set.
\end{proposition}

\noindent We repeat the proof for convenience only for weakly/strongly mixing, the case for ergodic vector fields is completely analogous \cite{Halmos:ergodic}.

\begin{proof}
Since the map $\tilde{\Phi}(b)(t=1) = T(b)$ defined in Proposition \ref{prop:gen} is continuous from $\mathcal{U}$ into $L^1(K,K)$, it is enough to prove that the set of weakly mixing maps is a $G_\delta$. For simplicity we define a new topology on $G(K)$ that coincides with the neighbourhood topology known as \emph{Von Neumann strong neighbourhood topology}. So take $T\in G(K)$ then it defines a linear operator $T:L^2(K,\C)\rightarrow L^2(K,\C)$ 
\begin{equation*}
    (Tf)(x)=f(Tx) \quad\forall f\in L^2(K,\C)
\end{equation*}
such that $||Tf||_{L^2}=||f||_{L^2}$ (being $T$ measure-preserving). Consider $f_i$ a countable dense subset in $L^2$: a base of open sets in the strong neighbourhood topology is given by
\begin{equation*}
    N(T)=\lbrace S\in G(K): ||Tf_{i}-Sf_{i}||_2\leq\epsilon,\quad i=1,\dots,n\rbrace.
\end{equation*}
Then we define
\begin{equation*}
    E(i,j,m,n) = \big\{ T \in G(K): \big| (T^{n}f_i,f_j)-(f_i,1)(1,f_j) \big| < 2^{-m} \big\},
\end{equation*}
where $(\cdot,\cdot)$ denotes the scalar product in $L^2$. Simply observing that $T\rightarrow (Tf,g)$ is continuous in the strong neighbourhood topology then by Proposition \ref{prop:1} it follows that $E(i,j,n,m)$ is open in $L^1(K,K)$, and then
\begin{equation*}
    G = \bigcap_{i,j,m} \bigcup_n E(i,j,m,n)
\end{equation*}
is a $G_\delta$-set. By the Mixing Theorem [Theorem 2, page 29, \cite{Ergodic:theory}] $G$ coincides with the set of weakly mixing maps in $L^1(K,K)$.  Indeed if $T$ is not mixing, then there exists $f\not =0$ and a complex eigenvalue $\lambda \in \{|z| = 1, z \not= 0,1\}$ such that $Tf=\lambda f$. We can assume that $f$ is orthogonal to the eigenvector $1$, that is $(f,1)=0$, and also that $||f||_{2}=1$. Now choose $i$ such that $||f-f_{i}||\leq \epsilon$ for some $\epsilon$ to be chosen later and take $f_j=f_i$. Then
\begin{align*}
1 &= |(T^{n}f,f)-(f,1)(f,1)| \\
&\leq |(T^{n}f,f)-(T^{n}f,f_i)|+|(T^{n}f,f_i)-(T^{n}f_i,f_i)| + |(T^{n}f_i,f_i)- (f_i,1)(1,f_i)|\\
& \quad + |(f_i,1)(1,f_i)-(f,1)(f_i,1)|+  | (f,1)(1,f_i)- (f,1)(f,1)| \\ 
& \leq 2||f-f_i||_2+2||f_{i}||_2||f-f_i||_2 +|(T^{n}f_i,f_i)- (f_i,1)(1,f_i)|, 
\end{align*}
so since $||f_i||_2\leq 1 +\epsilon$ we get that
\begin{equation*}
    1\leq 2\epsilon + 2(1+\epsilon)\epsilon +|(T^{n}f_i,f_i)- (f_i,1)(1,f_i)|.
\end{equation*}
With the choice of $\epsilon>0$ small enough we get that $\frac{1}{2}\leq |(T^{n}f_i,f_i)- (f_i,1)(1,f_i)|$, that is $T\not\in G$. This concludes the proof of the first part of the statement.

\noindent We next prove that the set of strongly mixing vector fields is a first category set. Let $A\subset K$ be a measurable set such that $|A|=\frac{1}{2}$. Then define the $F_\sigma$-set
\begin{equation*}
    F=\bigcup_n \bigcap_{k>n}\left\lbrace T\in G(K): \left| |(T^{-k}(A)\cap A)|-\frac{1}{4}\right|\leq \frac{1}{5}\right\rbrace.
\end{equation*}
Clearly strongly mixing maps are contained in $F$ by definition and therefore strongly mixing vector fields are contained in $\tilde{F}=\tilde{\Phi}^{-1}(t=1)(F)$. This $\tilde{F}$ is a set of first category: indeed consider the set
\begin{equation}
\label{good:set}
\bigcup_{k>n} \tilde{\Phi}^{-1}(t=1)\left(\left\lbrace T\in G(K): \left| |(T^{-k}(A)\cap A)|-\frac{1}{4}\right|\leq \frac{1}{5}\right\rbrace\right)^c\comme{.}
\end{equation}
By our main result (Theorem \ref{Prop:piececlo}) $\forall b\in \mathcal{U}$ for all $n\in\N$ there exists $k>n$ and $b^p\in L^\infty_t(\BV_x)$ such that the RLF $X^p_{t=1}$ associated with $b^p$ evaluated at $t=1$ is a permutation of subsquares of period $k$. Hence
$$\bigcup_{k>n}\lbrace b\in L^\infty_t(\BV_x) \text{ permutation of period } k \rbrace$$
is dense and contained in \eqref{good:set}, so that we conclude that \eqref{good:set} is open and dense for all $n$, i.e. $F$ is of first category.
\end{proof}

\begin{corollary}
\label{Cor:strong_mix_dense}
Assume that the set
\begin{equation*}
    SM = \big\{ b \in \mathcal{U} : b \ \text{is strongly mixing} \big\}
\end{equation*}
is dense in $\mathcal{U}$. Then the set of weakly mixing vector fields is residual.
\end{corollary}

\begin{proof}
Elementary.
\end{proof}

Our aim will be to prove the assumption of the above corollary, which together with Proposition \ref{Prop:weak_mix_G_delta} will conclude the proof of Theorem \ref{Theo:main_intro} once we show that the dense set of strongly mixing vector fields are actually exponentially mixing.

\begin{remark}
    The above situation, namely
    \begin{itemize}
        \item $b$ strongly mixing is dense in $\mathcal{U}$,
        \item $b$ weakly mixing is second category in $\mathcal{U}$,
    \end{itemize}
is in some sense the best situation we can hope in $\mathcal{U}$. Indeed, the strongly mixing vector fields are a set of first category and then it is not a "fat" set. On the other hand, the weakly mixing vector fields would be a "fat" set once we know their density, which one deduces from the density of the strongly mixing vector fields.
\end{remark}

\subsection{Markov Shifts}
\label{subS:markov}

When dealing with finite spaces $X=\lbrace 1,\dots,n\rbrace$ and processes whose outcome at time $k$ depends only on their outcome at time $k-1$ it is easier to determine some statistical properties of the dynamical system, as ergodicity and mixing (see for a reference \cite{Mane},\cite{Viana}). More precisely let $B(n)=\lbrace\theta:\Z\rightarrow X\rbrace$ the space of sequences and define a cylinder
$$
C(m,k_1,\dots,k_r)=\Big\lbrace \theta\in B(n): \theta(m+i)=k_{i+1},\ i=0,\dots,r-1\Big\rbrace
$$
where $m\in\Z$ and $k_i\in X$. Therefore, since the Borel $\sigma$-algebra on $B(n)$ is generated by disjoint union of cylinders, we can define a probability measure $\mu$ on $B(n)$ simply determining its value on cylinders. A \emph{Markov measure} $\mu$ is a probability measure on $B(n)$ for which there exist $p_i > 0, P_{ij}\geq 0$, $i,j=1,\dots n$, with
\begin{equation*}
\sum_i p_i = \sum_j P_{ij} = 1, \quad \sum_i p_i P_{ij} = p_j,
\end{equation*}
such that
  \begin{equation*}
      \mu(C(m,k_1,\dots,k_r))= p_{k_1} P_{k_1k_2} \dots P_{k_{r-1}k_r}
  \end{equation*}
  for every cylinder $C(m,k_1,\dots,k_r)$. The $P_{ij}$ are called \emph{transition probabilities} and $P= (P_{ij})$ is the \emph{transition matrix}. The transition matrix is a stochastic matrix, that is $\sum_{j}P_{ij}=1$ for every $i$. 
  Now define $P^{(m)}_{ij}$ the coefficients of the matrix $P^m$.
  \begin{definition}
  A matrix $P$ with positive coefficients is \emph{irreducible} if $\forall i,j$ there exists $m$ such that $P^{(m)}_{ij}>0$.  \end{definition}
  \begin{definition}
 \label{aperiodic}
  A matrix $P$ with positive coefficients is \emph{aperiodic} if  there exists $m$ such that $P^{(m)}_{ij}>0$ $\forall i,j$.  \end{definition}
  
  \noindent A \emph{Markov shift} is a map $\sigma:(B(n),\mu)\rightarrow (B(n),\mu)$ such that
  \begin{equation*}
      \sigma(\theta)(i)=\theta(i+1),\quad\forall\theta\in B(n). 
  \end{equation*}
  Then it can be proved that $\sigma_\sharp\mu=\mu$.  We conclude this subsection with the following results on ergodicity and mixing properties of Markov shifts.
  \begin{proposition}[Ergodicity]
      The following are equivalent:
      \begin{enumerate}
          \item $\sigma:(B(n),\mu)\rightarrow (B(n),\mu)$ is ergodic;
          \item $P$ is irreducible;
          \item $\displaystyle{\lim_{m\to\infty}\frac{1}{m}\sum_{k=0}^{m-1} P^{(k)}_{ij}=p_j}$.
      \end{enumerate}
  \end{proposition}
  \begin{proposition}[Mixing]
      \label{markov:mix}
      The following are equivalent:
      \begin{enumerate}
          \item $\sigma:(B(n),\mu)\rightarrow (B(n),\mu)$ is strongly mixing;
          \item $P$ is aperiodic;
          \item $\displaystyle{\lim_{m\to\infty} P^{(m)}_{ij}=p_j}$.
      \end{enumerate}
  \end{proposition}

%% file: cyclic.tex
\section{Density of Strongly Mixing vector fields}\label{S:density}

In Section \ref{S:permutation} we prove that \emph{permutation vector fields} (i.e. vector fields whose RLF $X_t$, when evaluated at $t=1$, is a permutation of subsquares, Point (3) of Theorem \ref{Prop:piececlo}) are dense in $\mathcal{U}$. In this section, we show
that each permutation vector field can be approximated by a vector field whose RLF evaluated at time $t=1$ is a unique cycle. This approximation result will be used to get first the density of ergodic vector fields, then the density of strongly mixing vector fields.

\subsection{Cyclic permutations of squares}

\noindent
We start by recalling some basic facts about permutations. Denote by $S_n$ the set of permutations of the elements $\lbrace 1,\dots,n\rbrace$.

\begin{definition}
	Let $\sigma \in S_{n}$ be a permutation and $k\leq n\in\N$. We say that $\sigma$ is a \emph{$k$-cycle} $c$ (or simply a \emph{cycle}) if there exist $k$ distinct elements $a_1,\dots,a_k\in\lbrace 1,\dots,n\rbrace$ such that
	\begin{equation*}
		\sigma(a_i)=a_{i+1}, \quad \sigma(a_k)=a_1, \quad \sigma(x)=x \quad\forall x\not= a_1,\dots,a_k.
	\end{equation*}
	We identify the permutation with the ordered set $c=(a_1a_2\dots a_k)$. The number $k$ is the \emph{length} of the cycle. We say that $c$ is \emph{cyclic} if $k=n$.  We call \emph{transpositions} the $2$-cycles.
\end{definition}

\begin{definition}
	Let $c_1,c_2$ be the cycles $c_1=(a_1\dots a_t)$ and $c_2=(b_1\dots b_s)$. We say that $c_1,c_2$ are \emph{disjoint} cycles if $a_i\not=b_j$ for every $i=1,\dots, t$, $j=1,\dots, s$.
\end{definition}

\noindent Recall the following result.

\begin{theorem}
	\label{prod:cycles}
	Every permutation $\sigma \in S_n$ is the product of disjoint cycles.
\end{theorem}

\noindent From now on we will address flows $X_t$ of divergence-free vector fields such that $X_{t=1}$ is a permutation of squares of size $\frac{1}{D}$. 
\\

\noindent Let us fix the size $D\in \N$ of the grid in the unit square $K$.  We enumerate the $D^2$ subsquares of the grid and we consider $S_{D^2}$ the set of the permutations of  $\lbrace \kappa_1,\dots, \kappa_{D^2}\rbrace$. We say that two squares (ore more in general two rectangles) are \emph{adjacent} if they have a common side. We will use also the word \emph{adjacent} for cycles: two disjoint cycles  of squares $c_1,c_2$ are adjacent 
if there exist $\kappa_1\in c_1$, $\kappa_2\in c_2$ adjacent subsquares. Two \emph{adjacent} squares can be \emph{connected} by a transposition, which can be defined simply as an exchange between the two squares: let $\kappa_i,\kappa_j$ two adjacent squares of size $\frac{1}{D}$ and let $R=\kappa_i\cup \kappa_j$, then the \emph{transposition flow} between $\kappa_i,\kappa_j$ is $T_t(\kappa_i,\kappa_j):[0,1]\times K\rightarrow K$ defined as
\begin{equation}
	\label{transposition flow}
		T_t(\kappa_i,\kappa_j)=
	\begin{cases}
	\chi^{-1}\circ r_{4t}\circ \chi & x\in \mathring{R}, \ t\in\left[0,\frac{1}{2}\right], \\
		\chi_i^{-1}\circ r_{4t}\circ \chi_i & x\in\mathring{\kappa}_i, \ t\in\left[\frac{1}{2},1\right], \\
		\chi_j^{-1}\circ r_{4t}\circ \chi_j & x\in\mathring{\kappa}_j, \ t\in\left[\frac{1}{2},1\right], \\
		x & \text{otherwise},
	\end{cases}
\end{equation}
where the map $\chi:R\rightarrow K$ is the affine map sending the rectangle $R$ into the unit square $K$, $\chi_{i},\chi_{j}$ are the affine maps sending $\kappa_i,\kappa_j$ into the unit square $K$ and $r$ is the rotation flow (\ref{rot:flow:square}). This invertible measure-preserving flow has the property to exchange the two subsquares in the unit time interval (Figure \ref{fig:transp}).
\begin{figure}
    \centering
    \includegraphics[scale=0.5]{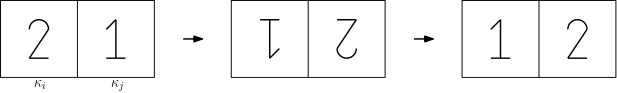}
    \caption{The action of the \emph{transposition flow} $T_t$.}
    \label{fig:transp}
\end{figure}
Moreover, by the computations done in  Lemma \ref{TV:rotations}, we can say that
\begin{equation}
\label{TV:trasp}
	\TV(\dot{T}_t(\kappa_i,\kappa_j))(\bar{R})\leq \frac{20}{D^2}.
\end{equation}

\begin{lemma}
\label{lem:cycl:1}
	Let $b\in L^\infty_t(\BV_x)$ be a divergence-free vector field and assume that its flow at time $t=1$, namely $X\llcorner_{t=1}$, is a $k$-cycle of squares of the grid $\N\times \N\frac{1}{D}$ where $k,D\in\N$. Then for every $M=2^p$, there exists $b^c\in L^\infty_t(\BV_x)$ divergence-free vector field  such that
	\begin{equation*}
		||b-b^c||_{L^\infty(L^1)}\leq \mathcal{O}\left(\frac{1}{D^3M}\right),
\end{equation*}
\begin{equation*}
||\TV(b^c - b)(K)||_{\infty} \leq \mathcal{O}\left(\frac{1}{D^2}\right),
	\end{equation*}
	and the map $X^c_{t=1}:K\rightarrow K$ is a $kM^2$-cycle of squares of size $\frac{1}{DM}$, where $X^c_t:[0,1]\times K\rightarrow K$ is the flow associated with $b^c$.
\end{lemma}

\noindent Here and in the following we will write $T(\kappa_i) = \kappa_j$ meaning that $T$ is a rigid translation of $\kappa_i$ to $\kappa_j$. This to avoid cumbersome notation.

\begin{proof}
	Let us call $T\doteq X_{t=1}$: being a cycle, there exist $\lbrace \kappa_1,\dots,\kappa_k\rbrace\subset\lbrace 1,\dots,D^2\rbrace$  such that
	\begin{equation*}
		T(\kappa_{i})=\kappa_{i+1},\quad T(\kappa_{n})=\kappa_1, \quad T(x)=x \quad \text{otherwise}.
	\end{equation*}
	Now fix some $M=2^p$ and divide each subsquare $\kappa_i$ into $M^2$ subsquares $\kappa^j_i$ with $j=1,\dots, M^2$. Since $T$ is a translation of subsquares and choosing cleverly the labelling $j \to \kappa^j_i$, then we have also  $T(\kappa^j_i)=\kappa^{j}_{i+1}$ so that $T$ is a permutation of subsquares $\kappa^j_i$. 
	More precisely, it is the product of $M^2$ disjoint cycles of length $k$. The idea is to connect these cycles with transpositions in order to have a unique cycle of length $kM^2$: we will need a parturbation inside $\kappa_1$. \\
	
	\noindent
	 Divide the $M^2$ subsquares of $\kappa_1$ into $\frac{M^2}{2}$ couples $R^2_h=\kappa^j_1\cup\kappa^{j'}_1$ with $h=1,\dots,\frac{M^2}{2}$ and $\kappa^j_1,\kappa^{j'}_1$ are adjacent squares.  In the time interval $\left[0,\frac{1}{2}\right]$ perform $\frac{M^2}{2}$ transpositions, one in each $R^2_{h}$, that is
	\begin{equation*}
		X^c_t(x)= X_t\circ T^2_t (x), \quad t\in\left[0,\frac{1}{2}\right],
	\end{equation*}
     where the flow $T^2:\left[0,\frac{1}{2}\right]\times K\rightarrow K$ 
$$
T^2_{t}\llcorner_ {R^2_h}\doteq T_{2t}(\kappa^j_1,\kappa^{j'}_1) \quad \text{and} \quad T^2_t(x)=x \ \text{otherwise},
$$
is the transposition flow (\ref{transposition flow}) between $\kappa_1^j$ and $\kappa_1^{j'}$ as defined in \eqref{transposition flow} above.
	Then for $t\in\left[0,\frac{1}{2}\right]$ fixed,
	\begin{align*}
		\TV(b^{c}_t-b_t)({\kappa}_1) \leq \mathcal O(1) 2 \frac{M^2}{2}\frac{20}{M^2D^2},
	\end{align*}
	where we have used \eqref{TV:trasp}. We observe that at this time step we have obtained $\frac{M^2}{2}$ disjoint $2k$-cycles. \\
	
	\noindent In the time interval  $\left[\frac{1}{2},\frac{3}{4}\right]$ we divide the unit square into squares $R^4_h=R^2_{j}\cup R^2_{j'}$ with $h=1,\dots, \frac{M^2}{4}$ where $R^2_j, R^2_{j'}$ are adjacent (in particular there exist $\kappa_1^{j}\subset R^2_j,\kappa_1^{j'}\subset R^2_{j'} $ adjacent squares). Now we perform  $\frac{M^2}{4}$ transpositions of squares connecting the two rectangles $R^2_j, R^2_{j'}$ as in Figure \ref{fig:cyclic:lem:1}. More precisely we define  for $t\in \left[\frac{1}{2},\frac{3}{4}\right] $
	\begin{equation*}
		X^c_t(x)=X_t\circ T^4_t (x), \quad t\in\left[\frac{1}{2},\frac{3}{4}\right],
	\end{equation*}
	where the flow $T^4:\left[\frac{1}{2},\frac{3}{4}\right]\times K\rightarrow K$ 
$$
T^4_{t}\llcorner_ {R^4_h}\doteq T_{4t-2}(\kappa^j_1,\kappa^{j'}_1) \quad \text{and} \quad T^4_t(x)=x \ \text{otherwise},
$$
is the transposition flow (\ref{transposition flow}) between $\kappa_1^j$ and $\kappa_1^{j'}$. Again,
	\begin{align*}
		\TV(b^c_t-b_t)({\kappa}_1) \leq \mathcal O(1) 4 \frac{M^2}{4}\frac{20}{M^2D^2}.
	\end{align*}
	\begin{figure}
	    \centering
	    \includegraphics[scale=0.55]{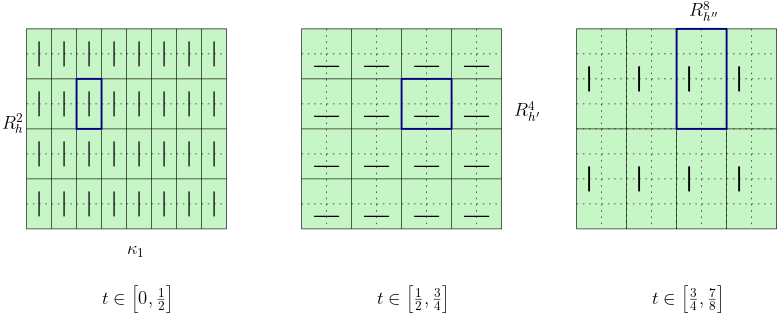}
	    \caption{Subdivision of the initial square $\kappa_1$ into subrectangles/subsquares where transpositions (the bars) occurr between subsquares $\kappa^j_i\subset R^{2i}_j,\kappa^{j'}_i\subset R^{2i}_{j'}$ of side $\frac{1}{DM}$ (dotted lines). Notice that at the first/third step the initial square $\kappa_1$ is divided into rectangles (see Point \eqref{proc:1} of the procedure), while at the second step it is divided into squares (see Point \eqref{proc:2}).}
	    \label{fig:cyclic:lem:1}
	\end{figure}
	
	\noindent Repeating the procedure (see Figure \ref{fig:cyclic:lem:1}),
	\begin{enumerate}
	    \item \label{proc:1} at the $2i-1$-th step we divide our initial square $\kappa_1$ into $\frac{M^2}{2^{2i-1}}$ rectangles (made of two squares of obtained at the step $2(i-1)$) so that we perform ${2^{2p-i}}$ transpositions of subsquares $\kappa_1^j$ in the time interval $\left[\sum_{j=1}^{2i-2} \frac{1}{2^j},\sum_{j=1}^{2i-1}\frac{1}{2^j}\right]$;
	    \item \label{proc:2} at the $2i$-th step, we divide our initial square $\kappa_1$ into $\frac{M^2}{2^{2i}}$ squares (made of 2 rectangles of the previous step) so that we perform ${2^{2p-i}}$ transpositions of subsquares $\kappa_1^j$ in the time interval $\left[\sum_{j=1}^{2i-1} \frac{1}{2^j},\sum_{j=1}^{2i}\frac{1}{2^j}\right]$.
	\end{enumerate}
	In both cases we find in the interval $\left[\sum_{j=1}^{i-1} \frac{1}{2^j},\sum_{j=1}^{i}\frac{1}{2^j}\right]$ that
	\begin{equation*}
		\TV(b^{c}_t-b_t)({\kappa}_1) \leq \mathcal O(1) 2^i\frac{M^2}{2^i} \frac{20}{M^2D^2}.
	\end{equation*}
	Call  
    $t_i=\sum_{j=1}^i 2^{-j}$. We will prove that the map $X(1,t_i) \circ X^c\llcorner_{t=t_i}$ is a permutation given by the product of $\frac{M^2}{2^i}$ disjoint $2^ik$-cycles simply by induction on $i$.  \\
    
    \noindent The case $i=1$ is immediate from the definition. So let us assume that the property is valid for $i$ and call $c_1,c_2,\dots,c_{\frac{M^2}{2^i}}$ the disjoint $2^ik$-cycles made of rectangles of subsquares as in Figure \ref{fig:cyclic:lem:1}, where we have ordered them in such a way that $c_{2h-1},c_{2h}$ with $h=1,\dots,\frac{M^2}{2^{i+1}}$ are adjacent along the long side.
    Then fix a couple of adjacent cycles, for simplicity $c_1,c_2$. Then
    \begin{equation*}
    \begin{split}
        & c_1=(\kappa^1_1\quad\dots\quad\kappa^1_{2^ik}),\\
        & c_2=( \kappa^2_1\quad\dots\quad\kappa^2_{2^ik}),
    \end{split}
     \end{equation*}
     and assume that there exist $j,j'$ such that $\kappa^1_j,\kappa^2_{j'}$ are the adjacent subsquares in which we perform the transposition. By simply observing that
     \begin{equation*}
         X^c\llcorner_{t=t_{i+1}}(x)=X^c\llcorner_{t=t_i}(x)+\int_{t_i}^{t_{i+1}} b^c(s,X^c_s(x))ds,
     \end{equation*}
     we deduce that, when restricted to $c_1\cup c_2$, the map $X(1,t_{i+1}) \circ X^c\llcorner_{t_{i+1}}$ is the following permutation
     \begin{equation*}
         \left(\begin{array}{lcllcllcllcl}
                \kappa^1_1 & \dots & \kappa^1_{j-1} & \kappa^1_j & \dots & \kappa^1_{2^ik} & \kappa^2_1 & \dots & \kappa^2_{j'-1} & \kappa^2_{j'} & \dots & \kappa^2_{2^ik} \\ &&&&&&&&&&&\\
                \kappa^1_2 & \dots & \kappa^2_{j'} & \kappa^1_{j+1} & \dots & \kappa^1_1 & \kappa^2_2 & \dots & \kappa^1_j& \kappa^2_{j'+1} & \dots & \kappa^2_1
         \end{array}\right).
     \end{equation*}
Clearly this is a single cycle of length $2^{(i+1)k}$, and it is supported on a rectangle.

    The procedure stops at $t=\sum_{j=1}^{2p}\frac{1}{2^j}$ when we have obtained a unique $M^2k$-cycle. Summing up, for $t$ fixed
	\begin{equation*}
			\TV(b^{c}_t - b)(K) \leq \mathcal O(1) \frac{20}{D^2},
	\end{equation*}
    that is
	\begin{equation*}
		|| \TV(b^c_t - b_t)(K)||_\infty \leq \mathcal{O}\left(\frac{1}{D^2}\right).
	\end{equation*}
	We conclude with the $L^\infty_t L^1_x$ estimate of the vector field: to do this computation it is necessary to observe that $b_t$ and $b^{c}_t$ differ only in the couples of adjacent squares in which we perform the transpositions. Using \eqref{Equa:comp_ff} and simple estimates on the rotation \eqref{rotation field} we obtain, for $t\in[t_{i-1},t_i]$ fixed,
	\begin{equation*}
	    \| b^c_t-b_t \|_1\leq \mathcal O(1) \frac{M^2}{2^i}2^i\frac{2}{DM}\frac{2}{D^2M^2}\leq\mathcal{O}\left(\frac{1}{D^3M}\right),
	\end{equation*}
	which concludes the proof. 
\end{proof}

\noindent We state now the approximation result by vector fields whose flow at time $t=1$ is a unique cycle.
\begin{proposition}
\label{unique:cycle}
	Let $b\in L^\infty_t(\BV_x)$ be a divergence-free vector field and assume that $b_t = 0$ for $t \in [0,\delta]$, $\delta > 0$, and its flow at time $t=1$, namely $X\llcorner_{t=1}$, is a permutation of squares of the grid $\N\times \N\frac{1}{D}$ where $D\in\N$. Then for every $M=2^p \gg 1$ there exists a divergence-free vector field $b^c\in L^\infty_t(\BV_x)$ such that
	\begin{equation*}
		||b-b^c||_{L^1(L^1)}\leq \mathcal{O}\left(\frac{1}{DM^3}\right), \quad ||\TV(b^c_t - b_t)(K)||_{\infty}\leq \mathcal{O}\left(\frac{1}{\delta M^2}\right) 
	\end{equation*}
	and the map $X^c_{1}:K\rightarrow K$, being $X^c_t:[0,1]\times K\rightarrow K$ is the flow associated with $b^c$, is a $M^2D^2$-cycle of subsquares of size $\frac{1}{DM}$.
\end{proposition}

\begin{proof}[Proof of Proposition \ref{unique:cycle}.]
	Let us fix $\epsilon>0$ and consider $M=2^p$ to be chosen later. Let $C\doteq X\llcorner_{t=1}$ be a permutation, which we write by Theorem (\ref{prod:cycles})
	\begin{equation*}
		C=(\kappa^1_1\dots \kappa^1_{k_1})(\kappa^2_1\dots \kappa^2_{k_2})\dots (\kappa^n_1\dots \kappa^n_{k_n})=c_1\dots c_n,
	\end{equation*}
	where $\sum_{i=1}^n k_i\leq D^2$. Define $c_{n+1},\dots c_{N}$, $N=D^2-\sum_i k_i +n$, the $1$-cycles representing the subsquares that are sent into themselves. By the previous lemma we can also assume that $C\llcorner_{c_i}$ $i=1,\dots,N$ is a cyclic permutation of subsquares $a^i_{jk}$, $j = 1,\dots,M^2$, of the grid $\N\times\N\frac{1}{MD}$. To find a $D^2M^2$ cycle we should consider all the couples of adjacent subsquares (of size $\frac{1}{MD}$), and then we should connect them by transpositions in a precise way. \\
	
	\noindent
	Fix $c_1$ and consider 
	\begin{equation*}
	    C^1=\lbrace c_h\not= c_1 \ \text{ s.t. } c_h \text{ adjacent to } c_1\rbrace=\lbrace c^1_1,\dots,c^1_{|C^1|}\rbrace.
	\end{equation*}
	Now for every $c^1_j\in C^1$ define by induction the disjoint families of cycles
	\begin{equation*}
	    C^2_j=\lbrace c_h\notin \lbrace c_1\rbrace \cup C^1 \cup C^2_1 \cup \dots\cup C^2_{j-1} \text{ s.t. }  c_h  \text{ is adjacent to } c^1_j\rbrace,
	\end{equation*}
	and call
	\begin{equation*}
	    C^2= C^2_1\cup\dots\cup C^2_{|C^1|}=\lbrace c^2_1,\dots,c^2_{|C^2|}\rbrace.
	\end{equation*}
	At the $i-1$-th step we have 
	\begin{equation*}
	    C^{i-1}= \lbrace c^{i-1}_1,\dots,c^{i-1}_{|C^i|}\rbrace.
	\end{equation*}
	and, for every $c^{i-1}_j\in C^{i-1}$,
	\begin{equation*}
	    C^{i}_j=\lbrace c_h\notin \lbrace c_1\rbrace \cup C^1 \cup  C^2\cup \dots\cup C^{i-1}\cup C^i_1\cup\dots C^i_{j-1} \text{ s.t. }  c_h \text{ is adjacent to } c^{i-1}_j\rbrace.
	\end{equation*}
    The procedure ends when we have arranged all $c_i$ into sets $C^i$, and hence for some $K \in \N$ we obtain $C^{K+1}=\emptyset$ (see Figure \ref{Fig:cicli_concat}). Indeed, by contradiction assume that $|\lbrace c_1\rbrace \cup C^1\cup C^2\cup\dots \cup C^K|<N$. Then this set has a boundary, i.e. there exists a cycle $c\not\in \{c_1\rbrace \cup C^1\cup C^2\cup\dots \cup C^K$ adjacent to a cycle of $\{c_1\rbrace \cup C^1\cup C^2\cup\dots \cup C^K$, which is a contradiction by definition. 

    The partition
$$
C(c_1)=\lbrace c_1\rbrace \cup C^2\cup\dots C^K
$$
has the natural structure of a directed tree: indeed every two cycles $c_i \in C^i$, $c_j \in C^j$ are connected by a unique sequence of cycles: 
the direction of each edge is given by the construction $c^{i-1}_j \to c_h$ whenever $c_h \in C^i_j$. This tree-structure gives us a selection of the $N-1$ couples of subsquares $\kappa_i^j$ of disjoint cycles $c_i$ in which we can perform a transposition among the subsquares $a^i_{jk}$ to connect all of them in a unique $D^2M^2$-cycle. More precisely, for every connected couple $c^{i-1}_j,c_h$ such that $c_h \in C^i_j$, there exist cubes $\kappa \in c^{i-1}_j,\kappa' \in c_h$, and hence there are adjacent subsquares $a \subset \kappa,a' \subset \kappa'$ of size $1/(MD)$: assuming $M \geq 4$, we can take $a,a'$ not being on the corners of $\kappa,\kappa'$, respectively.

	Let $T_t:[0,\delta]\times K\rightarrow K$ be the transposition flow \eqref{transposition flow} acting in the selected $N-1$ couples of subsquares $a,a'$ reparametrized on the time interval $[0,\delta]$ and define (being $X_t = \Id$ for $t \in [0,\delta]$)
	\begin{equation}
	\label{cyclic:flow}
		X^c_t(x)=\begin{cases}
			T_t(x) & t\in [0,\delta], \\
			X_{t}\circ T_{\delta} (x) & t\in[\delta,1].
		\end{cases}
	\end{equation}
	The transposition is well defined: indeed it can happen that  $c_i,c_j,c_k$ are adjacent cycles and the couples of adjacent squares (of size $\frac{1}{D}$) are $\kappa_i,\kappa_j$ and $\kappa_i,\kappa_k $ (where $\kappa_i\in c_i,\kappa_j\in c_j$ and $\kappa_k\in c_k$), that is: $\kappa_i$ is in common.  But since the transposition occurs between subsquares of size $\frac{1}{DM}$ nor belonging to the corners, it is always guaranteed that the transpositions act on disjoint subsquares. 
	By using the explicit formula \eqref{rotation field} we get that for $t\in[0,\delta]$
	\begin{equation*}
	\begin{split}
	    ||b^c_t - b_t||_1 &\leq 
\frac{\mathcal{O}(1)}{\delta} \left( \frac{N-1}{D^3 M^3} \right) \leq 
\frac{\mathcal{O}(1)}{\delta} \left( \frac{1}{D M^3} \right),
	\end{split}
	\end{equation*}
	while for $t\in[\delta,1]$ it clearly holds $b^c_t = b_t$. 
	Coupling these last two estimates we get the $L^1_t L^1_x$ estimate:
	\begin{equation*}
	    ||b-b^c||_{L^1(L^1)} \leq 
	    \mathcal{O}\left(\frac{1}{DM^3}\right), 
	\end{equation*}
for $\delta<<1$ and $M$ sufficiently large.
 
Next we compute the total variation for $t\in[0,\delta]$: by using \eqref{TV:trasp}, we get
	\begin{align*}
		\TV(b^c_t)(K)&  \leq \frac{N-1}{\delta}  \frac{20}{M^2D^2}  \leq \frac{1}{\delta}\frac{20}{M^2},
	\end{align*}
	while for $t\in[\delta, 1-\delta]$ we find
	\begin{equation*}
	   \TV(b^c_t)(K)= \TV(b_t)(K), 
	\end{equation*}
	therefore
	\begin{equation*}
	    ||\TV(b^c_t - b_t)(K)||_{\infty}\leq \mathcal{O}\left(\frac{1}{\delta M^2}\right).
	\end{equation*}

	To conclude we have to prove that $X^c_1$ is a unique cycle, which follows by the tree-structure of the selection of adjacent cycles. The end points of the tree are clearly cycles. By recurrence, assume that $c^{i-1}_j$ is connected to cycles $\gamma_h$, each one made of all squares belonging to $c_h \in C^i_j$ and all subsequent cycles to $c_h$. It is fairly easy to see that the transposition merging $c^{i-1}_j$ to each $c_h \in C^i_j$ generates a unique cycles $\gamma^i_j$, made of the cubes of $c^{i-1}_j$ and all $\gamma_h$. We thus conclude that the map $X^c_1$ is a cycle of size $M^2 D^2$.
%
%
%
\end{proof}

\begin{figure}[t]
    \centering
    \def\svgwidth{.7\textwidth}
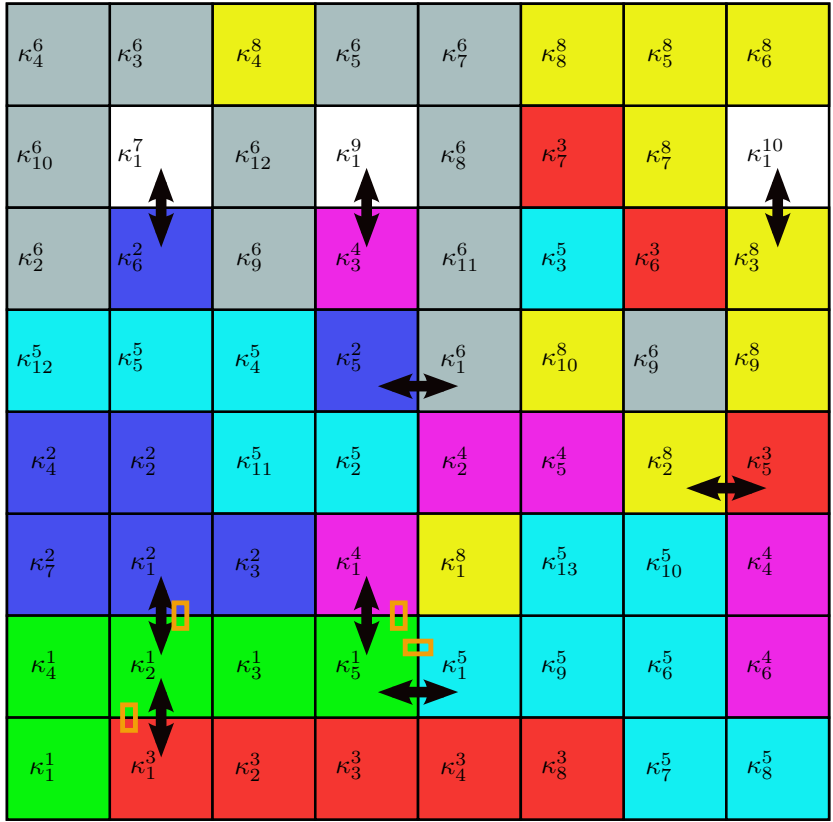
    \caption{Concatenation of cycles in a specific example, Remark \ref{Ex:cicliconc}. The orange subrectangles are the couples of $a^i_{jk}$ on which the transposition $T_t$ of \eqref{cyclic:flow} acts.}
\label{Fig:cicli_concat}
\end{figure}

\begin{remark}
\label{Ex:cicliconc}
An example of how the proof works is in Figure \ref{Fig:cicli_concat}: the decomposition in cycles is
    \begin{equation*}C = {\color{green}(\kappa_1^1\dots\kappa^1_5)}{\color{blue}(\kappa^2_1\dots\kappa^2_7)}{\color{red}(\kappa^3_1\dots\kappa^3_8)}{\color{magenta}(\kappa^4_1\dots\kappa^4_6)}{\color{cyan}(\kappa^5_1\dots\kappa^5_{13})}{\color{gray}(\kappa^6_1\dots\kappa^2_{12})}(\kappa^7_1){\color{yellow}(\kappa^8_1\dots\kappa^8_{10})}(\kappa^9_1)(\kappa^{10}_1).\end{equation*}
    The black arrow indicates the adjacent subsquares where the exchanges are performed: the tree of concatenation is then \\
\begin{center}
\begin{tikzpicture}
\node{${\color{green}(\kappa_1^1\dots\kappa^1_5)}$}
    child{node{${\color{blue}(\kappa^2_1\dots\kappa^2_7)}$}
        child{node{${\color{gray}(\kappa^6_1\dots\kappa^2_{12})}$}}
        child{node{$(\kappa^7_1)$}}
    }
    child{node[xshift= 1cm]{${\color{red}(\kappa^3_1\dots\kappa^3_8)}$}
        child{node{${\color{yellow}(\kappa^8_1\dots\kappa^8_{10})}$}}
    }
    child{node[xshift= 2cm]{${\color{magenta}(\kappa^4_1\dots\kappa^4_6)}$}
        child{node{$(\kappa^9_1)$}}
    }
    child{node[xshift= 3cm]{${\color{cyan}(\kappa^5_1\dots\kappa^5_{13})}$}
    };
\end{tikzpicture}
\end{center}
Note that in the subsquares $(\kappa^1_2,\kappa^2_1),(\kappa^1_2,\kappa^3_1)$ and $(\kappa^1_5,\kappa^4_1),(\kappa^1_5,\kappa^5_1)$ the exchange occurs actually in the subsquares $(a^1_{2j},a^2_{1\ell}),(a^1_{2j'},a^3_{1\ell'})$ and $(a^1_{5k},a^4_{1\ell''}),(a^1_{5k'},a^5_{1,\ell'''})$, so that it is always acting on different couples of subsquares.
\end{remark}

\begin{remark}
\label{Rem:better_mixing}
The construction of the cyclic flow \eqref{cyclic:flow} gives us only the $L^1_t L^1_x$ estimate on the vector fields, which is what we need for our genericity result. We can get the more refined estimate in $L^\infty_t L^1_x$ allowing for mass flowing (when performing the transposition) during the time evolution of the flow $X_t$ (see Figure \ref{trasp:flow}).
\begin{figure}
    \centering
    \includegraphics[scale=0.5]{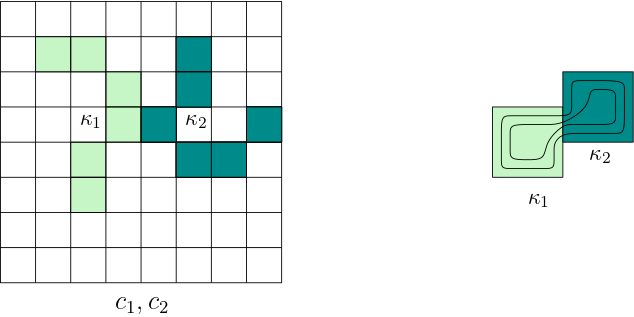}
    \caption{The two adjacent cycles $c_1$ (light green) and $c_2$ (blue) touch in $\kappa_1$ and $\kappa_2$, which exchange their mass during the time evolution.}
    \label{trasp:flow}
\end{figure}
In this case, the time spent by the squares of size $(MD)^{-1}$ to transfer the mass is of order $(MD)^{-1}$, so that the vector field moving it should be of the order
\begin{equation}
\label{Equa:O1_b_t}
\frac{\text{length}}{\text{time}} = \mathcal O(1), \quad \text{acting on a region of area} \ \frac{N-1}{(MD)^2} \leq M^{-2}.
\end{equation}
Hence the $L^\infty_t L^1_x$ estimate can be obtained by \eqref{Equa:comp_ff} as
\begin{equation*}
\|b^c_t - b_t\|_1 \leq \mathcal O(1) M^{-2},
\end{equation*}
while the total variation estimate becomes
\begin{equation*}
\TV(b^c_t-b_t) = \mathcal O(1) \frac{D}{M}.
\end{equation*}
The statement one can prove is then the following.
\begin{proposition}
\label{unique:cycle_star}
	Let $b\in L^\infty_t(\BV_x)$ be a divergence-free vector field and assume that its flow at time $t=1$, namely $X\llcorner_{t=1}$, is a permutation of squares of the grid $\N\times \N\frac{1}{D}$ where $D\in\N$. Then for every $\epsilon>0$ there exist $M=2^p$ and $b^c\in L^\infty_t(\BV_x)$ a divergence-free vector field such that
	\begin{equation*}
		||b^c_t-b_t||_{L^1} \leq \mathcal O(M^{-2}) \leq \epsilon, \quad \TV(b^c_t-b_t)(K) \leq \mathcal{O}(D/M),
	\end{equation*}
	and the map $X^c_{t=1}:K\rightarrow K$, being $X^c_t:[0,1]\times K\rightarrow K$ the flow associated with $b^c$, is a $M^2D^2$-cycle of subsquares of size $\frac{1}{DM}$.
\end{proposition}
\end{remark}

\subsection{Density of ergodic vector fields}\label{subS:ergodicity}

Starting from the cyclic permutation we have built in the previous section, we construct an ergodic vector field arbitrarily close to a given vector field in $L^\infty_t \BV_x$. The density of ergodic vector fields is not strictly relevant for the genericity result of weakly mixing vector fields, but it can be considered as a simple case study for the construction of strongly mixing vector fields. Moreover it will give a direct proof of Point \eqref{Point2:main_intro_ergo} of Theorem \ref{Theo:main_intro}.

We will use the \emph{universal mixer} that has been constructed in \cite{univ:mixer}: it is the time periodic divergence-free vector field $u\in L^\infty_t([0,1], BV_x(\R^2))$ whose flow $U_t:[0,1]\times K\rightarrow K$ of measure-preserving maps realizes at time $t=1$  the \emph{folded Baker's map}, that is
\begin{equation}
    \label{baker}
    U=U\llcorner_{t=1}=\begin{cases}
        \left(-2x+1,-\frac{y}{2}+\frac{1}{2}\right) & x\in\left[0,\frac{1}{2}\right), \\
         \left(2x-1,\frac{y}{2}+\frac{1}{2}\right) & x\in\left(\frac{1}{2},1\right],
    \end{cases} \qquad y \in [0,1],
\end{equation}
(see  Theorem $1$, \cite{univ:mixer}).

\begin{proposition}
\label{Prop:ergod_dense}
Let $b\in L^\infty_t(\BV_x)$ and let $X_t
$ be its RLF, and assume that $X_{t=1}$ is a cyclic permutation of squares of the grid $\N\times\N\frac{1}{D}$. Then there exists $b^e\in L^\infty_t(\BV_x)$ divergence-free ergodic vector field such that 
\begin{equation}\label{BV:ergodic}
\begin{split}
	&	||b-b^e||_{L^\infty(L^1)}\leq \mathcal{O}\left(\frac{1}{D^2}\right),  \\ & ||\TV(b^e)(K)||_{\infty}\leq  ||\TV(b)(K)||_{\infty} + \mathcal{O}\left(\frac{1}{D^2}\right).
		\end{split}
	\end{equation}
\end{proposition}

\begin{proof}
Let us call $T=X\llcorner_{t=1}$ and $\kappa_1,\dots,\kappa_{D^2}$ the subsquares of the grid where the numbering is chosen such that
\begin{equation*}
    T(\kappa_i)=\kappa_{i+1},\quad T(\kappa_n)=\kappa_1.
\end{equation*}Let us define
\begin{equation*}
    X^e_t=\begin{cases}
        X_t\circ U^1_t & x\in\kappa_1, \\
        X_t & \text{otherwise},
    \end{cases}
\end{equation*}
where the flow $U^1_t=\theta^{-1}\circ U_t\circ \theta$ and $\theta$ is the affine map from $\kappa_1$ to $K$, i.e. $\theta(x,y)=(Dx,Dy)$. \\

\noindent
We first prove the ergodicity of $T^e=X^e\llcorner_{t=1}$. Assume by contradiction that $T^e$ is not ergodic, then there exists a measurable set $B$ such that $T^e(B)=B$ and $0<|B|<1$. We claim that $|B\cap\kappa_1|>0$. Indeed, since $|B|>0$ there exists $i$ such that $|B\cap\kappa_i|>0$. If $i=1$ we have nothing to prove, if not, since $T^e$ is measure-preserving, then $|T^e(B\cap \kappa_i)|>0$. But
\begin{equation*}
    0<|T^e(B\cap \kappa_i)|=|T^e(B)\cap T^e(\kappa_i)|=|B\cap \kappa_{i+1}|
\end{equation*}
(we have used that the set $B$ is invariant) and re-applying the map $T^e$ sufficiently many times we have the claim. Moreover, $|B\cap \kappa_1|<\frac{1}{D^2}$. If not, that is $|B\cap\kappa_1|=\frac{1}{D^2}$, then $|B\cap\kappa_i|=\frac{1}{D^2}$ for every $i=1,\dots,D^2$, again by using the fact that $B$ is invariant and that $T^e(\kappa_i)=\kappa_{i+1}$ and $T^e(\kappa_{D^2})=\kappa_1$. But now
\begin{equation*}
    |B|=\sum_{i=1}^{D^2} |B\cap\kappa_i|=\sum_{i=1}^{D^2}\frac{1}{D^2}=1,
\end{equation*}
which is a contradiction, since $|B|<1$. Now, the fact that $0<|B\cap\kappa_1|<\frac{1}{D^2}$ implies that $U^1_1(B\cap\kappa_1)\not=B\cap\kappa_1$ because $U^1_1$ is mixing (and thus ergodic). But this is a contradiction because $T^e(B\cap\kappa_1)=B\cap\kappa_2$ and applying to both of them $T^{D^2-1}$ we find that
\begin{equation*}
   U^1_1(B\cap\kappa_1)=T^{D^2-1}(T^e(B\cap\kappa_1))=T^{D^2-1}(B\cap\kappa_2)=(T^e)^{D^2-1}(B\cap\kappa_2)=B\cap\kappa_1,
\end{equation*}
where we have used that $T^{D^2}=Id$.
To prove the estimates \eqref{BV:ergodic} we have to observe first that $U_t$  acts only on $\kappa_1$, then that it is the composition of two rotations (see [Figure 1, \cite{univ:mixer}]), that is $\TV(\dot{U}_t(U^{-1}_t))(\bar{\kappa}_1)\leq \mathcal{O}\left(\frac{1}{D^2}\right)$ (see again Lemma \ref{TV:rotations}). 
\end{proof}

%% file: Brenier/cicliconc.pdf_tex
\begingroup%
  \makeatletter%
  \providecommand\color[2][]{%
    \errmessage{(Inkscape) Color is used for the text in Inkscape, but the package 'color.sty' is not loaded}%
    \renewcommand\color[2][]{}%
  }%
  \providecommand\transparent[1]{%
    \errmessage{(Inkscape) Transparency is used (non-zero) for the text in Inkscape, but the package 'transparent.sty' is not loaded}%
    \renewcommand\transparent[1]{}%
  }%
  \providecommand\rotatebox[2]{#2}%
  \newcommand*\fsize{\dimexpr\f@size pt\relax}%
  \newcommand*\lineheight[1]{\fontsize{\fsize}{#1\fsize}\selectfont}%
  \ifx\svgwidth\undefined%
    \setlength{\unitlength}{240.75000433bp}%
    \ifx\svgscale\undefined%
      \relax%
    \else%
      \setlength{\unitlength}{\unitlength * \real{\svgscale}}%
    \fi%
  \else%
    \setlength{\unitlength}{\svgwidth}%
  \fi%
  \global\let\svgwidth\undefined%
  \global\let\svgscale\undefined%
  \makeatother%
  \begin{picture}(1,1.00000009)%
    \lineheight{1}%
    \setlength\tabcolsep{0pt}%
    \put(0,0){\includegraphics[width=\unitlength,page=1]{cicliconc.pdf}}%
    \put(0.02833759,0.05688691){\makebox(0,0)[lt]{\lineheight{1.25}\smash{\begin{tabular}[t]{l}$\kappa^1_1$\end{tabular}}}}%
    \put(0.15075128,0.05904742){\makebox(0,0)[lt]{\lineheight{1.25}\smash{\begin{tabular}[t]{l}$\kappa^3_1$\end{tabular}}}}%
    \put(0.15035813,0.18141156){\makebox(0,0)[lt]{\lineheight{1.25}\smash{\begin{tabular}[t]{l}$\kappa^1_2$\end{tabular}}}}%
    \put(0.02993316,0.18141156){\makebox(0,0)[lt]{\lineheight{1.25}\smash{\begin{tabular}[t]{l}$\kappa^1_4$\end{tabular}}}}%
    \put(0.27755878,0.05688691){\makebox(0,0)[lt]{\lineheight{1.25}\smash{\begin{tabular}[t]{l}$\kappa^3_2$\end{tabular}}}}%
    \put(0.39997243,0.05904742){\makebox(0,0)[lt]{\lineheight{1.25}\smash{\begin{tabular}[t]{l}$\kappa^3_3$\end{tabular}}}}%
    \put(0.39957933,0.18141156){\makebox(0,0)[lt]{\lineheight{1.25}\smash{\begin{tabular}[t]{l}$\kappa^1_5$\end{tabular}}}}%
    \put(0.27915435,0.18141156){\makebox(0,0)[lt]{\lineheight{1.25}\smash{\begin{tabular}[t]{l}$\kappa^1_3$\end{tabular}}}}%
    \put(0.02833759,0.30610806){\makebox(0,0)[lt]{\lineheight{1.25}\smash{\begin{tabular}[t]{l}$\kappa^2_7$\end{tabular}}}}%
    \put(0.15075128,0.30826857){\makebox(0,0)[lt]{\lineheight{1.25}\smash{\begin{tabular}[t]{l}$\kappa^2_1$\end{tabular}}}}%
    \put(0.15035813,0.43063271){\makebox(0,0)[lt]{\lineheight{1.25}\smash{\begin{tabular}[t]{l}$\kappa^2_2$\end{tabular}}}}%
    \put(0.02993316,0.43063271){\makebox(0,0)[lt]{\lineheight{1.25}\smash{\begin{tabular}[t]{l}$\kappa^2_4$\end{tabular}}}}%
    \put(0.27755878,0.30610806){\makebox(0,0)[lt]{\lineheight{1.25}\smash{\begin{tabular}[t]{l}$\kappa^2_3$\end{tabular}}}}%
    \put(0.39997243,0.30826857){\makebox(0,0)[lt]{\lineheight{1.25}\smash{\begin{tabular}[t]{l}$\kappa^4_1$\end{tabular}}}}%
    \put(0.39957933,0.43063271){\makebox(0,0)[lt]{\lineheight{1.25}\smash{\begin{tabular}[t]{l}$\kappa^5_2$\end{tabular}}}}%
    \put(0.27915435,0.43063271){\makebox(0,0)[lt]{\lineheight{1.25}\smash{\begin{tabular}[t]{l}$\kappa^5_{11}$\end{tabular}}}}%
    \put(0.52677993,0.05688691){\makebox(0,0)[lt]{\lineheight{1.25}\smash{\begin{tabular}[t]{l}$\kappa^3_4$\end{tabular}}}}%
    \put(0.64919367,0.05904742){\makebox(0,0)[lt]{\lineheight{1.25}\smash{\begin{tabular}[t]{l}$\kappa^3_8$\end{tabular}}}}%
    \put(0.64880048,0.18141156){\makebox(0,0)[lt]{\lineheight{1.25}\smash{\begin{tabular}[t]{l}$\kappa^5_9$\end{tabular}}}}%
    \put(0.5283755,0.18141156){\makebox(0,0)[lt]{\lineheight{1.25}\smash{\begin{tabular}[t]{l}$\kappa^5_1$\end{tabular}}}}%
    \put(0.77600117,0.05688691){\makebox(0,0)[lt]{\lineheight{1.25}\smash{\begin{tabular}[t]{l}$\kappa^5_7$\end{tabular}}}}%
    \put(0.89841481,0.05904742){\makebox(0,0)[lt]{\lineheight{1.25}\smash{\begin{tabular}[t]{l}$\kappa^5_8$\end{tabular}}}}%
    \put(0.89802172,0.18141156){\makebox(0,0)[lt]{\lineheight{1.25}\smash{\begin{tabular}[t]{l}$\kappa^4_6$\end{tabular}}}}%
    \put(0.77759674,0.18141156){\makebox(0,0)[lt]{\lineheight{1.25}\smash{\begin{tabular}[t]{l}$\kappa^5_6$\end{tabular}}}}%
    \put(0.77600117,0.30610806){\makebox(0,0)[lt]{\lineheight{1.25}\smash{\begin{tabular}[t]{l}$\kappa^5_{10}$\end{tabular}}}}%
    \put(0.89841481,0.30826857){\makebox(0,0)[lt]{\lineheight{1.25}\smash{\begin{tabular}[t]{l}$\kappa^4_4$\end{tabular}}}}%
    \put(0.89802172,0.43063271){\makebox(0,0)[lt]{\lineheight{1.25}\smash{\begin{tabular}[t]{l}$\kappa^3_5$\end{tabular}}}}%
    \put(0.77759674,0.43063271){\makebox(0,0)[lt]{\lineheight{1.25}\smash{\begin{tabular}[t]{l}$\kappa^8_2$\end{tabular}}}}%
    \put(0.52677993,0.30610806){\makebox(0,0)[lt]{\lineheight{1.25}\smash{\begin{tabular}[t]{l}$\kappa^8_1$\end{tabular}}}}%
    \put(0.64919367,0.30826857){\makebox(0,0)[lt]{\lineheight{1.25}\smash{\begin{tabular}[t]{l}$\kappa^5_{13}$\end{tabular}}}}%
    \put(0.64880048,0.43063271){\makebox(0,0)[lt]{\lineheight{1.25}\smash{\begin{tabular}[t]{l}$\kappa^4_5$\end{tabular}}}}%
    \put(0.5283755,0.43063271){\makebox(0,0)[lt]{\lineheight{1.25}\smash{\begin{tabular}[t]{l}$\kappa^4_2$\end{tabular}}}}%
    \put(0.77600117,0.80455045){\makebox(0,0)[lt]{\lineheight{1.25}\smash{\begin{tabular}[t]{l}$\kappa^8_7$\end{tabular}}}}%
    \put(0.89841481,0.80671096){\makebox(0,0)[lt]{\lineheight{1.25}\smash{\begin{tabular}[t]{l}$\kappa^{10}_1$\end{tabular}}}}%
    \put(0.89802172,0.9290751){\makebox(0,0)[lt]{\lineheight{1.25}\smash{\begin{tabular}[t]{l}$\kappa^8_6$\end{tabular}}}}%
    \put(0.77759674,0.9290751){\makebox(0,0)[lt]{\lineheight{1.25}\smash{\begin{tabular}[t]{l}$\kappa^8_5$\end{tabular}}}}%
    \put(0.76042479,0.5553293){\makebox(0,0)[lt]{\lineheight{1.25}\smash{\begin{tabular}[t]{l}$\kappa^6_9$\end{tabular}}}}%
    \put(0.88283852,0.55748981){\makebox(0,0)[lt]{\lineheight{1.25}\smash{\begin{tabular}[t]{l}$\kappa^8_9$\end{tabular}}}}%
    \put(0.88244542,0.67985391){\makebox(0,0)[lt]{\lineheight{1.25}\smash{\begin{tabular}[t]{l}$\kappa^8_3$\end{tabular}}}}%
    \put(0.76202035,0.67985391){\makebox(0,0)[lt]{\lineheight{1.25}\smash{\begin{tabular}[t]{l}$\kappa^3_6$\end{tabular}}}}%
    \put(0.52677993,0.5553293){\makebox(0,0)[lt]{\lineheight{1.25}\smash{\begin{tabular}[t]{l}$\kappa^6_1$\end{tabular}}}}%
    \put(0.64919367,0.55748981){\makebox(0,0)[lt]{\lineheight{1.25}\smash{\begin{tabular}[t]{l}$\kappa^8_{10}$\end{tabular}}}}%
    \put(0.64880048,0.67985391){\makebox(0,0)[lt]{\lineheight{1.25}\smash{\begin{tabular}[t]{l}$\kappa^5_3$\end{tabular}}}}%
    \put(0.5283755,0.67985391){\makebox(0,0)[lt]{\lineheight{1.25}\smash{\begin{tabular}[t]{l}$\kappa^6_{11}$\end{tabular}}}}%
    \put(0.52677993,0.80455045){\makebox(0,0)[lt]{\lineheight{1.25}\smash{\begin{tabular}[t]{l}$\kappa^6_8$\end{tabular}}}}%
    \put(0.64919367,0.80671096){\makebox(0,0)[lt]{\lineheight{1.25}\smash{\begin{tabular}[t]{l}$\kappa^3_7$\end{tabular}}}}%
    \put(0.64880048,0.9290751){\makebox(0,0)[lt]{\lineheight{1.25}\smash{\begin{tabular}[t]{l}$\kappa^8_8$\end{tabular}}}}%
    \put(0.5283755,0.9290751){\makebox(0,0)[lt]{\lineheight{1.25}\smash{\begin{tabular}[t]{l}$\kappa^6_7$\end{tabular}}}}%
    \put(0.27755878,0.80455045){\makebox(0,0)[lt]{\lineheight{1.25}\smash{\begin{tabular}[t]{l}$\kappa^6_{12}$\end{tabular}}}}%
    \put(0.39997243,0.80671096){\makebox(0,0)[lt]{\lineheight{1.25}\smash{\begin{tabular}[t]{l}$\kappa^9_1$\end{tabular}}}}%
    \put(0.39957933,0.9290751){\makebox(0,0)[lt]{\lineheight{1.25}\smash{\begin{tabular}[t]{l}$\kappa^6_5$\end{tabular}}}}%
    \put(0.27915435,0.9290751){\makebox(0,0)[lt]{\lineheight{1.25}\smash{\begin{tabular}[t]{l}$\kappa^8_4$\end{tabular}}}}%
    \put(0.27755878,0.5553293){\makebox(0,0)[lt]{\lineheight{1.25}\smash{\begin{tabular}[t]{l}$\kappa^5_4$\end{tabular}}}}%
    \put(0.39997243,0.55748981){\makebox(0,0)[lt]{\lineheight{1.25}\smash{\begin{tabular}[t]{l}$\kappa^2_5$\end{tabular}}}}%
    \put(0.39957933,0.67985391){\makebox(0,0)[lt]{\lineheight{1.25}\smash{\begin{tabular}[t]{l}$\kappa^4_3$\end{tabular}}}}%
    \put(0.27915435,0.67985391){\makebox(0,0)[lt]{\lineheight{1.25}\smash{\begin{tabular}[t]{l}$\kappa^6_9$\end{tabular}}}}%
    \put(0.0127613,0.5553293){\makebox(0,0)[lt]{\lineheight{1.25}\smash{\begin{tabular}[t]{l}$\kappa^5_{12}$\end{tabular}}}}%
    \put(0.13517498,0.55748981){\makebox(0,0)[lt]{\lineheight{1.25}\smash{\begin{tabular}[t]{l}$\kappa^5_5$\end{tabular}}}}%
    \put(0.13478184,0.67985391){\makebox(0,0)[lt]{\lineheight{1.25}\smash{\begin{tabular}[t]{l}$\kappa^2_6$\end{tabular}}}}%
    \put(0.01435686,0.67985391){\makebox(0,0)[lt]{\lineheight{1.25}\smash{\begin{tabular}[t]{l}$\kappa^6_2$\end{tabular}}}}%
    \put(0.0127613,0.80455045){\makebox(0,0)[lt]{\lineheight{1.25}\smash{\begin{tabular}[t]{l}$\kappa^6_{10}$\end{tabular}}}}%
    \put(0.13517498,0.80671096){\makebox(0,0)[lt]{\lineheight{1.25}\smash{\begin{tabular}[t]{l}$\kappa^7_1$\end{tabular}}}}%
    \put(0.13478184,0.9290751){\makebox(0,0)[lt]{\lineheight{1.25}\smash{\begin{tabular}[t]{l}$\kappa^6_3$\end{tabular}}}}%
    \put(0.01435686,0.9290751){\makebox(0,0)[lt]{\lineheight{1.25}\smash{\begin{tabular}[t]{l}$\kappa^6_4$\end{tabular}}}}%
    \put(0,0){\includegraphics[width=\unitlength,page=2]{cicliconc.pdf}}%
  \end{picture}%
\endgroup%

%% file: stronglymixing.tex
\subsection{Density of strongly mixing vector fields}
\label{S:mixing_maps}

As in the previous section, we use the density of cyclic permutations to show that the vector fields whose flow is strongly mixing are dense in $\mathcal{U}$ with the $L^1_{t,x}$-topology. Again we use the universal mixer constructed in \cite{univ:mixer}. The main result here is the following

\begin{proposition}
\label{Prop:smix_dense}
Let $b\in L^\infty_t(\BV_x)$ and let $X_t$ be its RLF, and assume that $b_t = 0$ for $t \in [0,2\delta]$ and $X_{t=1}$ is a cyclic permutation of squares of the grid $\N\times\N\frac{1}{D}$, $D = 2^p$. Then there exists $b^s \in L^\infty_t(\BV_x)$ divergence-free strongly mixing vector field such that 
\begin{equation}
\label{BV:strongl}
\begin{split}
	&||b-b^s||_{L^1(L^1)}\leq \mathcal{O}\left(\frac{1}{\delta D}\right),  \\ & ||\TV(b^s)(K)||_{\infty}\leq  ||\TV(b)(K)||_{\infty} + \mathcal{O}\left(\delta^{-1}\right).
		\end{split}
	\end{equation}
\end{proposition}

In the proof it is shown that the mixing is actually exponential, in the sense that for every set in a countable family of sets $\{B_i\}_i$ generating the Borel $\sigma$-algebra it holds
\begin{equation*}
\big| T^q(B_i) \cap B_j \big| - |B_i| |B_j| = \mathcal O(1) c_{ij}^q, \quad c_{ij} < 1.
\end{equation*}

\begin{proof}
Let us call $T=X\llcorner_{t=1}$ and $\kappa_1,\dots,\kappa_{D^2}$ the subsquares of the grid where the numbering is chosen such that
\begin{equation*}
    T(\kappa_i)=\kappa_{i+1},\quad T(\kappa_{D^2})=\kappa_1.
\end{equation*}
If $\{1,\dots,D^2\} \ni \ell \mapsto j(\ell) \in \{1,\dots,D^2\}$ is an enumeration of $\kappa_i$ such that $\kappa_{j(\ell)},\kappa_{j(\ell+1)}$ are adjacent, consider the rescaled universal mixer $U^{\ell,\ell+1}_t$ acting on $\kappa_\ell,\kappa_{\ell+1}$ in the time interval $[0,\delta]$, whose generating vector field $b^{U^{\ell,\ell+1}}$ satisfies the estimates
\begin{equation*}
\|b^{U^{\ell,\ell+1}}_t\|_{L^1} = \mathcal O(1) \frac{1}{\delta} \frac{1}{D^3}, \quad \TV(b^{U^{\ell,\ell+1}}_t) = \mathcal O(1) \frac{1}{\delta} \frac{1}{D^2}.
\end{equation*}

The idea is to define the a new vector field as in \eqref{cyclic:flow}
	\begin{equation*}
		X^s_t(x)=\begin{cases}
			M_t(x) & t\in [0,2\delta], \\
			X_{t}\circ M_{2\delta} (x) & t\in [2\delta,1],
		\end{cases}
	\end{equation*}
where the map $M_t$, $t \in [0,2\delta]$, is defined as follows: 
\begin{equation*}
M_t(x) = \begin{cases}
U^{\ell,\ell+1}_t(x) & t \in [0,\delta], \ell \ \text{even}, \\
U^{\ell,\ell+1}_t(x) & t \in [\delta,2\delta], \ell \ \text{odd}.
\end{cases}
\end{equation*}
The estimates \eqref{BV:strongl} follows as in Proposition \ref{Prop:ergod_dense}, so we are left with the proof that $T^s = X^s_1$ is strongly mixing.

The map $T^s$ is the composition of $3$ maps $T_3 \circ T_2 \circ T_1$ acting as follows (all indexes should be indended modulus $D^2$):
\begin{enumerate}
\item $T_1$ is the folded Baker's map $U$ acting on the couples $\ell,\ell+1$, $\ell = 0,2,\dots$ even;
\item $T_2$ is the folded Baker's map $U$ acting on the couples $\ell,\ell+1$, $\ell = 1,3,\dots$ odd;
\item $T_3$ is a cyclic permutation $\ell \to j^{-1}(j(\ell) + 1)$.
\end{enumerate}
We first compute the evolution of a rectangle $a$ of the form
\begin{equation*}
a = 2^{-p}[k,k+1] \times 2^{-p'} [k',k'+1]\frac{1}{D}, \quad k = 0,\dots, 2^{p} D -1, \ k' = 0,\dots,2^{p'} D -1, \ p,p' \in \N. 
\end{equation*}
By definition of $U$ \eqref{baker} we obtain that if $p \geq 1$ then the map $T_1$ does not split $a$ into disjoint rectangles, i.e. 
\begin{equation*}
T_1 a = 2^{1-p} [\tilde k,\tilde k+1] \times 2^{-p'-1}[\tilde k',\tilde k'+1] \frac{1}{D}, \quad \tilde k= 0,\dots, 2^{p-1} D-1, \tilde k'= 0,\dots,2^{p'+1} D -1, 
\end{equation*}
and the same happens for $T_2$:
\begin{equation*}
T_2 a = 2^{1-p} [\hat k,\hat k+1] \times 2^{-p'-1} [\hat k',\hat k'+1] \frac{1}{D}, \quad \hat k = 0,\dots, 2^{p-1} D-1, \hat k'= 0,\dots,2^{p'+1} D -1. 
\end{equation*}
Hence if
\begin{equation}
\label{Equa:a_rect_baker}
a = 2^{-2p}[k,k+1] \times 2^{-2p'} [k',k'+1]\frac{1}{D}, \quad k = 0,\dots, 2^{2p} D -1, \ k' = 0,\dots,2^{2p'} D -1, \ p,p' \in \N, 
\end{equation}
then
\begin{equation*}
T_2 \circ T_1 a = 2^{2(1-p)} [\check k,\check k+1] \times 2^{-2(p'+1)} [\check k',\check k'+1] \frac{1}{D}, \quad \check k = 0,\dots, 2^{2(p-1)} D-1, \check k'= 0,\dots,2^{2(p'+1)} D -1,
\end{equation*}
and being the action of $T_3$ just a permutation, the final form $T^s a = T_3 \circ T_2 \circ T_1 a$ is again a rectangle. \\
When $p=0$, instead the rectangle $a$ is mapped into two rectangles belonging to two different subsquares $\kappa,\kappa'$
\begin{equation*}
T_1 a = [\tilde k_1,\tilde k_1+1] \times 2^{-p'-1} [\tilde k_1',\tilde k_1' + 1] \frac{1}{D} \cup [\tilde k_2,\tilde k_2+1] \times 2^{-p'-1} [\tilde k_2',\tilde k_2' + 1] \frac{1}{D},
\end{equation*}
and the action of $T_2$ divides $T_1 a$ into 4 rectangles of horizontal length $1/D$ belonging to 4 different subsquares. As before, $T_3$ just shuffles them into new locations. \\
The same happens when considering $(T^s)^{-1}$: if $p' \leq 1$ and $a$ is given by \eqref{Equa:a_rect_baker} then $(T^s)^{-1} a$ is still a rectangle of side $2^{-2(p+1)} \times 2^{-2(p'-1)} \frac{1}{D}$, while for $p'=0$ it is split into 4 rectangles with vertical size equal to $1/D$.

In particular, starting from two squares $a,a'$ of side $(2^{-2p} D)^{-2}$, for $q \geq p$ the set $(T^s)^{q} a$ is made of disjoint rectangles whose horizontal side is $D^{-1}$, and $(T^s)^{-q} a'$ is made of disjoint rectangles whose vertical side is $D^{-1}$. Hence if the masses of $(T^s)^q a$, $(T^s)^{-q'} a'$ inside $\kappa_i$ are $m_i(q)$, $m_i'(-q')$, then by Fubini
\begin{equation*}
\mathcal L^2 \big( (T^s)^{q} a \cap (T^s)^{-q'} a' \big) = \sum_{i=1}^{D^2} D^2 m_i(q) m'_i(-q').
\end{equation*}
In order to prove the strong mixing it is enough to show that
\begin{equation*}
m_i(q) \to \frac{\mathcal L^2(a)}{D^2}, \ m'_i(-q') \to \frac{\mathcal L^2(a')}{D^2} \quad q,q' \to \infty.
\end{equation*}
Actually, we will show that the above convergence is exponential, which implies that the mixing is exponential. We prove the above exponential convergence for $m_i(q)$, the other being completely similar.

Once $(T^s)^q a$ has become a rectangle of horizontal side $1/D$, the distribution of mass by $T^s$ is computed by the action of the following matrices on the vector $(m_i)_i$:
\begin{enumerate}
\item the matrix $A_1$ corresponding to the map $T_1$,
\begin{equation*}
(A_1)_{\ell'\ell} = \frac{1}{2} \begin{cases}
\delta_{\ell'\ell} + \delta_{\ell'(\ell-1)} & \ell' = 0,2,\dots, \\
\delta_{\ell'(\ell+1)} + \delta_{\ell'\ell} & \ell' = 1,3,\dots;
\end{cases} 
\end{equation*}
\item the matrix $A_2$ corresponding to the map $T_2$,
\begin{equation*}
(A_2)_{\ell'\ell} = \frac{1}{2} \begin{cases}
\delta_{\ell'\ell} + \delta_{\ell'(\ell+1)} & \ell' = 0,2,\dots, \\
\delta_{\ell'(\ell-1)} + \delta_{\ell'\ell} & \ell' = 1,3,\dots;
\end{cases} 
\end{equation*}
\item the permutation matrix $A_3$ corresponding to $T_3$.
\end{enumerate}
Being the Markov process generated by the matrix $P = A_3 A_2 A_1$ finite dimensional, exponential mixing is equal to strong mixing, and we prove directly that $P$ has a simple eigenvalue of modulus $1$ whose eigenvector is necessarily the uniform distribution $(1/D^2,1/D^2,\dots)$: in particular this gives that $P$ is aperiodic (Definition \ref{aperiodic} and Proposition \ref{markov:mix}). Indeed, for $v \in \C^D$ one considers the functional $|v|$, and by simple computations it holds $|A_3 v| = |v|$ and
\begin{equation*}
|A_1 v| = |v| \quad \text{iff} \quad v_\ell = v_{\ell+1} \ \text{for} \ \ell=0,2,\dots,
\end{equation*}
\begin{equation*}
|A_2 v| = |v| \quad \text{iff} \quad v_{\ell} = v_{\ell+1} \ \text{for} \ \ell=1,3,\dots.
\end{equation*}
Hence the unique $v$ such that $|Av|=|v|$ is $v=(1/D^2,1/D^2,\dots)$, and $1$ is a simple eigenvector.
\end{proof}

\begin{remark}
\label{Rem:as_before}
As in Remark \ref{Rem:better_mixing}, one could let the Bakers map to act during the time evolution of $X_t$, but in this case the distance in $L^\infty L^1$ would be of order $1$. The problem is that the maps $T_1,T_2$ are acting on the whole set $K = [0,1]^2$, and the vector field $b^s_t - b_t$ is of order $1$ as in \eqref{Equa:O1_b_t}.
\end{remark}

\subsection{Proof of the density of strongly mixing vector fields}
\label{Ss:proof_main_intro}

We are now ready to prove the density of strongly mixing vector fields in $\mathcal{U}$, which implies the statement by Corollary \ref{Cor:strong_mix_dense}. It will be obtained through the following steps.

\begin{enumerate}
    \item Let $b \in \mathcal{U}$: by the very construction of the set $\mathcal{U}$ (Proposition \ref{prop:gen}), we can assume that $b \in L^\infty_t \BV_x$. Fix $\epsilon > 0$.
    \item By the continuity of translation in $L^1$, we can take $0 < \delta \ll 1$ such that defining
    \begin{equation*}
        b^\delta = \begin{cases}
        0 & t \in [0,3\delta), \\
    \frac{1}{1-3\delta} b_{(t-3\delta)/(1-3\delta)} & t \in [3\delta,1],
        \end{cases}
    \end{equation*}
    it holds
    \begin{equation*}
        \|b^\delta - b\|_{L^1_{t,x}} < \frac{\epsilon}{4}.
    \end{equation*}
    Since
    \begin{equation*}
        \|\TV(b^\delta)\|_\infty = \frac{1}{1-3\delta} \|\TV(b)\|_\infty
    \end{equation*}
    then $b^\delta \in \mathcal{U}$. Clearly we can also assume that $b^\delta$ is compactly supported in $K$.
    \item Use Theorem \ref{Prop:piececlo} to approximate $b^\delta$ in $[3\delta,1]$ with a vector field $b^{\epsilon\delta} \in L^\infty_t \BV_x \subset \mathcal{U}$ such that
    \begin{equation*}
        \|b^\delta - b^{\epsilon\delta}\|_{L^1_{t,x}} < \frac{\epsilon}{4},
    \end{equation*}
    and such that its RLF is a permutation of squares of size $D^{-1}$. We can assume that
    \begin{equation}
        \label{Equa:delta_D_5_eps}
        D \gg \frac{1}{\epsilon \delta}.
    \end{equation}
    \item Apply Lemma \ref{lem:cycl:1} together with Proposition \ref{unique:cycle} to $b^{\epsilon\delta}$ for $t \in [2\delta,1]$ obtaining a new vector field $b^{\epsilon\delta c} \in L^\infty_t \BV_x \subset \mathcal{U}$ such that
    \begin{equation*}
        \|b^{\epsilon\delta} - b^{\epsilon\delta c}\|_{L^1_{t,x}} \leq \mathcal O \bigg( \frac{1}{DM} \bigg) < \frac{\epsilon}{4}
    \end{equation*}
    for $M = 2^{p'} \gg 1$, and
    such that its RLF is a single cycle of squares of size $(DM)^{-1}$. 
    \item Finally, apply Proposition \ref{Prop:smix_dense} to $b^{\epsilon\delta c}$ in $t \in [0,1]$ obtaining a strongly (exponentially) mixing vector field $b^{\epsilon\delta cs} \in L^\infty_t \BV_x \subset \mathcal{U}$ such that
    \begin{equation*}
        \|b^{\epsilon\delta c} - b^{\epsilon\delta cs}\|_{L^1_{t,x}} \leq \mathcal O \bigg( \frac{1}{\delta D} \bigg) < \frac{\epsilon}{4}
    \end{equation*}
    by using \eqref{Equa:delta_D_5_eps}.
\end{enumerate}

We thus conclude that for every $b \in L^\infty_t \BV_x$ and $\epsilon > 0$ there is a vector field $b^s \in L^\infty_t \BV_x$ exponentially mixing such that
\begin{equation*}
    \|b - b^s\|_{L^1_{t,x}} < \epsilon,
\end{equation*}
which is our aim.

%% file: permutations.tex
\section{Permutation Flow} \label{S:permutation}

In this section we prove the key tool of this paper, namely the approximation in $L^1$ of any BV vector field with another BV vector field such that its flow at $t=1$ is a permutation of subsquares, i.e. it is a rigid translation of subsquares of a grid partition of $K = [0,1]^2$. The approach is inspired by \cite{Shnirelman}, with the additional difficulty that we need to control the BV norm of the approximating vector field. We will address also the $d$-dimensional case, explaining the additional technicalities needed to prove the same approximation result in the general case.

\noindent This section is divided into two parts: in the first one we collect some preliminary estimates which will be used as building blocks in the proof of the main theorem, while in the second part we state the main approximation theorem and give its proof.

\subsection{Affine approximations of smooth flows}
\label{Ss>pert_flow}

The next lemma is almost the same of \cite[Lemma 4.3]{Shnirelman}. In order to follow the original Shnirelman's Lemma we require the subrectangles in the next lemma to be dyadic (i.e. their corners belong to a dyadic partition, see Remark \ref{Shnirelman: sbaglio} however), but we notice that the proof of the main theorem works in the same way just asking subrectangles with rational coordinates to be mapped affinely onto subrectangles with rational coordinates. At the end this section we will address the same lemma in the general case $d > 2$, which in the original paper is not proved.

\noindent Let $T$ be a measure-preserving diffeomorphism $T:[0,1]^2\rightarrow [0,1]^2$ of class $C^3$ and such that $T = \Id$ in a neighborhood of $\partial [0,1]^2$. Assume that it is close to the identity, i.e. there exists $\delta>0$ sufficiently small such that $||T-\Id||_{\mathcal{C}^1}\leq\delta$.

\begin{lemma}
\label{lem: dyadic}
	 There exists $N\in\N$, $N=2^{p}$, and a path of measure-preserving invertible maps $t\rightarrow \sigma_t$ piecewise smooth w.r.t. the time variable $t$ such that $\sigma_0=T$ and $\sigma_1$ maps arbitrarily small dyadic rectangles $P_{ij} \in \N \times \N\frac{1}{N} = K_N$ (meaning that their boundaries are in the net $K_N$) affinely onto dyadic rectangles $\tilde P_{ij} \in K_N$.
	
\noindent	Moreover, the map $\sigma$ is of the form
	\begin{equation}
\label{Equa>pertur_form}
\sigma_t = T \circ \xi_{3t} \ind_{[0,1/3]}(t) + \zeta_{3t-1} \circ T \circ \xi_1 \ind_{[1/3,2/3]}(t) + \eta_{3t-2} \circ \zeta_1 \circ T \circ \xi_1 \ind_{[2/3,1]}(t).
\end{equation}
	where $\xi, \eta : [0,1] \times [0,1]^2 \to [0,1]^2$ are piecewise smooth and $\zeta : [0,1] \times [0,1]^2 \to [0,1]^2$ is smooth, so that for every $t\in[0,1]$, the map $\sigma_t$ is piecewise smooth on each subrectangle $\kappa$ and it extends continuosly on $\bar{\kappa}$. \\
	Finally, the space differential $D\sigma_1$ of $\sigma_{t=1}$ is a constant diagonal matrix in each subrectangle.
\end{lemma}

The number $N$ is used in the next results in order to have that the perturbation is arbitrarily small in $L^1_{t,x}$.

\begin{proof}
%
The proof is given in $3$ steps:
\begin{enumerate}
\item first by an arbitrarily small perturbation of the final configuration we make sure the area of the regions which will be mapped into rectangles is dyadic;
\item secondly we perturb along horizontal slabs in order to have that vertical sections of the slabs are mapped into vertical segments;
\item finally we perturb vertical slabs so that the image of particular rectangles are rectangles and vertical segments remains vertical segments.
\end{enumerate}
The composition of all 3 maps with $T$ as in \eqref{Equa>pertur_form} will be the movement $\sigma_t$. We will use the notation
\begin{equation*}
[0,1]^2 \ni (x_1,x_2) \mapsto T(x_1,x_2) = (z_1,z_2) \in [0,1]^2
\end{equation*}
to avoid confusion between the final coordinates and the initial ones. When piecing together maps which are defined in closed sets with piecewise regular boundaries, we will neglect the negligible superposition of boundaries for simplicity: this slight inaccuracy should not generate confusion.

\smallskip

\noindent{\it Step 0: initial grid and perturbation.}
For $N_0 = 2^{p_0} \gg 1$ define the horizontal and vertical slabs
\begin{equation*}
H_j = [0,1] \times 2^{-p_0} [j-1,j], \quad V_i = 2^{-p_0} [i-1,i] \times [0,1], \quad i,j = 1,\dots ,2^{p_0}.
\end{equation*}
The image of the  horizontal lines
\begin{equation*}
x_1 \mapsto T(x_1,x_2)
\end{equation*}
can be written as graphs of functions 
\begin{equation*} 
z_1 \mapsto g(z_1,x_2),
\end{equation*}
and divides every vertical slab $V_i$ into $N_0 = 2^{p_0}$ parts
\begin{equation*}
\tilde \omega_{ij} = \Big\{ (i-1) 2^{-p_0} \leq z_1 \leq i 2^{-p_0}, g(z_1,(j-1)2^{-p_0}) \leq z_2 \leq g(z_1,j2^{-p_0}) \Big\}.
\end{equation*}

Let $\zeta_t : [0,1]^2 \to [0,1]^2$ be a measure preserving flow, moving mass across the boundary of $\tilde\omega_{ij}$: we can assume w.l.o.g that the mass flow $\phi_{ij,i'j'}$ across the boundary from $\tilde \omega_{ij}$ to $\tilde \omega_{i'j'}$ occurs in the relative interior of $\partial \tilde \omega_{ij} \cap \partial \tilde \omega_{i'j'}$. The measure preserving condition requires that
\begin{equation*}
\phi_{ij,(i-1)j} + \phi_{ij,(i+1)j} + \phi_{ij,i(j+1)} + \phi_{ij,i(j-1)} = 0.
\end{equation*}
Set $T' = \zeta_1 \circ T$ and consider the new curves
\begin{equation*}
z_1 \mapsto g'(z_1,x_2), \quad \Graph g' = T'([0,1] \times \{x_2\}).
\end{equation*}
Let $\tilde \omega'_{ij}$ be the new regions
\begin{equation*}
\tilde \omega'_{ij} = \Big\{ (i-1) 2^{-p_0} \leq z_1 \leq i 2^{-p_0}, g'(z_1,(j-1)2^{-p_0}) \leq z_2 \leq g'(z_1,j2^{-p_0}) \Big\},
\end{equation*}
whose new area is
\begin{equation*}
\mathcal L^2(\tilde \omega'_{ij}) = \phi_{ij,i(j-1)} + \phi_{ij,i(j+1)}.
\end{equation*}

Starting with $\tilde \omega'_{11}$, we move a mass $\phi_{11,12} < 2^{-p_0-p_1} \ll 1$ so that
\begin{equation*}
\mathcal L^2(\tilde \omega'_{11}) = 2^{-p_0-p_1} n_{11} \in 2^{-p_0 - p_1} \N.
\end{equation*}
Hence a mass $-\phi_{11,21}$ is flowing to the region $\tilde \omega_{21}$. Assuming that we have
\begin{equation*}
\mathcal L^2(\tilde \omega'_{i1}) = 2^{-p_0-p_1} n_{i1} \in 2^{-p_0-p_1} \N,
\end{equation*}
and that the mass flowing by $\phi_{i1,i2},\phi_{i1,(i+1)1}$ is $< 2^{-p_1}$, we consider two cases:
\begin{enumerate}
\item if $\phi_{i1,(i+1)1} \in 2^{-p_0-p_1}[0,1)$, then we flow a mass $\phi_{(i+1)1,(i+1)2} \in 2^{-p_0-p_1} [0,1)$ so that
\begin{equation*}
\mathcal L^2(\tilde \omega'_{(i+1)1}) = 2^{-p_0-p_1} n_{(i+1)1} \in 2^{-p_1} \N,
\end{equation*}
and the flow to the right is then
\begin{equation*}
\phi_{(i+1)1,(i+2)i} = \phi_{i1,(i+1)i} - \phi_{(i+1)1,(i+1)2} \in (-1,1) 2^{-p_0-p_1},
\end{equation*}
by the balance and because they have different sign;
\item if $\phi_{i1,(i+1)1} \in 2^{-p_0-p_1}(-1,0)$, then we flow a mass $\phi_{(i+1)1,(i+1)2} \in 2^{-p_0-p_1} (-1,0)$ and obtain the same estimate.
\end{enumerate}
The last term $\tilde \omega'_{N_01}$ is computed by conservation: indeed
\begin{equation*}
\sum_i \phi_{i1,i2} = 0,
\end{equation*}
and then
\begin{equation*}
\mathcal L^2(T'([0,1] \times [0,2^{-p_0}]) = 2^{-p_0} = \sum \mathcal L^2(\tilde \omega'_{i1}) = 2^{-p_0-p_1} \sum_{i=0}^{2^{p_0}-1} n_{i1} + \mathcal L^2(\tilde \omega'_{N_01}),
\end{equation*}
so that
\begin{equation*}
\mathcal L^2(\tilde \omega'_{N_01}) = 2^{-p_0 - p_1} \bigg( 2^{p_1} - \sum_{i=0}^{2^{p_0}-1} n_{i1} \bigg) \in 2^{-p_0-p_1} \N.
\end{equation*}
The estimate of $\phi_{N_01,N_02}$ is automatic from the flow $\phi_{(N_0-1)1,N_01}$.

The above procedure is then repeated for each region
\begin{equation*}
T([0,1] \times 2^{-p_0} [j-1,j]) = \bigcup_{i=1}^{N_0} \tilde \omega_{ij},
\end{equation*}
and the flow across each boundary is $\leq 2^{-p_0-p_1}$: the conservation of the measure of $T([0,1] \times [0,j] 2^{-p_0})$ yields that the last element $\tilde \omega_{N_0j}$ is again dyadic.

From now on we work with the map $T' = \zeta_1 \circ T$.

\smallskip

\noindent{\it Step 1: perturbation along horizontal slabs.} Consider the curves
\begin{equation*}
z_2 \mapsto (T')^{-1}(z_1,z_2),
\end{equation*}
which can be parameterized as
\begin{equation*}
x_2 \mapsto f'(z_1,x_2)
\end{equation*}
being $T'$ close to the identity in $C^1$. In each $H_j$ we can determine uniquely the value
\begin{equation}
\label{Equa:aver_x_1}
x_{1,j}(z_1) = \fint_{(j-1) 2^{-p_0}}^{j 2^{-p_0}} f'(z_1,x_2) dx_2,
\end{equation}
and since $T'$ is close to identity, again every map $z_1 \mapsto x_{1,j}(z_1)$ is invertible: denote its inverse by $z_{1,j}(x_1)$.

In particular, we consider the values
\begin{equation}
\label{Equa:x_1ij}
x_{1,ij} = x_{1,j}(i2^{-p_0}).
\end{equation}
By \eqref{Equa:aver_x_1} it follows that
\begin{equation}
\label{eq:deltax}
\big( x_{1,ij} - x_{1,(i-1)j} \big) 2^{-p_0} = \mathcal L^2 \big( (T')^{-1}(\tilde \omega'_{ij}) \big) \in 2^{-p_0 - p_1} \N,
\end{equation}
so that we deduce that the elements $x_{1,ij}$ are dyadic, i.e. $x_{1,ij} \in 2^{-p_1} \N$ (being $x_{1,0j}=0$).

Consider the family of ordered curves parametrized by $x_1 \in [0,1]$
\begin{equation*}
[0,1] \times [j-1,j] 2^{-p_0} \ni t,x_2 \mapsto f'_{j,t}(x_1,x_2) = (1-t) x_1 + t f'(z_{1,j}(x_1),x_2),
\end{equation*}
and let $\xi_{j,t} : [0,1] \times [j-1,j] 2^{-p_0} \to [0,1] \times [j-1,j] 2^{-p_0}$ be the unique measure preserving map mapping each segment $\{x_1\} \times [j-1,j] 2^{-p_0}$ into the image of $(f'_{j,t}(x_1,x_2),x_2)$, $x_2 \in [j-1,j] 2^{-p_0}$. This map is uniquely defined by the balance of mass, which reads as
\begin{equation}
\label{Equa:aver_pp}
\int_{(j-1)2^{-p_0}}^{(\xi_{j,t})_2(x_1,x_2)} \partial_{x_1} f'_{j,t}(x_1,w) dw = x_2 - (j-1) 2^{-p_0}.
\end{equation}
Being $f'_{j,t}$ close to the identity, $\xi_{j,t}$ is smooth and close to the identity.

Let $\xi_t : [0,1]^2 \to [0,1]^2$ be the measure preserving map obtained by piecing together the maps $\xi_{j,t}$. By construction the map $T'' = T' \circ \xi_1$ maps each vertical segment $\{x_1\} \times [j-1,j] 2^{-p_0}$ into the vertical segment
$$
\{z_{1,j}(x_1)\} \times \big[ g'(z_{1,j}(x_1),(j-1)2^{-p_0}), g'(z_{1,j}(x_1),j2^{-p_0}) \big].
$$

\smallskip

\noindent{\it Step 2: construction of the affine maps.} The next step is to rectify the pieces of curves
\begin{equation}
\label{Equa:v_seg}
[i-1,i] 2^{-p_0} \ni z_1 \mapsto g_j(z_1) = g'(z_1,j2^{-p_0}),
\end{equation}
which are the horizontal slab of the sets $\tilde \omega'_{ij}$. Fixing a vertical slide $v_i$, one considers the unique measure preserving map $\eta_{i,t} : [i-1,i] \times [0,1] \to [i,i-1] \times [0,1]$ such that the segments \eqref{Equa:v_seg} are mapped into vertical segments and such that maps the curve $g_j([i-1,i] 2^{-p_0})$ into the curve
\begin{equation}
\label{Equa:z_2ij}
\begin{split}
g'_{t,ij}(z_1) &= (1-t) g'(z_1,j2^{-p_0}) + t \fint_{(i-1) 2^{-p_0}}^{i2^{-p_0}} g'(w,j2^{-p_0}) dw \\
&= (1-t) g'(z_1,j2^{-p_0}) + t z_{2,ij}.
\end{split}
\end{equation}
In each $\tilde \omega'_{ij}$ this map is uniquely determined by the balance
\begin{equation*}
\int_{(i-1)2^{-p_0}}^{(\eta_{i,t})_2(z_1,z_2)} \big( g'_{t,ij}(w) - g'_{t,i(j-1)}(w) \big) dw = \mathrm{constant}, 
\end{equation*}
while the vertical coordinate is affine in each vertical segment.

Let $\eta_t : [0,1]^2 \to [0,1]^2$ be the measure preserving map obtained by piecing together the maps $\eta_{i,t}$. 

\smallskip

\noindent{\it Conclusion.} Up to a time scaling, the map we are looking for is
\begin{equation*}
\sigma_t = T \circ\xi_t \ind_{[0,1]}(t) + \zeta_{t-1} \circ T \circ \xi_1 \ind_{[1,2]}(t) + \eta_{t-2} \circ \zeta_1 \circ T \circ \xi_1 \ind_{[2,3]}(t).
\end{equation*}
It is clearly measure preserving and at $t=3$ it maps affinely the rectangles with dyadic coordinates
\begin{equation*}
P_{ij} = \big[ x_{1,(i-1)j},x_{1,ij} \big] \times [j-1,j] 2^{-p_0}
\end{equation*}
into the rectangles with dyadic coordinates
\begin{equation*}
\tilde P_{ij} = [i-1,i] 2^{-p_0} \times \big[ z_{2,i(j-1)},z_{2,ij} \big].
\end{equation*}
The values $x_{1,ij}$, $z_{2,ij}$ are given by \eqref{Equa:x_1ij}, \eqref{Equa:z_2ij} and belong to $2^{-p_1} \N$. Thus $N_1=N$ is the number of the statement.

\noindent The fact that $\sigma_t$ is piecewise smooth and it extends continuously to the boundary of each $P_{ij}$ are immediate from the construction, and its smallness follows by observing that as $p_0,p_1$ diverge the maps $\xi,\zeta,\eta$ converge to the identity.
\end{proof}
\begin{figure}
    \centering
    \includegraphics[scale=0.5]{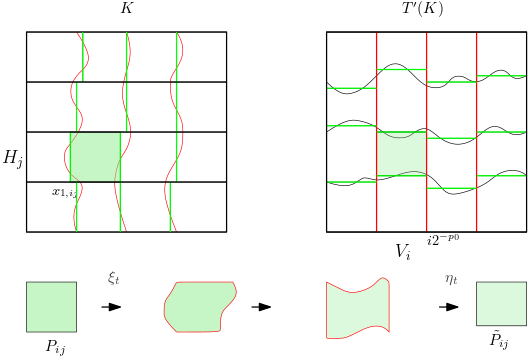}
    \caption{The action of $\sigma_t$ in Lemma \ref{lem: dyadic}: first the map $\xi_t$ moves the mass in $H_j$ in order to map the vertical green segments into the counterimages of the vertical red segments; then the map $T$ acts and the horizontal black boundaries of $H_j$ becomes the black curves, but vertical segments remain vertical; finally the action of $\eta_t$ rectifies the horizontal boundaries, while keeping vertical segments vertical.}
    \label{fig:2}
\end{figure}

Being the rectangles dyadic in the grid 
$\N\times\N \frac{1}{N}$,
then we have the following
\begin{corollary}
Every subsquare of the grid $\N\times\N\frac{1}{N}$ is sent by $\sigma_1$ affinely by a diagonal matrix onto a subrectangle with rational coordinates.
\end{corollary}

\begin{remark}
\label{Shnirelman: sbaglio}
The previous lemma also tells us that the map $\sigma_1$ is piecewise affine, in particular there exists $N=2^{p_1}\in \N$ refinement of the grid such that $\sigma_1$ maps each $\kappa$ subsquare of the grid $\N\times\N \frac{1}{N}$ affinely by a diagonal matrix onto a subrectangle $q$ with rational coordinates. It's false, in general, that $q$ has dyadic rational coordinates as stated in \cite[Lemma 4.3]{Shnirelman}. More precisely,
the previous lemma states that each 
\begin{equation*}
    P_{ij}=[x_{1,(i-1)j},x_{1,ij}]\times [j-1,j]2^{-p_0}
\end{equation*}
is sent into
\begin{equation*}
    \tilde{P}_{ij}=[i-1,i]2^{-p_0}\times [z_{2,i(j-1)},z_{2,ij}],
\end{equation*}
where $x_{1,(i-1)j},x_{1,ij},z_{2,i(j-1)},z_{2,ij}$ are dyadic. Call $\Delta x= x_{1,ij}-x_{1,(i-1)j}$ and $\Delta z=z_{2,ij}-z_{2,i(j-1)}$. Then by \eqref{eq:deltax} $\Delta x= 2^{-p_1}n_{ij}$ with $n_{ij}\in\N$. Up to translation the perturbed map $\sigma_{1}$ \eqref{Equa>pertur_form} can be written as
\begin{equation*}
    \sigma_{1}\llcorner_{P_{ij}}=\left(
    \begin{matrix}
    \frac{2^{-p_0}}{\Delta x} & 0 \\ 0 & 2^{p_0}(\Delta z)
    \end{matrix}\right).
\end{equation*}
Take a subsquare  $\kappa=[h-1,h]2^{-p_1}\times[k-1,k]2^{-p_1}\subset P_{ij}$, then $q=\sigma_{1}\llcorner_{P_{ij}}(\kappa)=\left[\frac{2^{-p_0}}{n_{ij
}},2^{p_0-p_1}(\Delta z)\right]$, which is dyadic only with further requirements on $n_{ij}$. For a more detailed analysis consider $H_j$ and call $\Delta x_i=x_{1,ij}-x_{1,(i-1)j}=2^{-p_1}n_{ij}$, where $i=1, \dots 2^{p_0}$. If we assume that every subsquare of the grid $\N\times\N\frac{1}{N}$ is sent into a dyadic rectangle then we find the conditions
\begin{equation*}
n_{ij}=2^{m_{ij}}, \quad\forall i,j.
\end{equation*}
 This condition tells us that, being measure-preserving,
 \begin{equation*}
    \sigma_{1}\llcorner_{P_{ij}}=\left(
    \begin{matrix}
   2^{-p_0+p_1-m_{ij}} & 0 \\ 0 & 2^{p_0-p_1+m_{ij}}
    \end{matrix}\right),
\end{equation*}
that is, all possible matrices are of the form 
\begin{equation*}
    \left(\begin{matrix}
   1 & 0 \\ 0 & 1
    \end{matrix}\right) \quad
    \left(\begin{matrix}
   \frac{1}{2} & 0 \\ 0 & 2
    \end{matrix}\right) \quad
    \left(\begin{matrix}
   2 & 0 \\ 0 & \frac{1}{2}
    \end{matrix}\right) \quad
    \left(\begin{matrix}
   \frac{1}{4} & 0 \\ 0 & 4
    \end{matrix}\right)\quad \dots \quad.
\end{equation*}
This condition is not compatible with the fact that $\sigma_{1}$ is an approximation of the original map $T$, which has been chosen to be close to the identity.
\end{remark}
\begin{remark}
\label{union:subrectangles}
From the previous lemma it easily follows that, if $N$ is the size of the grid, then every rectangle contained in the unit square $K$ is sent by the perturbed flow into a union of rectangles.  
\end{remark}
\begin{remark}
To use Theorem \ref{change of variables} we observe that in our case the change of variables $\phi$ is given by the flow $X_t$. In particular, since $X_t$ is close to the identity with all its derivatives, the costant $C_{X_t}$  given by the previous theorem, is $C_{X_t}\leq (1+\delta)^{d-1}$ ($d=2$ here).
\end{remark}

{

\subsubsection{The d-dimensional case}
\label{Sss:d-dim}

The analysis of the general case can be done as follows. 

The starting point is the following approximation assumption in $d-1$-dimension. 

\begin{assumption}
\label{Assu:recurs_d-1}
If 
the $\mathcal L^{d-1}$-measure-preserving diffeomorphism $T : [0,1]^{d-1} \rightarrow [0,1]^{d-1}$ is sufficiently close to the identity and equal to $\Id$ in a neighborhood of $\partial [0,1]^{d-1}$, then 
there exists $N\in\N$, $N=2^{p}$, and a measure-preserving piecewise smooth invertible map $\sigma$ close to $T$ such that $T \circ \sigma$ maps dyadic rectangles $P_{ij} \in \frac{\N^{d-1}}{N}$ onto dyadic rectangles $\tilde P_{ij} \in \frac{\N^{d-1}}{N}$ by a diagonal linear map (up to a translation). 
\end{assumption}

The above assumption is true for $d = 3$: indeed if $\sigma_t$ is the map of Lemma \ref{lem: dyadic}, then $\sigma = T^{-1} \circ \sigma_{t=1}$ does the job.

%
%
%

Now let $T : [0,1]^d \to [0,1]^d$ be a diffeomorphism sufficiently close to the identity and equal to the identity near $\partial [0,1]^d$ (Figure \ref{Fig:d-dim_case}). We will not address the perturbation $\zeta$ used to obtain dyadic parallelepipeds (Step 0 of the proof above), being the idea completely similar to the $2d$-case. We will also neglect the time dependence (i.e. how to split $t \in [0,1]$ into time intervals where the different maps are acting), because it is a fairly easy extension of the $2d$ case.

\smallskip

\noindent{ \it Step 1.}
Consider the curves
\begin{equation*}
z_d \mapsto T^{-1}(z_1,\dots,z_d).
\end{equation*}
The first step is to perturb $T$ to a map $T'$ in order to have that the above curves are segments along the $x_d$-direction in each slab $x_d \in [k_d,k_d+1]/N$ (Figure \ref{Fig:d-dim_case_rectif}). \\
Being $T$ close to the identity, the surface $T^{-1}(\{z_{d-1} = \text{const.}\})$ is parameterized by $x_1,\dots,x_{d-2},x_d$, and then in each strip
\begin{equation*}
(x_1,\dots,x_{d-2}) = \text{const.}, \quad x_{d-1} \in [0,1], \ x_d \in [k_d,k_d+1] \frac{1}{N}
\end{equation*}
one can use the same measure preserving map $\xi_1$ defined in Step 1 of the proof of Lemma \ref{lem: dyadic} above to obtain a perturbation $\hat T = T \circ \xi$ such that
\begin{equation*}
\hat T^{-1}(\{z_{d-1} = \text{const.}\}) \cap \big\{ x_d \in [k_d,k_d+1]/N \big\}
\end{equation*}
is independent of $x_d$, in the sense that it is the graph of a function depending only on $x_1,\dots,x_{d-2}$ times the segment $x_d \in [k_d,k_d+1]/N$. \\
Disintegrate the Lebesgue measure $\mathcal L^d$ as
\begin{equation*}
    \mathcal L^d \llcorner_{\{x_d \in [k_d,k_{d}+1]/N\}} = \int \Big[ a(x_1,\dots,x_{d-2},z_{d-1}) dx_1 dx_{d-2} dx_d \Big] dz_{d-1},
\end{equation*}
according to the partition $\hat T^{-1}(z_{d-1} = \text{const})$ (the density $a$ does not depend on $x_d$ because the surfaces contains the segments along $x_d$), and consider the $2$-dimensional surfaces
\begin{equation*}
    \hat T^{-1}(z_{d-1} = \text{const}) \cap \{x_1,\dots,x_{d-3} = \text{const}\}.
\end{equation*}
We use the same map $\xi_1$ of Step 1 of the proof above to rectify the curves
\begin{equation*}
    \begin{split}
    z_{d} \mapsto f(x_1,\dots,x_{d-3},z_{d-2},z_{d-1};z_d) &= (\hat T)^{-1}(z_{d-2},z_{d-1} = \text{const}) \\
    & \quad \cap \{x_1,\dots,x_{d-3} = \text{const}\} \cap \{x_d \in [k_d,k_d+1]/N\}.
    \end{split}
\end{equation*}
The main difference w.r.t. the maps \eqref{Equa:aver_x_1}, \eqref{Equa:aver_pp} is that instead of the Lebesgue measure we use the density $a(x_1,\dots,x_{d-2},z_{d-1})$. Eventually, the composition of the two maps above gives a new map $\check T$ such that $(\check T)^{-1}(z_{d-2},z_{d-1} = \text{const})$ is a $(d-2)$-dimensional surface made of the graph of a function depending on $x_1,\dots,x_{d-3}$ times the segment $x_d \in k_d,k_d+1]/N$. \\
The argument is then repeated in the $d-2$-regions $(\check T)^{-1}(z_{d-2},z_{d-1} = \text{const})$ (i.e. disintegrate the Lebesgue measure and shift along the $x_{d-3}$ direction to rectify $(\check T)^{-1}(z_{d-3},\dots,z_{d-1} = \text{const}$), and so on until we obtain that
%
%
%
%
%
a new map $T'$ such that
\begin{equation*}
(T')^{-1}(\{z_1,\dots,z_{d-1} = \text{const.}\}) \cap \big\{ x_d \in [k_d,k_d+1]/N \big\}
\end{equation*}
is independent on $x_d$. This means that
%
lines along the $z_d$ are mapped back into $N$ segments of length $1/N$ along $x_d$.

\begin{figure}[p]
\centering
\def\svgwidth{\textwidth}
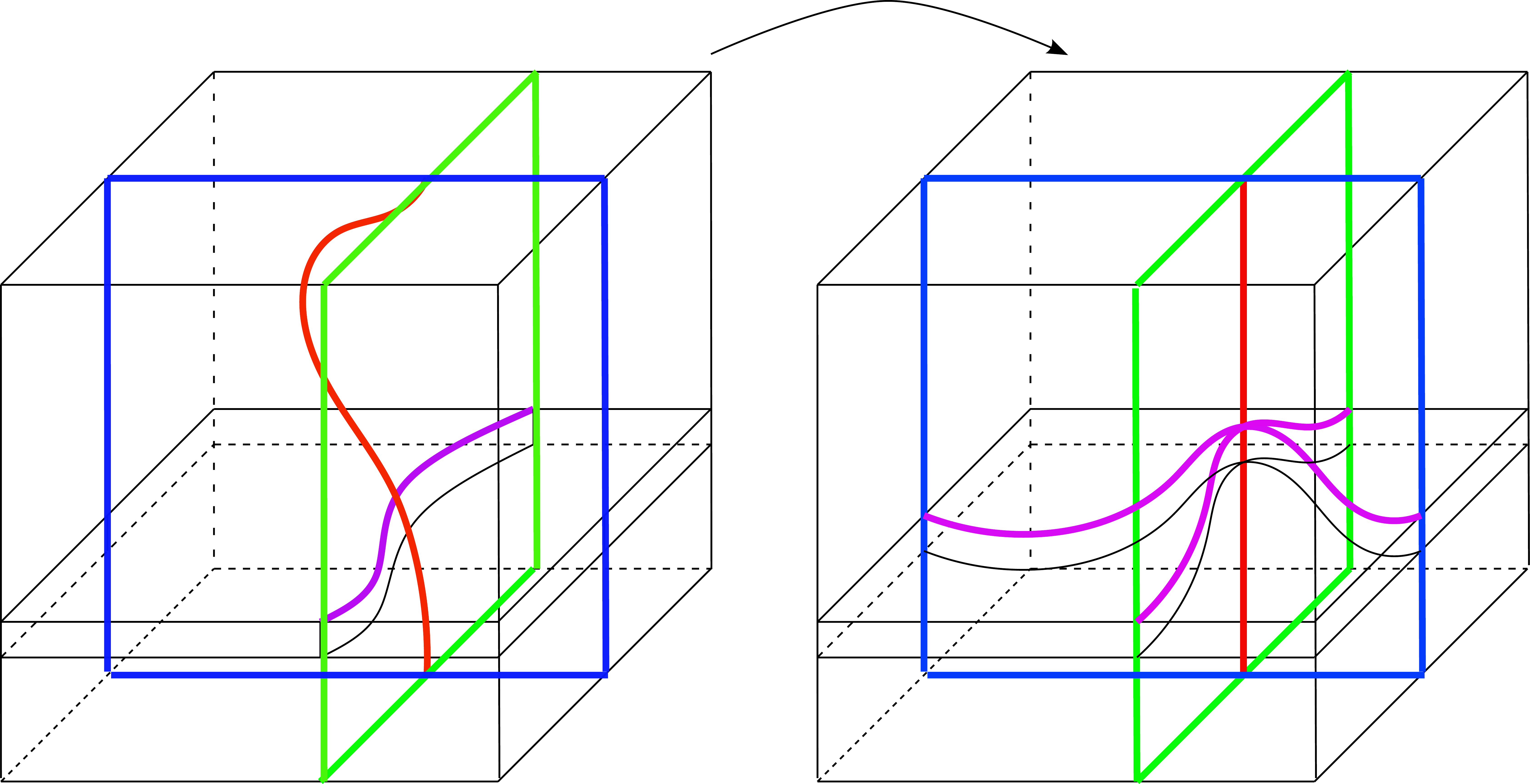
\caption{Starting point: the map $T : [0,1]^3 \to [0,1]^3$ maps the slap $x_3 \in [k_3,k_3+1]/N$ into a $3d$-set with purple intersections, and $T^{-1}(z_1,z_2)$ is the red curve at the left.}
\label{Fig:d-dim_case}
\end{figure}

\begin{figure}[p]
\centering
\def\svgwidth{\textwidth}
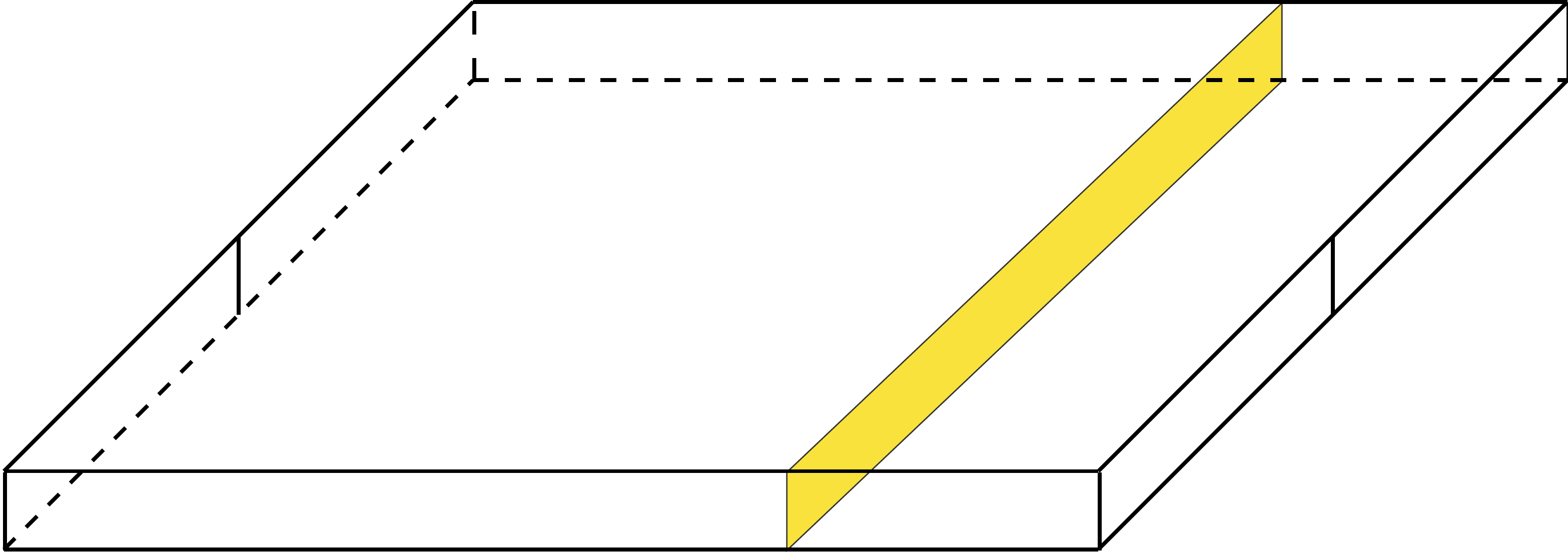
\caption{First move the mass in the yellow $2d$-rectangle $x_1 = \text{const}$ so that its intersection with $T^{-1}(z_2)$ is vertical in $x_3 \in [k_3,k_3+1]/N$, next move the mass along $T^{-1}(z_2)$ so that $T^{-1}(z_1,z_2)$ is vertical in $x_3 \in [k_3,k_3+1]/N$.}
\label{Fig:d-dim_case_rectif}
\end{figure}

\begin{figure}[p]
\centering
\def\svgwidth{\textwidth}
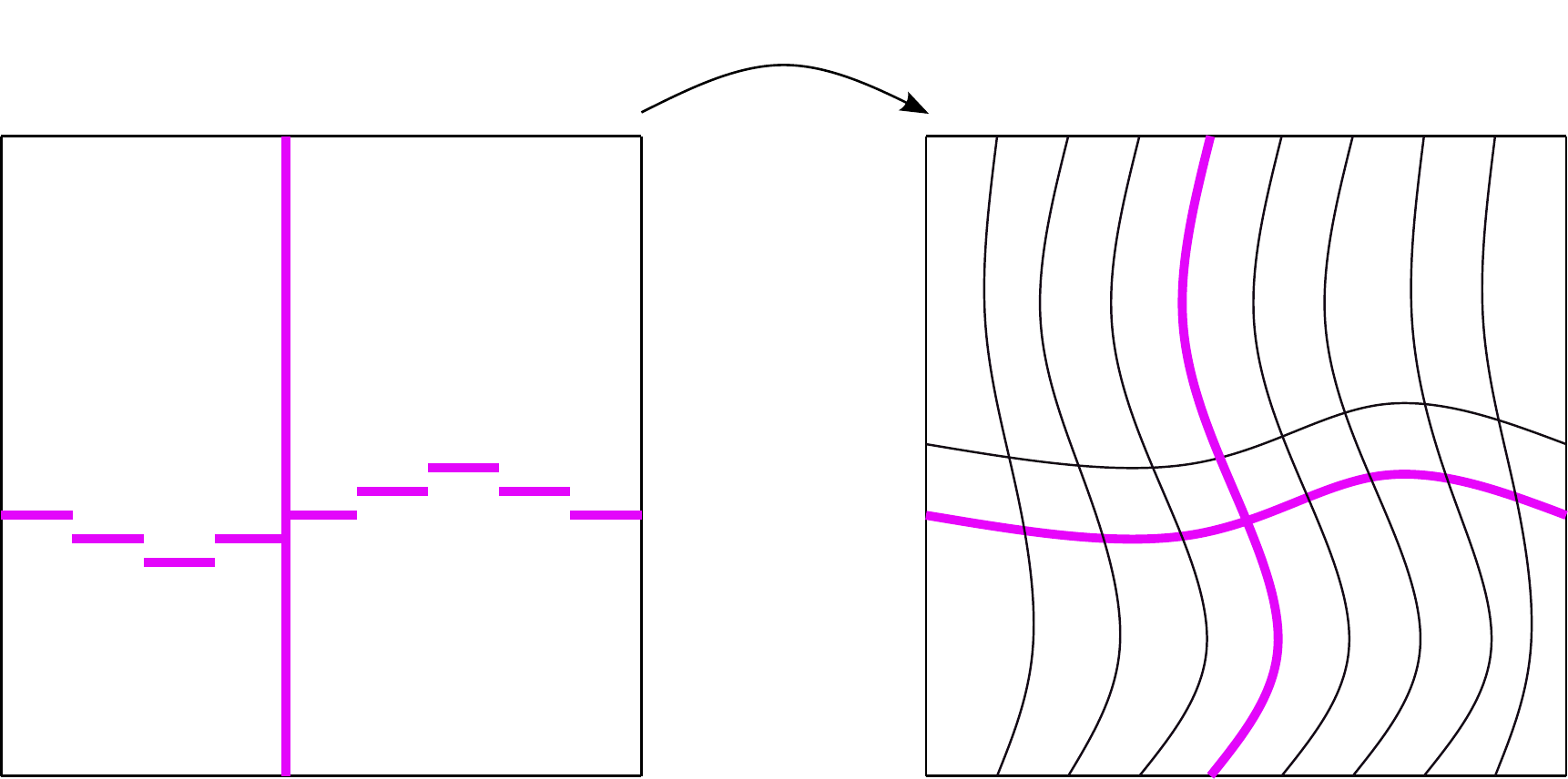
\caption{The recurrence assumption yields a map $\sigma$ which maps affinely subsquares into rectangles: in the picture it is shows how it acts before the composition with $T'$ (see also Figure \ref{fig:2}).}
\label{Fig:d-dim_case_recurrence}
\end{figure}

\begin{figure}[p]
\centering
\def\svgwidth{\textwidth}
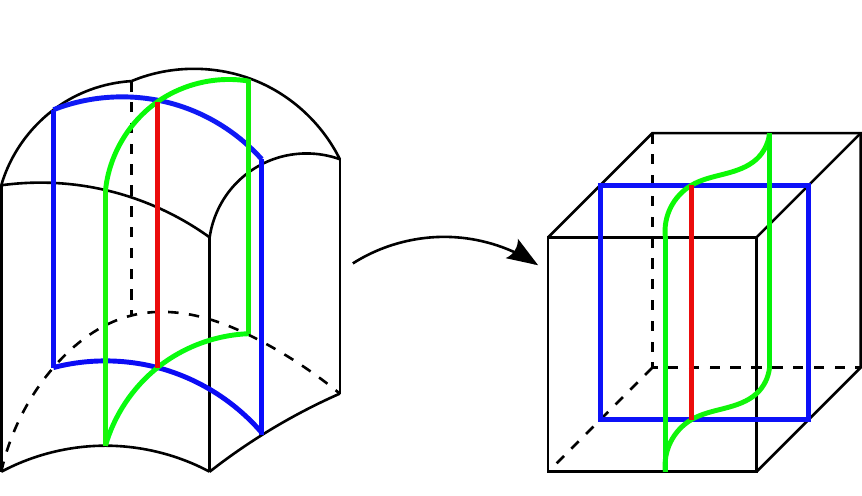
\caption{The last step is to map the subcubes deformed in the direction $x_3$ into parallelepipeds such that the Lebesgue measure is preserved and the map $G$ is of triangular form: these conditions imply that $\tilde T = G \circ \bar T$ is affine.}
\label{Fig:d-dim_case_laststep}
\end{figure}

\smallskip
{\it Step 2.}
Differently from the $2d$ case, it is not enough to perturb the vertical slab as in Step 2, since the sets $(T')^{-1}(\{z_k = \text{const.}\})$ are not of (piecewise) the form $\{x_k = \text{const.}\}$. Observe however that in each slab $\{x_d \in [k_d,k_d+1]/N\}$ the map
$$
(x_1,\dots,x_{d-1}) = (T'_{k_d})^{-1}(z_1,\dots,z_{d-1})
$$
is well defined, where $(T')^{-1}_{k_d}$ denotes the first $(d-1)$-components of $(T')^{-1}$ restricted to $\{x_d \in [k_d,k_d+1]/N\}$: we have used the property that segments along $z_d$ are mapped back into segments along $x_d$.

We use Assumption \ref{Assu:recurs_d-1} to get a map $\sigma_{k_d} : [0,1]^{d-1} \to [0,1]^{d-1}$ such that $T''_{k_d} = T'_{k_d} \circ \sigma_{k_d}$ 
maps affinely parallepipeds of a grid $\N^{d-1}/(NN_1)$, $N_1 = 2^{p_1}$, into cubes of the same grid: we can take $N_1 \gg 1$ in order to be independent of $k$ (Figure \ref{Fig:d-dim_case_recurrence}). \\
Hence the map $T''_{k_d} $ maps parallelepipeds of the form
\begin{equation*}
\prod_{i=1}^{d-1} \frac{[k_i,k_i+1]}{NN_1} \times \frac{[k_d,k_d+1]}{N}
\end{equation*}
into regions for the form
\begin{equation*}
\bigg\{ z_d \in \big[ g(x_1,\dots,x_{d-1},k_d/N),g(x_1,\dots,x_{d-1},(k_d+1)/N) \big], z_i \in \prod_{i=1}^{d-1} \frac{[k'_i,k'_i+1]}{NN_1} \bigg\},
\end{equation*}
and up to a translation it is a linear diagonal map in the first $d-1$ coordinates and segments along $x_d$ remains along $x_d$.

\smallskip
{\it Step 3.} Piecing together the maps $T''_{k_d}$, we obtain a measure preserving map $\bar T$ close to $T$ with the properties listed at the end of the previous step. We will use the fact that it is affine in the first $(d-1)$ coordinates to use a map similar $\eta_1$ of Step 2 of the proof of Lemma \ref{lem: dyadic} to rectify the set
\begin{equation*}
\bar T(\{x_d \in [k_d,k_d+1]/N\}) \cap \big\{ z_j \in [k'_j,k'_j+1]/(NN_1), j = 1,\dots,d-1 \big\}.
\end{equation*}
It is defined as the unique measure preserving map $G(z_1,\dots,z_{d-1})$ of the form
\begin{equation*}
    G(z_1,\dots,z_{d-1}) = \Big( G_1(z_1,\dots,z_{d-1}), G_2(z_2,\dots,z_{d-1}),\dots,G_{d-1}(z_{d-1}),G(z_1,\dots,z_d) \Big).
\end{equation*}
Note that since $z_d$ enters in the last component, segments along $z_d$ are mapped into segments along $z_d$, and the triangular form of the map assures its uniqueness (Figure \ref{Fig:d-dim_case_laststep}).
%
%
%
%
%
%
%
%

The last part of the analysis is to deduce that if a measure preserving transformation $\tilde T = G \circ \bar T  : [0,1]^d \to [0,1]^d$ is such that $\tilde T$ is of triangular form then it is the identity: we have rescaled every rectangle to a cube by linear scaling. \\ 
If the map has this triangular form, we conclude that the measure preserving condition reads as
\begin{equation*}
\prod_{i=1}^d \partial_i \tilde T_i(x_i,\dots,x_d) = 1,
\end{equation*}
which together with
\begin{equation*}
\int \partial_i \tilde T_i(x_i,\dots,x_d) dx_i = 1
\end{equation*}
gives $\partial_i T_i(x_i,\dots,x_d) = 1$, i.e. that the map is the identity.

}

\subsection{BV estimates of perturbations}
\label{Ss>BV_pertu}

Let $X_t:K\rightarrow K$ be a smooth flow of measure-preserving diffeomorphisms and assume 
	\begin{equation}
	\label{Equa:kappa_def}
	\|X_t - \Id\|_{C^3},\|X^{-1}_t - \Id\|_{C^3} \leq \wp, 
	\end{equation}
	with $\wp \ll 1$. 
Call $T(x)=X_{t=1}(x)$, and let $N=2^{p_0}\in\N$ be the dimension of the grid given by Lemma \ref{lem: dyadic}. 
In this section we compute the BV norms of the perturbations of the form \eqref{Equa>pertur_form} constructed in Lemma \ref{lem: dyadic}. 

We first address the action of the map $\zeta_t$ on $X(t)$. Define the perturbed flow
\begin{equation*}
t \mapsto X'_t(x) = \zeta_t \circ X_t(x).
\end{equation*}

\begin{lemma}
\label{Lem:zeta_map_X}
There exists a perturbation $\zeta_t$ as required by Step 0 of Lemma \ref{lem: dyadic} such that if $v$ is its associated vector field then
\begin{equation*}
    \|\zeta_t - \Id\|_{C_0} + 2^{-p_0} \|\nabla \zeta_t - \Id\|_{C_0} + 2^{-2p_0} \|\nabla^2 \zeta_t\|_{C_0} \leq \mathcal O(1) 2^{-p_1},
\end{equation*}
\begin{equation*}
    \|v\|_{C_0} + 2^{-p_0} \|\nabla v\|_{C_0} + 2^{-2p_0} \|\nabla^2 v\|_{C^0} \leq \mathcal O(1) 2^{-p_1}.
\end{equation*}
for $p_1$ sufficiently large.
\end{lemma}

\begin{proof}
The request of Step 0 of Lemma \ref{lem: dyadic} are that the flow across the each region $\tilde \omega_{ij}$ is $0$ (plus the dyadic condition on the new region $\tilde \omega_{ij}$). Hence the problem reduces in finding a suitable incompressible flow with a given boundary flux: we will construct a flow generated by a vector field constant in time.

Consider a function $v_n$ on $\partial \omega_{ij}$ such that
\begin{itemize}
    \item its support is at distance $2^{-p_0-2}$ from the corners of $\tilde \omega_{ij}$,
    \item the integral on each of the regular sides is the required flux $\phi_{ij,i'j'}$,
    \item $\|v_n\|_\infty, 2^{-p_0}\|v_n'\|, 2^{-2p_0}\|v_n''\| \leq \mathcal O(1) 2^{-p_1}$, where $v'_n$ is the derivative of $v_n$.
\end{itemize}
Its existence follows from the fact that $\omega_{ij}$ is close to a square of side $2^{-p_0}$, being $X$ close to the identity. The last point is a consequence of the fact that $|\phi_{ij,i'j'}| \leq 2^{-p_0-p_1}$.

The integral of $v$ on $\partial \tilde \omega_{ij}$ is a potential function $p$, which is constant in the $2^{-p_0-2}$-neighborhood of every corner and such that
\begin{equation*}
    \|p\|_{C_0}, 2^{-p_0} \|p'\|_{C_0}, 2^{-2p_0} \|p''\|_{C_0}, 2^{-3p_0} \|p'''\|_{C_0} \leq \mathcal O(1) 2^{-p_0-p_1},
\end{equation*}
where $p',p''$ are its first and second derivative.

Extend $p$ to a $C^2$-function inside $\tilde \omega_{ij}$: since this extension can be required to vary in a region of size $2^{-p_0-2}$, we get
\begin{equation*}
    \|p\|_{C_0} \leq \mathcal O(1) 2^{-p_0-p_1}, \quad \|\nabla p\|_{C_0} \leq \mathcal O(1) \frac{2^{-p_0-p_1}}{2^{-p_0-2}} = \mathcal O(1) 2^{-p_1},
\end{equation*}
\begin{equation*}
     \|\nabla^2 p\|_{C_0} \leq \mathcal O(1) \frac{2^{-p_0-p_1}}{2^{-2p_0-4}} = \mathcal O(1) 2^{p_0-p_1}.
\end{equation*}
\begin{equation*}
     \|\nabla^3 p\|_{C_0} \leq \mathcal O(1) \frac{2^{-p_0-p_1}}{2^{-3p_0-6}} = \mathcal O(1) 2^{2p_0-p_1}.
\end{equation*}
In particular the vector field $v = \nabla^\perp p$ satisfies the statement, and if $\zeta_t$ is the flow generated by $v$ then the same holds by the estimates
\begin{equation*}
    \|\zeta_t - \Id\|_{C^0} \leq \|v\|_{C_0} t, \quad \|\nabla \zeta_t - \Id\|_{C^0} \leq e^{\|\nabla v\|_{C_0} t} - 1,
\end{equation*}
\begin{equation*}
    \|\nabla^2 \zeta(t)\|_{C^0} \leq e^{\|\nabla v\|_{C_0} t} \|\nabla^2 v\|_{C_0} \|\nabla \zeta\|_{C_0}^2,
\end{equation*}
with $p_1 \gg 1$.
\end{proof}

\begin{corollary}
\label{cor:xi_prime_TV}
If $b'$ is the vector field associated with $X'_t = \zeta_t \circ X_t$, then
\begin{equation*}
    \|b'_t - b_t\|_{C_1} \leq \mathcal O(1) 2^{2p_0-p_1}.
\end{equation*}
\end{corollary}

\begin{proof}
From the formula \eqref{Equa:comp_ff}
\begin{equation*}
    b'(t,x) - b(t,x) = v(x) + \big( \nabla \zeta_t(t,\zeta_t^{-1}(x)) - \Id \big) b(t,x),
\end{equation*}
where $v(x)$ is the time independent vector field associated with $\zeta_t$. Hence from the previous lemma
\begin{equation*}
\begin{split}
    \|b'(t) - b(t)\|_{C^1} &\leq \|v\|_{C^1} + \mathcal O(1) \| \zeta_t - \Id\|_{C^2} \|b\|_{C^1} \\
    &\leq \mathcal O(1) 2^{2p_0-p_1}. \qedhere
\end{split}
\end{equation*}
\end{proof}

Define the \textbf{perturbed flow}  $t\rightarrow \tilde{X}_t$ as  
	\begin{equation*}
	\tilde{X}_t(x)=\begin{cases}
	X'(t,0,\xi_t(x)) & t\in\left[0,\frac{1}{2}\right], \\ X'(t,1,\eta(t,1,w(x))) & t\in\left[\frac{1}{2},1\right],
	\end{cases}
	\end{equation*}
	where $\xi_t$ and $\eta_t$ are given by formula \eqref{Equa>pertur_form} of Lemma \ref{lem: dyadic} (here since the map $\zeta$ is not needed we rescale $\xi_t,\eta_t$ with $t \in [0,1/2]$) and
	\begin{equation*}
	w(x)=\eta_1\circ T'\circ\xi_{\frac{1}{2}}(x), \quad T' = \zeta_1 \circ T.
	\end{equation*}
Call $\tilde{b}_t$ the vector field associated with $\tilde{X}_t$.

\begin{lemma}[BV estimates]
\label{BV:estim:pert:flow}
	There exists a positive constant $C = C(\wp)$ such that
	\begin{equation}
	\label{Equa:tilde_b_b}
	\|b-\tilde{b}\|_{L^\infty({L^1})} \leq \frac{C \wp}{N},\quad \|\TV({\tilde{b}})\|_{\infty}\leq C \|\TV(b)\|_\infty + \frac{C\wp}{N}. 
	\end{equation}
\end{lemma}



\begin{proof}
From Corollary \ref{cor:xi_prime_TV}, we have that (for $p_1 \gg 1$)
\begin{equation*}
    \|b-b'\|_{C^1} \leq \mathcal O(1) 2^{2p_0 - p_1} \ll \frac{\wp}{N},
\end{equation*}
so that we are left to prove \eqref{Equa:tilde_b_b} with $b'$ in place of $b$:
\begin{equation*}
	\|b'-\tilde{b}\|_{L^\infty({L^1})} \leq \frac{C \wp}{N},\quad \|\TV({\tilde{b}_t})\|_{\infty}\leq C \wp. 
	\end{equation*}
We will prove the above estimates for $t \in [0,1/2]$, i.e. only for $X'_t \circ \xi_t$, being the analysis of $X'(t,1,\eta(t,1,w(x)))$ completely analogous.

	We start by observing that there exists a constant $C>0$  such that 
	\begin{equation*}
	\|\dot{\xi_t}\|_{\infty}\leq\frac{C \wp}{N}. 
	\end{equation*} 
	Indeed, the map $\xi^i_t=\xi_t\llcorner _{H_i}$ is given by the formulas (see Step 1 of the proof of Lemma \ref{lem: dyadic})
	\begin{equation*}
	    \xi^i_{1,t}(x_1,x_2) = f'_{i,2t}(x_1,\xi^i_t(x_1,x_2)) = (1-2t) x_1 + 2t f'(z_{1,i}(x_1),\xi^i_{2,t}(x_1,x_2)),
	\end{equation*}
	\begin{equation*}
	    \int_{(i-1)2^{-p_0}}^{\xi^i_{2,t}(x_1,x_2)} \partial_{x_1} f'_{i,2t}(x_1,w) dw = x_2 - (i-1) 2^{-p_0}.
	\end{equation*}
	Since it holds by \eqref{Equa:kappa_def}
	\begin{equation}
\label{Equa:estim_f_prime_z_1i}
	    \|f' \llcorner_{H_i} - \Id\|_{C^3}, \|z_{1,i} - \Id\|_{C^3} \leq \mathcal O(\wp),
	\end{equation}
	then 
	\begin{equation*}
	   \| f'_{i,2t}(x_1,w) - x_1 \|_{C_3} \leq \mathcal O(\wp).
	\end{equation*}
	Hence the function
	\begin{equation*}
	    F(t,x_1,x_2,\xi) = \int_{(i-1)2^{-p_0}}^{\xi_2} \partial_{x_1} f'_{i,2t}(x_1,w) dw - x_2 + (i-1) 2^{-p_0}
	\end{equation*}
	satisfies
	\begin{equation*}
\begin{split}
	    \|F(t,x_1,x_2,\xi) - (\xi_2 -  x_2) \|_{C^2} &= \bigg\| \int_{(i-1)2^{-p_0}}^{\xi_2} \big( \partial_{x_1} f'_{i,2t}(x_1,w) - 1 \big) dw \bigg\|_{C_2} \leq \frac{\mathcal O(\wp)}{N}.
\end{split}
	\end{equation*}
	By the Implicit Function Theorem we deduce that
	\begin{equation*}
	    \|\xi^i_{2,t} - x_2\|_{C^2} \leq \frac{\mathcal O(\wp)}{N}, 
	\end{equation*}
and in particular
\begin{equation*}
\|\dot \xi^i_{2,t}\|_{C^0}, \|\nabla \dot \xi^i_{2,t}\|_{C^0} \leq \frac{\mathcal O(\wp)}{N}.
\end{equation*}
Similarly
\begin{equation*}
	    \|\xi^i_{1,t} - x_1\|_{C^2} \leq \frac{\mathcal O(\wp)}{N}, 
\end{equation*}
and then
\begin{equation*}
\|\dot \xi^i_{1,t}\|_{C^0}, \|\nabla \dot \xi^i_{1,t}\|_{C^0} \leq \frac{\mathcal O(\wp)}{N}.
\end{equation*}

We next estimate the total variation of the vector field
\begin{equation*}
v_t = \dot \xi_t(\xi^{-1}_t(x)).
\end{equation*}
We will use the following elementary formulas:
 \begin{equation}\label{4:14}
        \partial_{z_1}f'(z_1,x_2)=\frac{1}{\partial_{z_2} X_2^{-1} (z_1,z_2(z_1,x_2))} \end{equation}
        \begin{equation}\label{4:15}
        \partial_{x_2}f'(z_1,x_2)=\frac{\partial_{z_2}(X_1)^{-1}(z_1,z_2(z_1,x_2))}{\partial_{z_2}(X_2)^{-1}(z_1,z_2(z_1,x_2))}, \quad X_2(z_1,z_2(z_1,x_2)) = x_2,
    \end{equation}
 \begin{equation*}
        \partial_{x_1}z_1=\frac{1}{\dashint_{{(i-1)}2^{-p_0}}^{{i}2^{-p_0}} \partial_{z_1}f'({z}_{1,i}(x_1),w)dw}\comme{,}
    \end{equation*}
    from which it follows
\begin{equation*}
\begin{split}
     &\sum_i ||\partial_{z_1}f'(z_{i,1},x_2)\partial_{x_1}z_{i,1}-1||_{L^1(H_i)} + \sum_i \|\partial_{x_2} f'(z_{i,1},x_2)\|_{L^1(H_i)} \leq C||\TV(b)(K)||_\infty.
\end{split}
\end{equation*}
Indeed
\begin{equation*}
\begin{split}
    ||\partial_{z_1}f'(z_{i,1},x_2)\partial_{x_1}z_{i,1}-1||_{L^1(H_i)} & \leq ||\partial_{x_1}z_{i,1}||_\infty||\partial_{z_{1}}f'(z_{i,1},x_2)-1||_{L^1(H_i)} \\
    & \quad +||\partial_{x_1}z_{i,1}-1||_{L^1(H_i)} \\
    & \leq C||\partial_{z_1}f'(z_{i,1},x_2)-1||_{L^1(H_i)} \\
    & \quad + C \bigg\| \dashint_{(i-1)2^{-p_0}}^{i2^{-p_0}}\partial_{z_1}f'(z_{i,1},x_2) dw - 1 \bigg\|_{L^1(H_i)} \\ & \leq C||\partial_{z_1}f'(z_{i,1},x_2)-1||_{L^1(H_i)},
\end{split}
\end{equation*}
    therefore, by \eqref{4:14}, we get
    \begin{equation*}
\begin{split}
&\sum_i||\partial_{z_1}f'(z_{i,1},x_2)-1||_{L^1(H_i)} + \sum_i \|\partial_{x_2} f'(z_{i,1},x_2)\|_{L^1(H_i)} \\
& \qquad \leq C||\nabla (X^{-1}-Id)||_1 \leq C\int_0^1 \TV(b_s)(K)ds \leq C||\TV(b)(K)||_\infty.
\end{split}    
\end{equation*}

By the Implicit Function Theorem we recover the following estimate for $|\nabla \dot{\xi}|$\comme{:}
\begin{equation*}
    \begin{split}
    |\nabla\dot{\xi}| &\leq C \bigg( 
    |\partial_{x_1} f'(z_{1,i}(x_1),\xi_2)-1| + \int_{(i-1)2^{-p_0}}^{\xi_2} |\partial^2_{x_1} f'(z_{1,i}(x_1),w)| dw \\
    & \quad \qquad + \|\nabla f'(z_{1,i}(x_1),\xi_2)\|_{C^1} \big( |\dot \xi| + |\nabla \xi| + |\dot \xi||\nabla \xi| \big) \bigg) \\ 
    &\leq C 
    |\partial_{x_1} f'(z_{1,i}(x_1),\xi_2)-1| 
    +\frac{\mathcal{O}(\wp)}{N}.
\end{split}
    \end{equation*}
Hence
\begin{equation*}
    \|\nabla \dot \xi\|_1 \leq C \|\TV(b)(K)||_\infty + \frac{\mathcal{O}(\wp)}{N}.
\end{equation*}
    
For the jump part, we estimate the vector $v_t$ at the boundaries of $H_i$: from the definition
\begin{equation*}
\xi_t(x_1,(i-1)2^{-p_0}) = (1-2t) x_1 + 2t f' \big( z_{1,i}(x_1),(i-1) 2^{-p_0} \big),
\end{equation*}
\begin{equation*}
\xi_t(x_1,i2^{-p_0}) = (1-2t) x_1 + 2t f' \big( z_{1,i}(x_1),i 2^{-p_0} \big),
\end{equation*}
We consider only the second one, being the analysis of the first completely similar. Differentiating $\xi_t(x_1,i2^{-p_0})$ w.r.t. $t$ and using the definition of $z_{1,i}(x_1)$ we have
\begin{equation*}
\begin{split}
        \dot \xi_t(x_1,i2^{-p_0}) &= 2 \big( f' \big( z_{1,i}(x_1),i 2^{-p_0} \big) - x_1 \big) \\
        &= 2 \bigg( f' \big( z_{1,i}(x_1),i 2^{-p_0} \big) - \fint_{(i-1)/N}^{i/N} f' \big( z_{1,i}(x_1),w \big) dw \bigg).
\end{split}
\end{equation*}
and then
\begin{equation*}
    \begin{split}
    |\dot\xi_t(x_1,i/N)| &\leq 2 \int_{(i-1)/N}^{i/N} \big| \partial_{x_2} f'(z_{1,i}(x_1),w) \big| dw \\
    &\leq C \int_{(i-1)/N}^{i/N} \big| \partial_{z_2}(X_1)^{-1}(z_{1,i}(x_1),z_2(z_{1,i}(x_1),w)) \big| dw.
    \end{split}
\end{equation*}
%
Hence, by \eqref{Equa:estim_f_prime_z_1i} and from the definition of $z_2(z_1,x_2)$ in \eqref{4:15} we have that
\begin{equation*}
    \big\| (z_{1,i}(x_1),z_2(z_{1,i}(x_1),x_2)) - \Id \big\|_{C^1} \leq \frac{\mathcal{O}(\wp)}{N},
\end{equation*}
so that using
\begin{equation*}
    \big\{ \big( z_{1,i}(x_1),z_2(z_{1,i}(x_1),x_2) \big), (x_1,x_2) \in H_i \big\} = X_1(H_i),
\end{equation*}
we have
\begin{equation}\label{Equa:slab_jump_i}
    \begin{split}
        &\int_0^1 |\dot\xi_t(x_1,i/N)| dx_1 \\
        &\leq C \int_0^1 \int_{(i-1)/N}^{i/N} \big| \partial_{z_2}(X_1)^{-1}(z_{1,i}(x_1),z_2(z_{1,i}(x_1),x_2)) \big| dx_1dx_2 \\
        &\leq C \int_{X_1(H_i)} \big| \partial_{z_2}(X_1)^{-1}(z_1,z_2) \big| dz.
    \end{split}
\end{equation}
Again, since the change of variable which associate $t,x_1$ with its position $w_1$ on the jump line $[0,1] \times \{i/N\}$ is given by
\begin{equation*}
    t,x_1 \mapsto w_1 = (1-2t) x_1 + 2t f'(z_{1,i}(x_1),i/N),
\end{equation*}
up to a constant $1 + \mathcal O(\wp)/N$ (again by \eqref{Equa:estim_f_prime_z_1i}) the first integral in \eqref{Equa:slab_jump_i} corresponds to the jump part of $\dot \xi(x_1,i/N)$ on $[0,1] \times \{i/N\}$ when extended to $0$ outside $H_i$.

The same estimate holds for the jump of $\dot \xi_t(x_1,(i-1)2^{-p_0})$ on $[0,1] \times \{(i-1)/N$.

We conclude that
\begin{equation*}
\begin{split}
\|D^\jump v\| &\leq C \sum_i \int_{X(H_i)} \big| \partial_{z_2}(X_1)^{-1}(z_1,z_2) \big| dz \leq C \|\TV(b)\|_\infty.
\end{split}
\end{equation*}
We thus deduce
\begin{equation*}
\TV(v_t) \leq  C \|\TV(b)||_\infty + \frac{\mathcal{O}(\wp)}{N}. 
\end{equation*}

%
Collecting all estimates we have: \\
\smallskip \noindent \textbf{$L^1$ estimate.} Fix $t\in\left[0,\frac{1}{2}\right]$.
		From \eqref{Equa:comp_ff} it follows that
		\begin{equation*}
		|\tilde{b}_t(x)-b_t(x)|\leq \|\dot{\xi}_t\|_{\infty}|\nabla X'_t((X')^{-1}_t(x)| \leq \frac{C \wp}{N}.
		\end{equation*}
		
		
		
\smallskip
		
\noindent \textbf{BV estimate.}	
Again from \eqref{Equa:comp_ff}
\begin{equation*}
    \TV(\tilde b_t - b_t) \leq \TV \big( \nabla X_t(X^{-1}_t(x)) \dot{\xi}_t(\tilde X^{-1}_t(x)) \big),
\end{equation*}
so that we have to compute the total variation of
		\begin{equation*}
		\nabla X_t(X^{-1}_t(x))\dot{\xi}_t(X^{-1}_t(x)).
		\end{equation*}
By using Theorem \ref{change of variables} we have
\begin{equation*}
\begin{split}
\TV \big( \nabla X_t(X^{-1}_t)\dot{\xi}_t(\tilde X^{-1}_t) \big) &\leq \Lip(X_t) \TV \big( \nabla X_t \dot \xi_t(\xi^{-1}_t) \big) \\
&\leq \Lip(X_t) \TV(\nabla X_t) \|\dot \xi_t\|_\infty \\
& \quad + \Lip(X_t) \|\nabla X_t\|_\infty \TV(\dot \xi_t \circ \xi^{-1}_t) \\
&= \Lip(X_t) \TV(\nabla X_t) \|\dot \xi_t\|_\infty \\
& \quad + \Lip(X_t) \|\nabla X_t\|_\infty \TV(v_t).
\end{split}
\end{equation*}
The first term can be estimated by
\begin{equation*}
\TV(\nabla X_t) \|\dot \xi_t\|_\infty \leq 
\frac{\mathcal O(\wp)}{N},
\end{equation*}
while the second term is controlled by
\begin{equation*}
\|\nabla X_t\|_\infty \TV(v_t) \leq 
C \|\TV(b)||_\infty + \frac{\mathcal{O}(\wp)}{N}.
\end{equation*}
This is the statement.
\end{proof}

\begin{remark}
\label{Rem:d-dim_BV}
The above estimates can be obtained also for the $d$-dimensional case, since the maps used in that case is a composition of maps of the $2d$ case: the estimates are completely similar (but a lot more complicated).
\end{remark}

\begin{remark}
\label{Rem:const_BV_C}
We notice here that the constant $C$ in front of the total variation is larger than $1$: this fact is one of the reasons why we need to work in the $G_\delta$-set $\mathcal U$.
\end{remark}

\subsection{BV estimates for rotations}
\label{Ss:BV_rot}

In this part we address the analysis of rotations: these are needed because the map of Lemma \ref{lem: dyadic} maps affinely subrectangles into subrectangles, while we need squares translated into squares.

\noindent The approach here differs from the one of \cite[Lemma 4.5]{Shnirelman}, because the rotations used in that paper have a BV norm which is not bounded by the area (Lemma \ref{lem:rotations}) and instead it depends on the size of the squares (actually it blows up when the size of the squares goes to $0$). \\

\noindent Let $N = 2^{p_0}, p_0 \in \N,$ and $\sigma_1:K\rightarrow K$ be respectively the dimension of the grid and the map given by Lemma \ref{lem: dyadic}. 

\begin{lemma}\label{lem:rotations}
	\noindent
	 There exists $M\in \N$ and a flow $\bar{R}_t:K\rightarrow K$ invertible, measure-preserving and piecewise smooth such that the map $\sigma_1\circ \bar{R}_1$ translates each subsquare of the grid $ \N\times\N\frac{1}{M}$ into a subsquare of the same grid, i.e. it is a {permutation} of squares.
\end{lemma}

In particular note that $\nabla (\sigma_1\circ \bar{R}_1)$ is the identity inside each subsquare $\kappa$.

\begin{proof}
	Fix $\kappa$ a subsquare of the grid $\N\times\N\frac{1}{N}$ and call $q=\sigma_1(\kappa)$ its image. Since $\sigma\llcorner\kappa$ is an affine measure-preserving map of diagonal form, then, up to translations,
	\begin{equation*}
			\sigma_1 (x)=\left(\begin{matrix}
				\lambda_1 & 0 \\ 0 &\lambda_2
			\end{matrix}\right)x,\qquad\forall x\in\kappa\comme{,}
	\end{equation*}
		where $\lambda_1,\lambda_2\in\Q_{>0}$ and $\lambda_1\lambda_2=1$. Therefore the rectangle $q$ has horizontal side of length $\frac{\lambda_1}{N}$ and vertical side of length $\frac{1}{\lambda_1N}$.
		Decompose now the square $\kappa$ into rectangles $R_{ij}$ with $i=1,\dots,\frac{1}{l_1}$ and $j=1,\dots,\frac{1}{l_2}$  with horizontal side of length $\frac{l_1}{N}$ and vertical side of length $\frac{l_2}{N}$. The numbers $\frac{1}{l_1},\frac{1}{l_2}\in\N$ are chosen such that $\lambda_1=\frac{l_2}{l_1}$, i.e. $\sigma_1(R_{ij})=R^{\perp}_{ij}$, where $R^{\perp}_{ij}$ is the rotated rectangle counterclockwise of an angle $\frac{\pi}{2}$. \\
		\noindent
		In each $R_{ij}$ we perform a rotation given by the flow 
		\begin{equation*}
			R^{ij}_t=\chi^{-1}\circ r_t\circ \chi,
		\end{equation*}
	where $\chi:R_{ij}\rightarrow K$ is the affine map, up to translation, sending  each $R_{ij}$ into the unit square $K$, namely
	\begin{equation*}
		\chi x=\left(\begin{matrix}
			\frac{N}{l_1} & 0 \\
			0 & \frac{N}{l_2}
		\end{matrix}\right)x, \quad\forall x\in R_{ij},
	\end{equation*}
	whereas $r_t:K\rightarrow K$ is the rotation flow \eqref{rot:flow:square}. Finally define $R_t:K\rightarrow K$ such that $R_t\llcorner_{R_{ij}}=R^{ij}_t$. This flow rotates the interior of each rectangle $R_{ij}$ by $\pi/2$ during the time evolution. \\
	\noindent
	
    Now we choose $M\in\N$ large enough to refine the grid $\N\times\N\frac{1}{N}$ in such a way that for every subsquare $\kappa\in \N\times\N\frac{1}{N}$, $\forall i,j$, each rectangle $R_{ij}\subset \kappa$ is the union of squares of the grid $\N\times\N\frac{1}{M}$ and each rectangle $q=\sigma_1(\kappa)$ is union of subsquares of $\N\times\N \frac{1}{M}$, which is possible since the vertices of the squares and rectangles we are considering are all rationals. \\ 
    \noindent
    We claim that the map $\sigma_1\circ R_t:K\rightarrow K$ is a flow of invertible, measure-preserving maps such that $\sigma_1\circ R_1:K\rightarrow K$ is a permutation of subsquares of size $\frac{1}{M}$ up to a rotation of $\pi/2$. 
    Indeed, fix $R_{ij}$ and assume that it contains $a_{ij}$ subsquares $\kappa^h_{ij}$ of size $\frac{1}{M}\times\frac{1}{M}$ along the horizontal side and $b_{ij}$ subsquares along the vertical one. The rotation $R_1$ stretches each square $\kappa_{ij}^h$  into a rectangle $r^{h}_{ij}$ whose size is $\frac{l_1}{Ml_2}\times\frac{l_2}{Ml_1}$, i.e. $\frac{1}{\lambda_1 M}\times\frac{\lambda_1}{M}$. Now it is clear that $\sigma_1(r^h_{ij})$ is a square of size $\frac{1}{M}\times\frac{1}{M}$.

The Jacobian of $R_1$ is 
\begin{equation*}
J R_1 = \left( \begin{array}{cc}
0 & -\lambda_2 \\ \lambda_1 & 0
\end{array} \right).
\end{equation*}
Thus the composition of the two maps acts in each square $r^h_{ij}$ as
\begin{equation*}
\left( \begin{array}{cc}
\lambda_1 & 0 \\ 0 & \lambda_2
\end{array} \right) \left( \begin{array}{cc}
0 & -\lambda_2 \\ \lambda_1 & 0
\end{array} \right) = \left( \begin{array}{cc}
0 & -1 \\ 1 & 0
\end{array} \right),
\end{equation*}
i.e. a rotation of $\pi/2$. 

Define then the map $\bar{R}_t$ as
\begin{equation*}
\bar{R}_t \llcorner_{\kappa^h_{ij}} = R_t \circ r^{-\pi/2}_t,
\end{equation*}
where $r^{-\pi/2}_{t}$ is the rotation of the square $\kappa^h_{ij}$ of $-\pi/2$: now the map has Jacobian $\Id$ in each square $\kappa^h_{ij}$. This is the map of the statement.
\end{proof}

{
\begin{remark}
\label{Rem:multi_d_rot}
A completely similar construction can be done in dimension $d \geq 3$: in this case, in each cube $\kappa \in \N^d/N$ the piecewise affine map $\sigma$ has the form (up to a translation)
\begin{equation*}
\sigma = \diag(\lambda_1,\dots,\lambda_d), \quad \lambda_1 \lambda_2 \dots \lambda_d = 1.
\end{equation*}
Hence the subpartition of $\kappa$ is done into parallelepipeds $\ell_1 \times \ell_2 \times \dots \times \ell_d$ such that
\begin{equation*}
\lambda_i = \frac{\ell_{i+1}}{\ell_i}.
\end{equation*}
The action of $\sigma$ transform each of these parallelepipeds into the new ones $\ell_2 \times \ell_3 \times \dots \times \ell_d \times \ell_1$, which is the range of the rotation
\begin{equation*}
\left(
\begin{array}{ccccc}
0 & -1 & 0 & \cdots & 0 \\
0 & 0 & -1 & \cdots & 0 \\
\vdots & \vdots & \vdots & \ddots & \vdots \\
1 & 0 & 0 & \dots & 0
\end{array}
\right) = \left(
\begin{array}{ccccc}
1 & 0 & \cdots & 0 & 0 \\
0 & 1 & \cdots & 0 & 0\\
\vdots & \vdots & \Id & \vdots & \vdots \\
0 & 0 & \cdots & 0 & -1 \\
0 & 0 & \dots & 1 & 0
\end{array}
\right) \cdots \dots \cdot \left(
\begin{array}{ccccc}
0 & -1 & \cdots & 0 & 0 \\
1 & 0 & \cdots & 0 & 0 \\
\vdots & \vdots & \Id & \ddots & \vdots \\
0 & 0 & \dots & 1 & 0 \\
0 & 0 & \dots & 0 & 1
\end{array}
\right).
\end{equation*}
(The above formula is the decomposition into $2d$ rotations.) Hence, as in Lemma \ref{lem:rotations} above, a rotation of the parallelpipeds $\ell_1 \times \dots \times \ell_d$ and a counter-rotation of the subcubes of the parallelepipeds gives the transformation.
\end{remark}
}

\subsection{Main approximation theorem}
\label{Ss:main_appro_th}

We are ready to prove our main result.  

\begin{theorem}
\label{Prop:piececlo}
Let $b\in L^\infty([0,1];BV(K))$ be a divergence-free vector field and assume that there exists $\delta>0$ such that for $\mathcal{L}^1$-a.e. $t\in[0,1]$, $\supp b_t \subset\subset K^\delta$. Then for every $\epsilon>0$ there exist $\delta', C_1,C_2>0$ positive constants, $D \in \N$ arbitrarily large and a divergence-free vector field $b^\epsilon\in L^\infty([0,1];BV(K))$ such that 
\begin{enumerate}
\item $\supp b^\epsilon_t\subset\subset K^{\delta'}$,
\item it holds
\begin{equation}
\label{main result}
    \|b - b^\epsilon\|_{L^\infty(L^1)} \leq \epsilon, \quad || \TV(b^\epsilon) (K)||_{\infty}\leq  C_1 ||\TV(b)(K)||_\infty + C_2,
\end{equation}
\item the map $X^\epsilon\llcorner_{t=1}$ generated by $b^\epsilon$ at time $t=1$ translates each subsquare of the grid $\N \times \N \frac{1}{D}$ into a subsquare of the same grid, i.e. it is a permutation of squares.
\end{enumerate}
\end{theorem}

\begin{remark}
\label{Rem:delta_0}
Observe that the theorem can be easily extended to vector fields $b \in L^\infty([0,1],\BV(\R^2))$ such that $\supp b_t \subset K$. We keep here the original setting of \cite{Shnirelman}.
\end{remark}

\begin{remark}
By inspection of the proof one can check that $C_1,C_2$ are independent of $b$. This is in any case not needed for the proof of the main theorem.
\end{remark}

\begin{remark}
\label{Rem:shnirel_blows}
A possible approach would be to divide the time interval $[0,1]$ into sufficiently small time steps  $\sim\tau$ in  order to apply  Lemmas \ref{lem: dyadic} , \ref{lem:rotations} and hence to compose the resulting maps as done in \cite{Shnirelman}, however by Lemma \ref{TV:rotations},
\begin{equation*}
    ||\TV(\bar{R})(\kappa)||_\infty\sim \frac{\text{Area}(K)}{\tau}
\end{equation*}
so that the total variation blows up as the time step goes to zero.
\end{remark}
\begin{proof}
We divide the proof into several steps. \\

\noindent {\it Step 1.} Let $\rho \in C^\infty_c(\R^2)$ be a mollifier, and define
\begin{equation*}
    b^{\alpha}_{t}\doteq b_t \ast \rho_\alpha,
\end{equation*}
where $\rho_\alpha(x)\doteq \alpha^{-2} \rho(\frac{x}{\alpha})$ and $\alpha << 1$ is chosen such that $\supp b_{\alpha,t}\subset\subset K$.


By well known estimates (see \eqref{est:translat}) we obtain
%
    \begin{equation*}
        \|b^\alpha_t-b_t\|_{L^1}\leq \alpha\TV(b_t)(K),
\quad         \TV(b^{\alpha}_{t})(K)\leq \TV(b_t)(K).
    \end{equation*}
   then if $\alpha$ is chosen such that $\alpha\leq \frac{\epsilon}{2||\TV(b)(K)||_\infty}$, we conclude that 
   \begin{equation*}
       ||b^\alpha_t -b_t||_1\leq \frac{\epsilon}{2}, \quad \TV(b^{\alpha}_{t})(K)\leq \TV(b_t)(K)
   \end{equation*} and we have to prove the theorem for $b^\alpha$. Moreover $b^\alpha$ satisfies the estimates 
   \begin{equation*}
       ||b^\alpha_t||_\infty\leq \frac{1}{\alpha^2}||b_t||_1, \quad \|\nabla^n b^\alpha_t\|_\infty \leq \frac{C_n}{\alpha^{1+n}} \|\TV(b_t)(K)\|_\infty.
   \end{equation*}
  From now on we will call $b^\alpha=b$ to avoid cumbersome notations. \\
  
%
%
%

  \noindent {\it Step 2.}  Let us consider a partition of the time interval $0=t_0<t_1<\dots<t_n=1$ where $n\in\N$ and $t_i=\frac{i}{n}$, where $n$ will be chosen later on. Let us call $X_{j}\doteq X(t_{j},t_{j-1},x)$ and $X_j(t)\doteq X(t,t_{j-1},x)$ defined for $t\in [t_{j-1},t_j]$. Then each flow $X_j(t)$ is close to the identity with its derivatives, indeed
  \begin{equation}
  \label{Equa:ODE_X}
  X_j(t,x)=x+\int_{t_{j-1}}^{t} b(s,X(s,t_{j-1},x))ds,
  \end{equation} 
  so that
  \begin{equation*}
      \|X_j(t) - \Id\|_{C^k} \leq C(k) (t-t_{j-1}) \|b\|_{C^k}.
  \end{equation*}
  More precisely, if $\wp$ is the constant of \eqref{Equa:kappa_def}, there exists $n\in\N$ such that
  \begin{equation*}
        \|X_j(t) - \Id\|_{C^3}, \|X_j^{-1} - \Id\|_{C^3} \leq \wp \\, \quad \forall t \in [t_{j-1},t_j], \quad\forall j=1,\dots,n.
    \end{equation*}
    Therefore we can apply Lemma \ref{lem: dyadic} to each $X_j(t)$ finding $N_j=2^{p_j}$ dyadic and $\tilde{X}_j:[t_{j-1},t_j]\times K\rightarrow K$ with the property that, at time $t=t_j$, the map $\tilde{X}_j(t_j)$ sends subsquares of the grid $\N\times\N \frac{1}{N_j}$ into rational rectangles with vertices in $\frac{\N}{N_j R_j}$. In particular, the eigenvalues of all affine maps $\sigma$ for $\tilde X_j(t_j)$ belongs to $\frac{\N}{R_j}$. We can moreover assume that $N_j = N$ for all maps $\tilde X_j(t_j)$ by taking $N$ sufficiently large. Finally from Lemma \ref{BV:estim:pert:flow} we have that  in each interval $[t_{j-1},t_j]$ it holds
    \begin{equation*}
	\|b-\tilde{b}_j\|_{L^\infty({L^1})} \leq \frac{C \wp}{N},\quad \|\TV({\tilde{b}_j})\|_{\infty}\leq C \|\TV(b_t)\|_\infty + \frac{C\wp}{N}. 
	\end{equation*}
so that for $N \gg 1$ we have
\begin{equation*}
    \|b - \tilde b_j\|_{L^\infty(L^1)} \leq \epsilon, \quad || \TV(\tilde b_j) (K)||_{\infty}\leq  C ||\TV(b)(K)||_\infty + \epsilon,
\end{equation*}
where $\tilde b_j$ is the vector field associated with $\tilde X_j$.

We define $t\rightarrow \tilde{X}_t$ the \emph{perturbed flow} 
    \begin{equation*}
        \tilde{X}(t) = \begin{cases}
        \tilde{X}_1(t) & 0\leq t\leq t_1, \\
        \tilde{X}_2(t) \circ \tilde{X}_1(t_1) & t_1\leq t\leq t_2, \\
        \dots \\
        \tilde{X}_{i+1}(t)\circ \tilde{X}_i(t_i)\circ\dots\circ \tilde{X}_1(t_1) & t_i\leq t\leq t_{i+1}, \\
        \dots \\
        \tilde{X}_{n}(t)\circ \tilde{X}_{n-1}(t_{n-1})\circ\dots\circ \tilde{X}_1(t_1) & t_{n-1}\leq t\leq 1.
        \end{cases}
    \end{equation*}
%
%
%
%
The map is clearly piecewise affine, and the eigenvalues of each affine piece $\sigma$ belong to $\frac{\N}{\prod_j R_j}$. \\

\noindent{\it Step 3.} The map $\tilde{X}(1)$ has the property of sending subsquares of the grid $\N\times\N\frac{1}{N}$ into union of rational rectangles. Let $D = N (\prod_j R_j)^2$: we now show that $\tilde X(1)$ maps subsquares of the grid $\N \times \N \frac{1}{D}$ into rational rectangles.

Let $R = \prod_j R_j$ and assume that the map $\tilde X(1,t_{j+1})$ maps the subsquares of the grid
$$
\N \times \N \frac{1}{N R \prod^n_{k = j+1} R_k}
$$
into rational rectangles. Since $\tilde X_j$ maps affinely the subsquares of $\N \times \N \frac{1}{N}$ into rectangles of the grid
$$
\N \times \N \frac{1}{N R_j} \subset \N \times \N \frac{1}{N R \prod^n_{k = j+1} R_k},
$$
then
\begin{equation*}
    \tilde X^{-1}_j \bigg( \N \times \N \frac{1}{N R \prod^n_{k=j+1} R_k} \bigg) \subset \N \times \N \frac{1}{N R \prod^n_{k=j} R_k}.
\end{equation*}
In particular we obtain that
\begin{equation*}
    \tilde X^{-1} \bigg( \N \times \N \frac{1}{N R} \bigg) \subset \N \times \N \frac{1}{N R^2}.
\end{equation*}

    

    We rename the flow $\tilde X^D_t$ to indicate the size of the grid on which it acts as a piecewise affine map. Note that the above estimates (as well as the next ones) improve whenever $N$ becomes larger, so that the size of the grid $D$ can be taken arbitrarily large.
    \begin{figure*}
        \centering
        \includegraphics[scale=0.6]{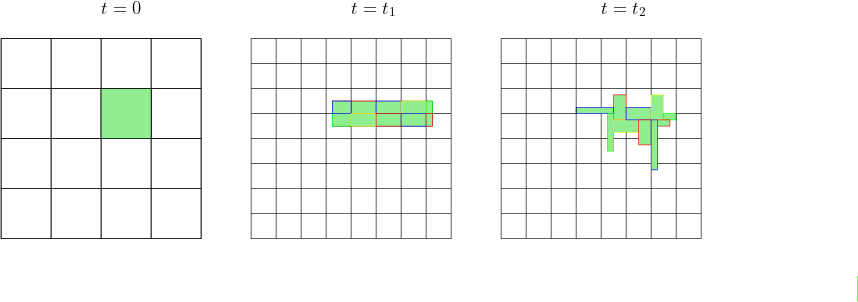}
        \caption{A square is sent into a union of rectangles by the map $\tilde{X}(1)$.}
    \end{figure*}


\smallskip

\noindent {\it Step 4.}  To conclude the proof we want to modify the flow $\tilde X^D_t$ slightly in such a way that the new flow $\check{X}^D_t$ evaluated at $t=1$ sends subsquares into subsquares by translations. The key idea is to perform rotations as in Lemma \ref{lem:rotations} balancing two effects: one one hand the cost of a rotation is at least of the order of the area (Lemma \ref{TV:rotations}), on the other hand if the squares are too much deformed the cost is exponentially large w.r.t. the total variation used to deform the square. The idea will be to use these rotations only when the deformation reaches a critical threshold.  \\

\noindent
First let us fix $\kappa_0$ a subsquare of the grid $\N\times\N\frac{1}{D}$ and call $\kappa_i$ its images through the maps $\kappa_i=\tilde {X}^D_{ t=t_i}(\kappa_0)$. Since each map $\tilde{X}_i(t_i)$ is affine and measure-preserving on $\kappa_{i-1}$, up to a translation it can be represented as
\begin{equation*}
	\left(\begin{matrix}
		\sigma_i & 0 \\
		0 & \frac{1}{\sigma_i}
	\end{matrix}\right)
\end{equation*}
where $\sigma_i\in\Q$ and $ |\sigma_{i}-1|,|\frac{1}{\sigma_{i}}-1|\leq \wp \ll 1$, where $\wp$ is the one given by the partition. 
Moreover, being
\begin{equation}
\label{equa5.36}
\nabla^2 \tilde X^D_{t=t_j}(x) = 0
\end{equation}
whenever $x$ belongs to the interior of the subsquares, we deduce that for the same $x$
\begin{equation*}
    |\nabla^2 \tilde X^D_t(x)| \leq \wp.
\end{equation*}
We can also observe that, by \eqref{Equa:ODE_X} 
\begin{equation*}
|\sigma_i - 1|, \bigg| \frac{1}{\sigma_i} - 1 \bigg| \leq \frac{C}{\mathcal L^2(\kappa_0)} \int_{t_{i-1}}^{t_i} \TV(b^D_s)(\kappa_0) ds.
\end{equation*}
By elementary computations one can prove that 
\begin{equation*}
\bigg|\prod_i \sigma_i - 1 \bigg|, \bigg| \prod_i \frac{1}{\sigma_i} - 1 \bigg| \leq \max \bigg\{ \sigma, \frac{1}{\sigma} \bigg\} \sum_i |\sigma_i -1|.
\end{equation*}
Hence if we have the bound 
\begin{equation*}
	3 \leq \left( \sigma_{j+1} \cdot \ldots \cdot \sigma_{i} + \frac{1}{ \sigma_{j+1} \cdot \ldots \cdot \sigma_{i}}\right)\leq 4,
\end{equation*}
then
\begin{equation}
	\label{form:princ}
	\frac{3}{8} \mathcal{L}^2(\kappa_j)\leq \int_{t_j}^{t_i} \TV(b^D_s)(\kappa_0) ds.
\end{equation}

\noindent    
The idea is to find now a new sequence of times $\lbrace t_{i_j}\rbrace\subset\lbrace t_{i}\rbrace$, $j=1,\dots,n'$ and $i=1,\dots,n$, in which we perform the rotations of Lemma \ref{TV:rotations} in order to have both the total variation is controlled be the total variation of $\tilde b^D$ and the property of sending {subsquares} into {subsquares} by translation. 

Let us start defining $t_{i_0}=0$ and
\begin{equation*}
	t_{i_1}= \min T_0,
\end{equation*}
where
$$
T_0=\left\lbrace t_{i}>0: 3\leq\left (\sigma_{1} \cdot \dots \cdot \sigma_{i} + \frac{1}{\sigma_{1} \cdot \dots \cdot \sigma_{i}}\right)\leq 4\right\rbrace.
$$
Then two situations may occur.
\begin{enumerate}
	\item The set $T_0$ is empty, that is $ \left(\sigma_{1}\dots\sigma_{i}+ \frac{1}{\sigma_{1}\dots\sigma_{i}}\right)\leq 3$ for all $i$. In this case we perform a rotation in $[0,1]$, that is $R^1:[0,1]\times K\rightarrow K$ (as in Lemma \ref{lem:rotations}) where $R^1\llcorner_{K\setminus\mathring{\kappa}_0}(x)=x$ and it is such that $X^D\circ R^1\llcorner_{t=1}$ sends subsquares of $\kappa_0$ into subsquares of $\kappa_n$. In this case 
	\begin{equation*}
		\hat{X}^D_t = \tilde X^D_t\circ R^1_t
	\end{equation*}
    \begin{equation*}
    	\hat{b}^D_t(x) = \tilde b^D_t(x) + \nabla \tilde X^D_t \circ (\tilde X^D_t)^{-1}(x) \dot R^1_t((\tilde X^D_t \circ R^1_t)^{-1}(x) ) \quad x \in \kappa_0
    \end{equation*}
    (where we have recalled that all functions can be extended smoothly to $\kappa_0$) and
	\begin{equation}
		\label{est:case:1}
	\begin{split}
	\TV(\hat{b}^D_t - \tilde b^D_t)(\kappa_0) &\leq \|\nabla \tilde X^D_t\|_\infty^2 TV(\dot{R}_t(R^{-1}_t))(\kappa_0) \\
	& \quad + \|\nabla \tilde X^D_t\|_\infty \|\dot R\|_\infty \TV(\nabla \tilde X^D_t)(\kappa_0) \\
	&\leq
	\|\nabla \tilde X^D_t\|_\infty \frac{4}{D^2}\left(\sigma_{1}\dots\sigma_{n}+\frac{1}{\sigma_{1}\dots\sigma_{n}}\right) + \frac{\mathcal O(\wp)}{D^3} \\& \leq
	\|\nabla \tilde X^D_t\|_\infty \frac{12}{D^2} + \frac{\mathcal O(\wp)}{D^3}.
	\end{split}
	\end{equation}
We have observed that
\begin{equation*}
    \nabla^2 \tilde X^D_{t=t_j}(x) = 0 \quad \text{for $x \in \mathring{\kappa}_0$ by (\ref{equa5.36}),}
\end{equation*}
so that for $t \in [t_{j-1},t_j]$
\begin{equation*}
    \|\nabla^2 \tilde X^D_t\|_\infty \leq \mathcal O(1) \int_{t_{j-1}}^t |\nabla^2 b(s,\tilde X^D_s)| ds = \frac{O(\wp)}{D^2}. 
\end{equation*}

Since $\|\nabla \tilde X^D_t\|_\infty \leq C$ by the assumptions that these are sets with small deformations, we obtain
\begin{equation*}
\TV(\hat b^D_t - \tilde b^D_t)(\kappa_0) \leq \frac{\mathcal O(1)}{D^2} = \mathcal O(1) \mathcal L^2(\kappa_0),
\end{equation*}
where we have used \eqref{form:princ}.

    \item The set $T_0$ is non empty. Then if $t_{i_1}=1$ the procedure stops and you perform a rotation as in Lemma \ref{lem:rotations} in $[0,1]$ finding 
	\begin{equation*}
	\begin{split}
	\TV(\hat{b}^D_t - \tilde b^D_t)(\kappa_0) &\leq  \|\nabla \tilde X^D_t\|\frac{4}{D^2}\left(\sigma_{1}\dots\sigma_{n}+\frac{1}{\sigma_{1}\dots\sigma_{n}}\right) + \frac{\mathcal O(\wp)}{D^3}\\
	& \leq
	\mathcal O(1) \|\nabla \tilde X^D_t\|\int_{0}^1\TV(\tilde b^D_s)({\kappa}_0)ds + \frac{\mathcal O(\wp)}{D^3},
    \end{split}
    \end{equation*}
    where we have used (\ref{form:princ}) and we have estimated the higher order term as in \eqref{est:case:1}. 
	\\
	
	\noindent
	If $t_{i_1}<1$ we compute
	\begin{equation*}
		t_{i_2}= \min T_1,
	\end{equation*}
	where
	$$
	T_1=\left\lbrace t_{i}>t_{i_1}: 3\leq\left(\sigma_{i_1+1}\dots\sigma_{i}+ \frac{1}{\sigma_{i_1+1}\dots\sigma_{i}}\right)\leq 4\right\rbrace.
	$$
	If $T_1=\emptyset$ we stop the procedure and we perform a rotation in $[t_{i_1-1},1] = [0,1]$ finding, for $t\in [t_{i_1-1},1]$ 
	\begin{equation}
\label{est:case:3}
\begin{split}			\TV(\hat{b}^D_t- \tilde b^D_t)(\kappa_0)&\leq \|\nabla \tilde X^D_t\|\frac{4}{D^2}\frac{1}{1-t_{i_0}}\left(\sigma_{i_1-1}\dots\sigma_{n}+\frac{1}{\sigma_{i_0}\dots\sigma_{n}}\right)\\
& \quad + \frac{\mathcal O(\wp)}{D^3} \\
& \leq  \mathcal{O}(1) 
\frac{1}{1-t_{i_0}}\int_{t_{i_0}}^{t_{i_1}}\TV(\tilde b^D_s)({\kappa}_0)ds + \frac{\mathcal O(\wp)}{D^3} \\ 
		&\leq  \mathcal O(1) \dashint_{t_{i_0}}^{1}\TV(\tilde b^D_s)({\kappa}_0)ds + \frac{\mathcal O(\wp)}{D^3}.
		\end{split}
	\end{equation}
	
	\noindent
	If $T_1$ is non empty we perform a rotation in $[0,t_{i_1}]$ finding for $t\in[0,t_{i_1}]$ the following estimate
	\begin{equation*}
		\begin{split}
		\TV(\hat{b}^D_t - \tilde b^D_t)(\kappa_0) &\leq \|\nabla \tilde X^D\|\frac{4}{D^2}\frac{1}{t_{i_1}-1}\left(\sigma_{1}\dots\sigma_{i_1}+\frac{1}{\sigma_{1}\dots\sigma_{i_1}}\right) + \frac{\mathcal O(\wp)}{D^3 t_{i_1}}\\ 
		&\leq 
		\mathcal{O}(1) \dashint_{0}^{t_{i_1}}\TV(\tilde b^D_s)({\kappa}_0)ds + \frac{\mathcal O(n \wp)}{D^3}.
		\end{split}
	\end{equation*}
	
	Again if $t_{i_2}=1$ the procedure stops and we perform another rotation in $[t_{i_1},t_{i_2}]$. If not we consider the set $T_2$ and we proceed.
\end{enumerate}
The general step if the following: we consider the set
$$
T_{j+1} = \left\lbrace t_{i}>t_{i_j}: 3\leq\left (\sigma_{1} \cdot \dots \cdot \sigma_{i} + \frac{1}{\sigma_{1} \cdot \dots \cdot \sigma_{i}}\right)\leq 4\right\rbrace.
$$
If $T_{j+1}$ is empty, then we perform a rotation in $[t_{i_{j-1}},1]$ finding for $t \in [t_{i_{j-1}},1]$ the same estimate as (\ref{est:case:3}). If $T_{j+1}$ is non empty, we perform a rotation in $[t_{i_{j-1}},t_{i_j}]$ and we consider 
$$
T_{j+2} = \left\lbrace t_{i}>t_{i_{j+1}}: 3\leq\left (\sigma_{1} \cdot \dots \cdot \sigma_{i} + \frac{1}{\sigma_{1} \cdot \dots \cdot \sigma_{i}}\right)\leq 4\right\rbrace.
$$
\\

\noindent
At the end of this procedure there are two possible scenarios: let $n'=\sup \lbrace j: T_j\not=\emptyset\rbrace$. If $t_{i_{n'}}=1$, the procedure ends by performing a rotation in $[t_{i_{n'-1}},1]$. If we find $t_{i_{n'}}<1$ with the property that
$$
\left(\sigma_{i_{n'}+1}\dots\sigma_{i}+ \frac{1}{\sigma_{i_{n'}+1}\dots\sigma_{i}}\right)\leq 3
$$
for all $i=i_{n'}+1,\dots,n$, the construction ends with a rotation in $[t_{i_{n'-1}},1]$ (as in the case of the estimate (\ref{est:case:3})). \\


\noindent
In particular, for each subsquare $\kappa_0$ we find a sequence of times $\lbrace t_{i_j}(\kappa_0)\rbrace$, $j=1,\dots,n'(\kappa_0)$, where we are performing a rotation. 
   %
    There are two cases to be considered: if $T_0(\kappa_0)$ is empty then 
    \begin{equation*}
       \TV(\hat{b}^D_t-\tilde b^D_t)(\kappa_0) \leq \frac{\mathcal O(1)}{D^2},
    \end{equation*}
    otherwise
    \begin{equation*}
       \TV(\hat{b}^D_t- \tilde b^D_t)(\kappa_0) \leq \mathcal O(1) \| \TV(\tilde b^D)({\kappa}_0)\|_\infty + \frac{\mathcal O(n\wp)}{D^3}.
    \end{equation*}
    
    \noindent
    Summing over all possible $\kappa_0$ we find that
    \begin{equation*}
       \begin{split}
       \TV(\hat{b}^D_s)(K)
       &\leq \TV(\tilde b^D_s)(K) + D^2 \bigg( \frac{\mathcal O(n\wp)}{D^3} + \frac{\mathcal O(1)}{D^2} \bigg) + C_2 \| \TV(\tilde b^D)(K)\|_\infty \\
       &\leq \TV(\tilde b^D_s)(K) + C_1 + C_2 \| \TV(\tilde b^D)(K)\|_\infty
       \end{split}
    \end{equation*}
    if $D \gg 1$,
    therefore we can conclude the proof finding a positive constant $C>0$ such that
    \begin{equation*}
       \|\TV(\hat{b}^D)(K)\|_\infty\leq C_1 + C_2 \| \TV(\tilde b^D)(K)\|_\infty
    \end{equation*}
    which is the desired estimate.
%
%
\end{proof}

\begin{remark}
The same result can be obtained for the $d$-dimensional case, by using the maps of Section \ref{Sss:d-dim} and Remarks \ref{Rem:d-dim_BV} and \ref{Rem:multi_d_rot}.
\end{remark}

%% file: Brenier/cubetti_3d.pdf_tex
\begingroup%
  \makeatletter%
  \providecommand\color[2][]{%
    \errmessage{(Inkscape) Color is used for the text in Inkscape, but the package 'color.sty' is not loaded}%
    \renewcommand\color[2][]{}%
  }%
  \providecommand\transparent[1]{%
    \errmessage{(Inkscape) Transparency is used (non-zero) for the text in Inkscape, but the package 'transparent.sty' is not loaded}%
    \renewcommand\transparent[1]{}%
  }%
  \providecommand\rotatebox[2]{#2}%
  \newcommand*\fsize{\dimexpr\f@size pt\relax}%
  \newcommand*\lineheight[1]{\fontsize{\fsize}{#1\fsize}\selectfont}%
  \ifx\svgwidth\undefined%
    \setlength{\unitlength}{1939.23103393bp}%
    \ifx\svgscale\undefined%
      \relax%
    \else%
      \setlength{\unitlength}{\unitlength * \real{\svgscale}}%
    \fi%
  \else%
    \setlength{\unitlength}{\svgwidth}%
  \fi%
  \global\let\svgwidth\undefined%
  \global\let\svgscale\undefined%
  \makeatother%
  \begin{picture}(1,0.51264219)%
    \lineheight{1}%
    \setlength\tabcolsep{0pt}%
    \put(0,0){\includegraphics[width=\unitlength,page=1]{cubetti_3d.pdf}}%
    \put(0.57363349,0.52283433){\color[rgb]{0,0,0}\makebox(0,0)[lt]{\lineheight{1.25}\smash{\begin{tabular}[t]{l}$T$\end{tabular}}}}%
    \put(0.68552614,0.40959561){\color[rgb]{0,0,0}\makebox(0,0)[lt]{\lineheight{1.25}\smash{\begin{tabular}[t]{l}$z_2 = \mathrm{const}$\end{tabular}}}}%
    \put(0.19521963,0.42137145){\color[rgb]{0,0,0}\makebox(0,0)[lt]{\lineheight{1.25}\smash{\begin{tabular}[t]{l}$T^{-1}(z_1,z_2)$\end{tabular}}}}%
    \put(0.43354818,0.17464073){\color[rgb]{0,0,0}\makebox(0,0)[lt]{\lineheight{1.25}\smash{\begin{tabular}[t]{l}$x_3 \in [k_3,k_3+1]/N$\end{tabular}}}}%
    \put(0.79609185,0.44985844){\makebox(0,0)[lt]{\lineheight{1.25}\smash{\begin{tabular}[t]{l}$z_1 = \mathrm{const}$\end{tabular}}}}%
    \put(0,0){\includegraphics[width=\unitlength,page=2]{cubetti_3d.pdf}}%
  \end{picture}%
\endgroup%

%% file: Brenier/cubetti_3d_rectif_d.pdf_tex
\begingroup%
  \makeatletter%
  \providecommand\color[2][]{%
    \errmessage{(Inkscape) Color is used for the text in Inkscape, but the package 'color.sty' is not loaded}%
    \renewcommand\color[2][]{}%
  }%
  \providecommand\transparent[1]{%
    \errmessage{(Inkscape) Transparency is used (non-zero) for the text in Inkscape, but the package 'transparent.sty' is not loaded}%
    \renewcommand\transparent[1]{}%
  }%
  \providecommand\rotatebox[2]{#2}%
  \newcommand*\fsize{\dimexpr\f@size pt\relax}%
  \newcommand*\lineheight[1]{\fontsize{\fsize}{#1\fsize}\selectfont}%
  \ifx\svgwidth\undefined%
    \setlength{\unitlength}{902.53827433bp}%
    \ifx\svgscale\undefined%
      \relax%
    \else%
      \setlength{\unitlength}{\unitlength * \real{\svgscale}}%
    \fi%
  \else%
    \setlength{\unitlength}{\svgwidth}%
  \fi%
  \global\let\svgwidth\undefined%
  \global\let\svgscale\undefined%
  \makeatother%
  \begin{picture}(1,0.35194778)%
    \lineheight{1}%
    \setlength\tabcolsep{0pt}%
    \put(0,0){\includegraphics[width=\unitlength,page=1]{cubetti_3d_rectif_d.pdf}}%
    \put(0.76445494,0.04068903){\color[rgb]{0,0,0}\makebox(0,0)[lt]{\lineheight{1.25}\smash{\begin{tabular}[t]{l}$x_3 \in [k_3,k_3+1]/N$\end{tabular}}}}%
    \put(0,0){\includegraphics[width=\unitlength,page=2]{cubetti_3d_rectif_d.pdf}}%
    \put(0.39136624,0.06843507){\color[rgb]{0,0,0}\makebox(0,0)[lt]{\lineheight{1.25}\smash{\begin{tabular}[t]{l}$x_1 = \text{const}$\end{tabular}}}}%
    \put(0.0447773,0.22533322){\color[rgb]{0,0,0}\makebox(0,0)[lt]{\lineheight{1.25}\smash{\begin{tabular}[t]{l}$T^{-1}(z_2)$\end{tabular}}}}%
    \put(0.35224351,0.20040353){\color[rgb]{0,0,0}\makebox(0,0)[lt]{\lineheight{1.25}\smash{\begin{tabular}[t]{l}$T^{-1}(z_1,z_2)$\end{tabular}}}}%
    \put(0,0){\includegraphics[width=\unitlength,page=3]{cubetti_3d_rectif_d.pdf}}%
  \end{picture}%
\endgroup%

%% file: Brenier/cubetti_3d_recurrence.pdf_tex
\begingroup%
  \makeatletter%
  \providecommand\color[2][]{%
    \errmessage{(Inkscape) Color is used for the text in Inkscape, but the package 'color.sty' is not loaded}%
    \renewcommand\color[2][]{}%
  }%
  \providecommand\transparent[1]{%
    \errmessage{(Inkscape) Transparency is used (non-zero) for the text in Inkscape, but the package 'transparent.sty' is not loaded}%
    \renewcommand\transparent[1]{}%
  }%
  \providecommand\rotatebox[2]{#2}%
  \newcommand*\fsize{\dimexpr\f@size pt\relax}%
  \newcommand*\lineheight[1]{\fontsize{\fsize}{#1\fsize}\selectfont}%
  \ifx\svgwidth\undefined%
    \setlength{\unitlength}{495.87267778bp}%
    \ifx\svgscale\undefined%
      \relax%
    \else%
      \setlength{\unitlength}{\unitlength * \real{\svgscale}}%
    \fi%
  \else%
    \setlength{\unitlength}{\svgwidth}%
  \fi%
  \global\let\svgwidth\undefined%
  \global\let\svgscale\undefined%
  \makeatother%
  \begin{picture}(1,0.49699797)%
    \lineheight{1}%
    \setlength\tabcolsep{0pt}%
    \put(0,0){\includegraphics[width=\unitlength,page=1]{cubetti_3d_recurrence.pdf}}%
    \put(0.39634862,0.48014398){\color[rgb]{0,0,0}\makebox(0,0)[lt]{\lineheight{1.25}\smash{\begin{tabular}[t]{l}$(T')^{-1} \circ \sigma$\end{tabular}}}}%
    \put(0,0){\includegraphics[width=\unitlength,page=2]{cubetti_3d_recurrence.pdf}}%
    \put(0.7418739,0.42528134){\color[rgb]{0,0,0}\makebox(0,0)[lt]{\lineheight{1.25}\smash{\begin{tabular}[t]{l}$(T')^{-1}(z_1)$\end{tabular}}}}%
    \put(0.47610474,0.15785302){\color[rgb]{0,0,0}\makebox(0,0)[lt]{\lineheight{1.25}\smash{\begin{tabular}[t]{l}$(T')^{-1}(z_2)$\end{tabular}}}}%
  \end{picture}%
\endgroup%

%% file: Brenier/cubetti_3d_laststep.pdf_tex
\begingroup%
  \makeatletter%
  \providecommand\color[2][]{%
    \errmessage{(Inkscape) Color is used for the text in Inkscape, but the package 'color.sty' is not loaded}%
    \renewcommand\color[2][]{}%
  }%
  \providecommand\transparent[1]{%
    \errmessage{(Inkscape) Transparency is used (non-zero) for the text in Inkscape, but the package 'transparent.sty' is not loaded}%
    \renewcommand\transparent[1]{}%
  }%
  \providecommand\rotatebox[2]{#2}%
  \newcommand*\fsize{\dimexpr\f@size pt\relax}%
  \newcommand*\lineheight[1]{\fontsize{\fsize}{#1\fsize}\selectfont}%
  \ifx\svgwidth\undefined%
    \setlength{\unitlength}{248.24973039bp}%
    \ifx\svgscale\undefined%
      \relax%
    \else%
      \setlength{\unitlength}{\unitlength * \real{\svgscale}}%
    \fi%
  \else%
    \setlength{\unitlength}{\svgwidth}%
  \fi%
  \global\let\svgwidth\undefined%
  \global\let\svgscale\undefined%
  \makeatother%
  \begin{picture}(1,0.57445928)%
    \lineheight{1}%
    \setlength\tabcolsep{0pt}%
    \put(0,0){\includegraphics[width=\unitlength,page=1]{cubetti_3d_laststep.pdf}}%
    \put(0.28988295,0.03362224){\color[rgb]{0,0,0}\makebox(0,0)[lt]{\lineheight{1.25}\smash{\begin{tabular}[t]{l}$z_2 = \text{const}$\end{tabular}}}}%
    \put(0.26921288,0.49717584){\color[rgb]{0,0,0}\makebox(0,0)[lt]{\lineheight{1.25}\smash{\begin{tabular}[t]{l}$z_1 = \text{const}$\end{tabular}}}}%
    \put(0.49716234,0.31579471){\color[rgb]{0,0,0}\makebox(0,0)[lt]{\lineheight{1.25}\smash{\begin{tabular}[t]{l}$G$\end{tabular}}}}%
  \end{picture}%
\endgroup%

%% file: appendix.tex
\section{Appendix}
\label{S:appendix}

\begin{proof}[Proof of Lemma \ref{map:to:flow}] 

{
By the Ergodic Theorem, $T = X_{t=1}$ is ergodic iff 
\begin{equation*}
    \frac{1}{n} \sum_{i=0}^{n-1} \chi_{T^i(A)} \to_{L^1} |A|
\end{equation*}
In particular, if $T$ is ergodic, then by writing
\begin{equation*}
    \frac{1}{n} \int_0^n \chi_{X_t(A)} dt = \int_0^1 \frac{n-1}{n} \bigg( \frac{1}{n-1} \sum_{i=0}^{n-1} \chi_{T^i(X_s(A))} \bigg) ds
\end{equation*}
we see that
\begin{equation*}
    \dashint_0^t \chi_{X_s(A)} ds \to_{L^1} |A|.
\end{equation*}
It is immediate to find a counterexample to the converse implication: just consider rotation of the unit circle with period $1$.

The proof of the implication $\Rightarrow$ in the second point is analogous. For the converse, let $A,B \in \Sigma$ such that
\begin{equation*}
    \frac{1}{n} \sum_{i=0}^n \big[ |T^i(A) \cap B| - |A||B| \big]^2 > \epsilon.
\end{equation*}
By the continuity of $s \mapsto X_s$ in the neighborhood topology we have that there exists $\bar s$ such that for $0 \leq s \leq \bar s$ it holds
\begin{equation*}
    \big| X_s(B) \triangle B \big| = \big| B \triangle (X_s)^{-1}(B) \big| < \frac{\epsilon}{2}.
\end{equation*}
Hence we can write
\begin{equation*}
    \begin{split}
        \dashint_0^{n} \big[ |X_t(A) \cap B| - |A| |B| \big]^2 dt &\geq \frac{n}{T} \int_0^{\bar s} \frac{1}{n} \sum_{i=0}^{n-1} \big[ \big| X_s(T^i(A)) \cap B \big| - |A \cap B| \big]^2 ds \\
        &= \int_0^{\bar s} \frac{1}{n} \sum_{i=0}^{n-1} \big[ \big| T^i(A) \cap (X_s)^{-1}(B) \big| - |A \cap B| \big]^2 ds \\
        &\geq \int_0^{\bar s} \frac{1}{n} \sum_{i=0}^{n-1} \big[ \big| T^i(A) \cap B \big| - |A \cap B| \big]^2 ds - \bar s \frac{\epsilon}{2} > \bar s \frac{\epsilon}{2}
    \end{split} 
\end{equation*}
for $n \gg 1$. Hence
\begin{equation*}
    \liminf \dashint_0^{T} \big[ |X_t(A) \cap B| - |A| |B| \big]^2 dt \not= 0.
\end{equation*}

Finally, if $T$ is strongly mixing, the continuity of $s \mapsto X_s$ in the neighborhood topology gives that $s \mapsto X_s^n = X_s \circ T^n$ is a family of equicontinuous functions, and since for all $s$ fixed
\begin{equation*}
\lim_{n \to \infty} |X_s(T^n(A)) \cap B| = |A| |B|    
\end{equation*}
we conclude that $X_s^n$ converges to $0$ uniformly in $s$. The opposite implication is trivial.
}
\end{proof}